\documentclass[a4paper]{amsart}

\usepackage[utf8]{inputenc}
\usepackage[T1]{fontenc}
\usepackage{amsthm}
\usepackage{xspace}
\usepackage{amssymb,amsmath,stmaryrd}
\usepackage[english]{babel}
\usepackage{enumerate}
\usepackage[all]{xy}
\usepackage{hyperref}

\theoremstyle{plain}
\newtheorem{theorem}{Theorem}[section]
\newtheorem{remark}[theorem]{Remark}
\newtheorem{lemma}[theorem]{Lemma}
\newtheorem{corollary}[theorem]{Corollary}
\newtheorem{proposition}[theorem]{Proposition}

\theoremstyle{definition}
\newtheorem{assumption}[theorem]{Assumption}
\newtheorem{definition}[theorem]{Definition}
\newtheorem{notation}[theorem]{Notation}
\newtheorem{example}[theorem]{Example}
\numberwithin{equation}{section}

\newcommand{\ca}{\mbox{$C\sp*$-}al\-ge\-bra\xspace}
\newcommand{\cas}{\mbox{$C\sp*$-}al\-ge\-bras\xspace}
\newcommand{\Z}{\ensuremath{\mathbb{Z}}\xspace}
\newcommand{\N}{\ensuremath{\mathbb{N}}\xspace}

\newcommand{\Q}{\ensuremath{\mathbb{Q}}\xspace}
\newcommand{\K}{\ensuremath{\mathbb{K}}\xspace}
\newcommand{\ie}{\emph{i.e.}\xspace}
\newcommand{\cf}{\emph{cf.}\xspace}

\newcommand{\id}{\ensuremath{\operatorname{id}}\xspace}

\newcommand{\cok}{\operatorname{cok}}

\newcommand{\FKRplus}{\ensuremath{\operatorname{FK}^+_\mathcal{R}}\xspace}
\newcommand{\FKR}{\ensuremath{\operatorname{FK}_\mathcal{R}}\xspace}
\newcommand{\A}{\ensuremath{\mathfrak{A}}\xspace}
\newcommand{\B}{\ensuremath{\mathfrak{B}}\xspace}
\newcommand{\Asf}{\mathsf{A}}
\newcommand{\Bsf}{\mathsf{B}}
\newcommand{\Esf}{\mathsf{E}}
\newcommand{\Prim}{\operatorname{Prim}}
\newcommand{\calP}{\ensuremath{\mathcal{P}}\xspace}
\newcommand{\GLZ}[1][n]{\ensuremath{\operatorname{GL}(#1,\Z)}\xspace}
\newcommand{\GL}{\ensuremath{\operatorname{GL}}\xspace}
\newcommand{\SLZ}[1][n]{\ensuremath{\operatorname{SL}(#1,\Z)}\xspace}
\newcommand{\SL}{\ensuremath{\operatorname{SL}}\xspace}
\newcommand{\GLPZ}[1][\mathbf{n}]{\ensuremath{\operatorname{GL}_\calP(#1,\Z)}\xspace}
\newcommand{\GLP}{\ensuremath{\operatorname{GL}_\calP}\xspace}
\newcommand{\SLPZ}[1][\mathbf{n}]{\ensuremath{\operatorname{SL}_\calP(#1,\Z)}\xspace}
\newcommand{\SLP}{\ensuremath{\operatorname{SL}_\calP}\xspace}

\newcommand{\MZ}[1][\mathbf{m}\times\mathbf{n}]{\ensuremath{\mathfrak{M}(#1,\Z)}\xspace}
\newcommand{\MPZ}[1][\mathbf{m}\times\mathbf{n}]{\ensuremath{\mathfrak{M}_\calP(#1,\Z)}\xspace}
\newcommand{\Mplus}[1][m\times n]{\ensuremath{\mathfrak{M}^+(#1,\Z)}\xspace}
\newcommand{\MPplusZ}[1][\mathbf{m}\times\mathbf{n}]{\ensuremath{\mathfrak{M}^+_\calP(#1,\Z)}\xspace}
\newcommand{\ftn}[3]{ #1 \colon #2 \rightarrow #3 }
\newcommand{\setof}[2]{\left\{ #1 \mid #2 \right\}}
\newcommand{\GLPE}{\GLP-equivalent\xspace}
\newcommand{\SLPE}{\SLP-equivalent\xspace}
\newcommand{\GLPEe}{\GLP-equivalence\xspace}
\newcommand{\SLPEe}{\SLP-equivalence\xspace}
\newcommand{\Meq}{\ensuremath{\sim_{M}}\xspace}
\newcommand{\MCeq}{\ensuremath{\sim_{M'}}\xspace}
\newcommand{\OO}{\mbox{\texttt{\textup{(O)}}}\xspace}
\newcommand{\II}{\mbox{\texttt{\textup{(I)}}}\xspace}
\newcommand{\RR}{\mbox{\texttt{\textup{(R)}}}\xspace}
\newcommand{\SSS}{\mbox{\texttt{\textup{(S)}}}\xspace}
\newcommand{\CC}{\mbox{\texttt{\textup{(C)}}}\xspace}
\newcommand{\TT}{\mbox{\texttt{\textup{(T)}}}\xspace}

\newenvironment{smallpmatrix}{\left(\begin{smallmatrix}}{\end{smallmatrix}\right)}

\title[Geometric classification of unital graph $C^*$-algebras]{Geometric classification of \\ unital graph $C^*$-algebras of real rank zero}
\date{\today}
\author{S\o{}ren Eilers}
\address{Department of Mathe\-matical Sciences, University of Copen\-hagen, Universi\-tets\-park\-en~5, DK-2100 Copen\-hagen, Den\-mark}
\email{eilers@math.ku.dk}
\author{Gunnar Restorff}
\address{Department of Science and Technology, University of the Faroe Islands, N\'{o}at\'{u}n~3, FO-100 T\'{o}rshavn, the Faroe Islands}
\email{gunnarr@setur.fo}
\author{Efren Ruiz}
\address{Department of Mathematics, University of Hawaii, Hilo, 200 W.~Kawili St., Hilo, Hawaii, 96720-4091 USA}
\email{ruize@hawaii.edu}
\author{Adam P.~W.~S\o{}rensen}
\address{Department of Mathematics, University of Oslo, PO BOX 1053 Blindern, N-0316 Oslo, Norway}
\email{apws@math.uio.no}

\begin{document}


\begin{abstract}
We generalize the classification result of Restorff (\cite{MR2270572}) on Cuntz-Krieger algebras to cover all unital graph \cas with real rank zero, showing that Morita equivalence in this case is determined 
by ordered, filtered $K$-theory as conjectured by three of the authors. The classification result is \emph{geometric} in the sense that it establishes that any Morita equivalence between $C^*(E)$ and $C^*(F)$ in this class can be realized by a sequence of moves leading from $E$ to $F$ in a way resembling the role of Reidemeister moves on knots. As a key technical step, we prove that the so-called \emph{Cuntz splice} leaves unital graph \cas invariant up to Morita equivalence.

The results of this preprint will be generalized in a forthcoming paper.
\end{abstract}

\maketitle


\tableofcontents


\section{Introduction}

Ever since the inception of graph $C^*$-algebras, it has been a key ambition to classify these objects by their $K$-theory, either up to isomorphism or stable isomorphism.  
With the simple case resolved by appeal to the celebrated classification results of Elliott on one hand and Kirchberg and Phillips on the other, the focus has been on the nonsimple \cas, and in fact this endeavour has evolved in parallel with the gradual realization of what invariants may prove to be complete in the case when the number of ideals is finite and the \cas in question are not stably finite. In this sense, the fundamental results obtained on the classification of certain classes of graph \cas are playing a role parallel to the one played by R\o rdam's classification of simple Cuntz-Krieger algebras as a catalyst for the Kirchberg-Phillips classification mentioned above.

The first two results on the classification problem for nonsimple graph \cas were obtained by R\o rdam in \cite{MR1446202} and by Restorff in \cite{MR2270572} by very different methods. R\o rdam showed the importance of involving the full data contained in six-term exact sequences of the \cas given and proved a very complete classification theorem while restricting the ideal lattice to be as small as possible: only one nontrivial ideal. In Restorff's work, the ideal lattice was arbitrary among the finite ideal lattices, but as his method was to reduce the problem to classification of shifts of finite type and appeal to deep results by Boyle and Huang  from symbolic dynamics (\cite{MR1907894}, \cite{MR1990568}), only graph \cas in the Cuntz-Krieger class were covered.

Subsequent progress has mainly followed the approach in \cite{MR1446202} (see \cite{arXiv:1101.5702v3}, \cite{MR2563693}, \cite{MR3056712}), and hence applies only to restricted kinds of ideal lattices but with few further restrictions on the nature of the underlying graphs. The case of purely infinite graph \cas with finitely many ideals has been resolved (very interestingly, by a different invariant than what was proposed in \cite{MR2666426}) in recent work by Bentmann and Meyer (\cite{arXiv:1405.6512v1}), but as summarized in \cite{MR3142033} there is not at present sufficient technology to take this approach much farther in the mixed cases than to \cas with three or four primitive ideals. 

In the paper at hand we complete the stable  classification of unital graph \cas with real rank zero, following the  strategy from \cite{MR2270572} as generalized by the authors in various constellations over a period of 5 years (\cite{MR3082546}, \cite{MR2666426}, \cite{MR3047630}, \cite{Eilers-Ruiz-Sorensen}).
Our method of proof, a substantial elaboration of key ideas from the authors' earlier work along with key ideas from the papers of Boyle and Huang, leads to a  \emph{geometric} classification, allowing us to conclude from Morita equivalence between a pair of graph \cas $C^*(E)$ and $C^*(F)$ that a sequence  of  basic moves on the graphs may lead from $E$ to $F$ in a way resembling the role of Reidemeister moves on knots.
 
These moves are closely related to those defining flow equivalence for shift spaces, apart from the so-called \emph{Cuntz splice} which has no counterpart in dynamics and also fails to preserve the canonical diagonal Abelian subalgebra of the graph \cas (\cf~\cite{MR3276420}, \cite{arXiv:1410.2308v1}).
In all cases when classification has been established, invariance of the Cuntz splice follows immediately from the fact that it will not change the $K$-theory, and in particular it was observed in \cite{arXiv:1405.6512v1} that Cuntz splice is invariant  in the class of graph \cas which are purely infinite with finitely many ideals. But since our goal is to use the Cuntz splice to establish classification results in classes outside the scope of these results, we must prove here that in the case under investigation, the Cuntz splice leaves the \cas invariant. In fact, this result covers the full case of unital graph \cas without any reference to real rank zero.
 
Although the real rank zero condition is often seen to bear importance in classification theory (\cite{MR1241132}, \cite{MR1402768}, \cite{MR1465599}, \cite{MR2949216}) its role in our proof is of a substantially different nature than in the papers listed. Indeed, since we gave in \cite{Eilers-Ruiz-Sorensen} an example of two finite graphs yielding Morita equivalent graph \cas of real rank one for which no sequence of moves suffices to lead from one to another, we require real rank zero, through its graph algebraic characterization Condition (K), to ensure that the classification result is indeed geometric in the sense of passing through moves.

After posting the first version of this paper, we realized that it was possible to obtain 
classification by $K$-theory in the general unital case by exhibiting a new move which allows us to connect the two examples mentioned above, and we have recently completed the proof that this move leaves the $C^*$-algebra invariant up to Morita equivalence. Consequently, we will present a complete classification result in a forthcoming paper which will contain the results in the present paper as a special case.


\section{Preliminaries for statement of main theorem}

\subsection{Graphs and their matrices}

By a \emph{graph} we mean a directed graph. Formally:

\begin{definition}
A graph $E$ is a four tuple $E = (E^0 , E^1 , r, s)$ where $E^0$ and $E^1$ are sets, and $r$ and $s$ are maps from $E^1$ to $E^0$. 
The elements of $E^0$ are called \emph{vertices}, the elements of $E^1$ are called \emph{edges}, the map $r$ is called the \emph{range map}, and the map $s$ is called the \emph{source map}. 
\end{definition}

All graphs considered will be \emph{countable}, \ie, there are countably many vertices and edges. 
We call a graph \emph{finite}, if there are only finitely many vertices and edges.
As usual, two graphs $E_1 = (E_1^0 , E_1^1 , r_1, s_1)$ and $E_2 = (E_2^0 , E_2^1 , r_2, s_2)$ are called isomorphic if there exist bijections $\phi^i$ from $E_1^i$ to $E_2^i$ such that $s_2\circ\phi^1=\phi^0\circ s_1$ and $r_2\circ\phi^1=\phi^0\circ r_1$. 
We will freely identify graphs up to isomorphism.

\begin{definition}
A \emph{loop} is an edge with the same range and source. 

A \emph{path} $\mu$ in a graph is a finite sequence $\mu = e_1 e_2 \cdots e_n$ of edges satisfying 
$r(e_i)=s(e_{i+1})$, for all $i=1,2,\ldots, n-1$, and we say that the \emph{length} of $\mu$ is $n$. 
We extend the range and source maps to paths by letting $s(\mu) = s(e_1)$ and $r(\mu) = r(e_n)$. 
Vertices in $E$ are regarded as \emph{paths of length $0$} (also called empty paths). 

A \emph{cycle} is a nonempty path $\mu$ such that $s(\mu) = r(\mu)$.
A \emph{return path} is a cycle $\mu = e_1 e_2 \cdots e_n$ such that $r(e_i) \neq r(\mu)$ for $i < n$.

For a loop, cycle or return path, we say that it is based at the source vertex of its path. 
We also say that a vertex supports a certain loop, cycle or return path if it is based at that vertex. 
\end{definition}

\begin{definition}
A vertex $v\in E^0$ in $E$ is called \emph{regular} if $s^{-1}(v)$ is finite and nonempty. 

A vertex $v\in E^0$ in $E$ is called \emph{source} if $r^{-1}(v)=\emptyset$. 
A vertex $v\in E^0$ in $E$ is called a \emph{sink} if $s^{-1}(v)=\emptyset$.
Note that an isolated vertex is both a sink and a source. 
\end{definition}

\begin{notation}
If there exists a path from vertex $u$ to vertex $v$, then we write $u \geq v$ --- this is a preorder on the vertex set, \ie, it is reflexive and transitive, but need not be antisymmetric. 
\end{notation}

It is key to our approach to graph \cas to be able to shift between a graph and its adjacency matrix. 
In what follows, we let \N denote the set of positive integers, while $\N_0$ denotes the set of nonnegative integers.

\begin{definition}
Let $E = (E^0 , E^1 , r, s)$ be a graph.
We define its \emph{adjacency matrix} $\Asf_E$ as a $E^0\times E^0$ matrix with the $(u,v)$'th entry being
$$\left\vert\setof{e\in E^1}{s(e)=u, r(e)=v}\right\vert.$$
As we only consider countable graphs, $\Asf_E$ will be a finite matrix or a countably infinite matrix, and it will have entries from $\N_0\sqcup\{\infty\}$.

Let $X$ be a set.
If $A$ is an $X \times X$ matrix with entries from $\N_0\sqcup\{\infty\}$ we let $\Esf_{A}$ be the graph with vertex set $X$ and between two vertices $x,x' \in X$ we have $\Asf(x,x')$ edges.
\end{definition}

It will be convenient for us to alter the adjacency matrix of a graph in two very specific ways, removing singular rows and subtracting the identity, so we introduce notation for this. 

\begin{notation}
Let $E$ be a graph and $\Asf_E$ its adjacency matrix. 
Denote by $\Asf_{E}^\bullet$ the matrix obtained from $\Asf_{E}$ by removing all rows corresponding to singular vertices of $E$.

Let $\Bsf_E$ denote the matrix $\Asf_{E} - I$, and let $\Bsf_{E}^\bullet$ be $\Bsf_E$ with the rows corresponding to singular vertices of $E$ removed. 
\end{notation}

\subsection{Graph \texorpdfstring{$C^*$}{C*}-algebras}
We follow the notation and definition for graph \cas in \cite{MR1670363}; this is not the convention used in Raeburn's monograph \cite{MR2135030}. 

\begin{definition} \label{def:graphca}
Let $E = (E^0,E^1,r,s)$ be a graph.
The \emph{graph \ca} $C^*(E)$ is defined as the universal \ca generated by
a set of mutually orthogonal projections $\setof{ p_v }{ v \in E^0 }$ and a set $\setof{ s_e }{ e \in E^1 }$ of partial isometries satisfying the relations
\begin{itemize}
	\item $s_e^* s_f = 0$ if $e,f \in E^1$ and $e \neq f$,
	\item $s_e^* s_e = p_{r(e)}$ for all $e \in E^1$,
	\item $s_e s_e^* \leq p_{s(e)}$ for all $e \in E^1$, and,
	\item $p_v = \sum_{e \in s^{-1}(v)} s_e s_e^*$ for all $v \in E^0$ with $0 < |s^{-1}(v)| < \infty$.
\end{itemize}
\end{definition}

It is clear from the definition that an isomorphism between graphs induces a canonical isomorphism between the corresponding graph \cas.

\begin{definition}
Let $E$ be a graph.  
We say that $E$ satisfies \emph{Condition (K)} if for all vertices $v\in E^0$ in $E$, either there is no return path based at $v$ or there are at least two distinct return paths based at $v$. 
\end{definition}

\begin{remark}
The graph \ca $C^*(E)$ is isomorphic to a Cuntz-Krieger algebra if and only if the graph~$E$ is finite with no sinks, see~\cite[Theorem~3.13]{arXiv:1209.4336v3}.
If all vertices in $E$ support two loops, then $C^*(E)$ is purely infinite, see~\cite[Theorem~2.3]{MR1989499}.
In our main result, Theorem~\ref{thm:main}, the graphs are assumed to have finitely many vertices and to satisfy Condition (K) --- for all such graphs the associated graph \cas are separable, unital, of real rank zero \cite[Theorem~2.5]{MR1989499} and have finitely many ideals. 
\end{remark}

\subsection{Filtered \texorpdfstring{$K$}{K}-theory}

\begin{definition}
Let \A be a \ca with finitely many ideals, and let $\Prim\A$ denote the primitive ideal space of \A equipped with the hull-kernel topology. 
A subset of $\Prim\A$ is called \emph{locally closed}, if it is the set difference between two open subsets of $\Prim\A$. 
There is a canonical lattice isomorphism between the open subsets of $\Prim\A$ and the (closed, two sided) ideals of \A{} --- let us denote this correspondence with $O\mapsto\A(O)$. 
If $V\subseteq\Prim\A$ is a difference set, then $V=U\setminus O$ for some open subsets $O\subseteq U\subseteq\Prim\A$. 
If also $V=U'\setminus O'$ for some other open subsets $O'\subseteq U'\subseteq\Prim\A$, then there exists a canonical isomorphism between $\A(U)/\A(O)$ and $\A(U')/\A(O')$. Thus we can let 
$$\A(V)=\A\left(\bigcap_{V=U\setminus O,O\subseteq U \subseteq\Prim\A}U\right) / \A\left(\bigcap_{V=U\setminus O,O\subseteq U \subseteq\Prim\A}O\right)$$
with a slight abuse of notation, since we identify $\A(O)$ with $\A(O)/\{0\}$ whenever $O$ is open. 
Note, that all singletons of $\Prim\A$ are locally closed. 

For each $x\in\Prim\A$ we let $S_x$ denote the smallest open subset that contains $x$, and we let 
$R_x=S_x\setminus\{x\}$, which is an open subset. 
Whenever we have two open subsets $O\subseteq U\subseteq \Prim\A$, we get a cyclic six term exact sequence in $K$-theory:
\begin{equation}\label{eq:sixtermktheory}\vcenter{
\xymatrix{
K_0(\A(O))\ar[r] & K_0(\A(U))\ar[r] & K_0(\A(U\setminus O))\ar[d] \\
K_1(\A(U\setminus O))\ar[u] & K_1(\A(U))\ar[l] & K_1(\A(O)).\ar[l] \\
}}
\end{equation}
In fact, this holds even if $O$ and $U$ are locally closed. 
If \A is a real rank zero algebra, then the map from $K_0$ to $K_1$ will be the zero map. 

Let 
\begin{align*}
I_0(\A)&=\setof{R_x}{x\in\Prim\A,R_x\neq\emptyset}\cup\setof{S_x}{x\in\Prim\A}\cup\setof{\{x\}}{x\in\Prim\A},\\
I_1(\A)&=\setof{\{x\}}{x\in\Prim\A},
\end{align*}
and let $\operatorname{Imm}(x)$ denote the set 
$$\setof{y\in\Prim\A}{S_y\subsetneq S_x\wedge \not\exists z\in\Prim\A\colon S_y\subsetneq S_z\subsetneq S_x}.$$
The \emph{reduced filtered $K$-theory} of \A, $\FKR(\A)$, consists of the families of groups 
$(K_0(\A(V)))_{V\in I_0(\A)}$ and 
$(K_1(\A(O)))_{O\in I_1(\A)}$ together with the maps in the sequences
$$K_1(\A(\{x\}))\to K_0(\A(R_x))\to  K_0(\A(S_x))\to  K_0(\A(\{x\}))$$
originating from the sequence \eqref{eq:sixtermktheory}, for all $x\in\Prim\A$ with $R_x\neq\emptyset$, and 
the maps in the sequences
$$K_0(\A(S_y))\to  K_0(\A(R_x))$$
originating from the sequence \eqref{eq:sixtermktheory}, for all pairs $(x,y)\in\Prim\A$ with $y\in\operatorname{Imm}(x)$ and $\operatorname{Imm}(x)\setminus\{y\}\neq\emptyset$.

Let also $\mathfrak{B}$ be a \ca with finitely many ideals. 
An isomorphism from $\FKR(\A)$ to $\FKR(\B)$ consists of a homeomorphism $\rho\colon\Prim\A\rightarrow\Prim\B$ and families of isomorphisms
$$(\phi_V\colon K_0(\A(V))\rightarrow K_0(\B(\rho(V))))_{V\in I_0(\A)}$$
$$(\psi_O\colon K_1(\A(O))\rightarrow K_1(\B(\rho(O))))_{O\in I_1(\A)}$$ 
such that all the ladders coming from the above sequences commute. 

Analogously, we define the \emph{ordered reduced filtered $K$-theory} of \A, $\FKRplus(\A)$, just as $\FKR(\A)$ where we also consider the order on all the $K_0$-groups --- and for an isomorphism, we demand that the isomorphisms between the $K_0$-groups are order isomorphisms. 
\end{definition}

\subsection{Moves on graphs}
In this section we describe the moves on graphs used in \cite{MR3082546}. We mention that these moves have been considered by other authors, and were previously noted to preserve the Morita equivalence class of the associated graph \ca (see \cite{MR2054048}). 

\begin{definition}[Move \SSS: Remove a regular source] 
Let $E = (E^0 , E^1 , r, s)$ be a graph, and let $w\in E^0$ be a source that is also a regular vertex. 
Let $E_S$ denote the graph $(E_S^0 , E_S^1 , r_S , s_S )$ defined by
$$E_S^0 := E^0 \setminus \{w\}\quad
E_S^1 := E^1 \setminus s^{-1} (w)\quad
r_S := r|_{E_S^1}\quad
s_S := s|_{E_S^1}.$$
We call $E_S$ the \emph{graph obtained by removing the source $w$ from $E$}, and say $E_S$ is formed by performing move \SSS to $E$.
\end{definition}

\begin{definition}[Move \RR: Reduction at a regular vertex] 
Suppose that $E = (E^0 , E^1 , r, s)$ is a graph, and let $w \in E^0$ be a regular vertex with the property that $s(r^{-1} (w)) = \{x\}$, $s^{-1} (w) = \{f \}$, and $r(f ) \neq w$. 
Let $E_R$ denote the graph $(E_R^0, E_R^1, r_R , s_R )$ defined by
\begin{align*}
E_R^0&:= E^0 \setminus \{w\} \\
E_R^1&:= \left(E^1 \setminus (\{f \} \cup r^{-1}(w))\right) \cup \setof{e_f}{e \in E^1 \text{ and } r(e) = w} \\
r_R (e) &:= r(e)\text{ if }e \in E^1 \setminus (\{f \} \cup r^{-1}(w)) \quad\text{and}\quad r_R (e_f ) := r(f ) \\
s_R (e) &:= s(e)\text{ if }e \in E^1 \setminus (\{f \} \cup r^{-1}(w)) \quad\text{and}\quad s_R (e_f ) := s(e) = x.
\end{align*}
We call $E_R$ the \emph{graph obtained by reducing $E$ at $w$}, and say $E_R$ is a reduction
of $E$ or that $E_R$ is formed by performing move \RR to $E$.
\end{definition}

\begin{definition}[Move \OO: Outsplit at a non-sink]
Let $E = (E^0 , E^1 , r, s)$ be a graph, and let $w \in E^0$ be vertex that is not a sink. 
Partition $s^{-1} (w)$ as a disjoint union of a finite number of nonempty sets
$$s^{-1}(w) = \mathcal{E}_1\sqcup \mathcal{E}_2\sqcup \cdots \sqcup\mathcal{E}_n$$
with the property that at most one of the $\mathcal{E}_i$ is infinite. 
Let $E_O$ denote the graph $(E_O^0, E_O^1, r_O , s_O )$ defined by 
\begin{align*}
E_O^0&:= \setof{v^1}{v \in E^0\text{ and }v \neq w}\cup\{w^1, \ldots, w^n\} \\
E_O^1&:= \setof{e^1}{e \in E^1\text{ and }r(e) \neq w}\cup \setof{e^1, \ldots , e^n}{e \in E^1\text{ and }r(e) = w} \\
r_{E_O} (e^i ) &:= 
\begin{cases}
r(e)^1 & \text{if }e \in E^1\text{ and }r(e) \neq w\\
w^i & \text{if }e \in E^1\text{ and }r(e) = w
\end{cases} \\
s_{E_O} (e^i ) &:= 
\begin{cases}
s(e)^1 & \text{if }e \in E^1\text{ and }s(e) \neq w \\
s(e)^j & \text{if }e \in E^1\text{ and }s(e) = w\text{ with }e \in \mathcal{E}_j.
\end{cases}
\end{align*}
We call $E_O$ the \emph{graph obtained by outsplitting $E$ at $w$}, and say $E_O$ is formed by
performing move \OO to $E$.
\end{definition}

\begin{definition}[Move \II: Insplit at a regular non-source]
Suppose that $E = (E^0 , E^1 , r, s)$ is a graph, and let $w \in E^0$ be a regular vertex that is not a source.
Partition $r^{-1} (w)$ as a disjoint union of a finite number of nonempty sets 
$$r^{-1} (w) = \mathcal{E}_1\sqcup \mathcal{E}_2\cdots\sqcup\mathcal{E}_n.$$
Let $E_I$ denote the graph $(E_I^0 , E_I^1 , r_I , s_I )$ defined by
\begin{align*}
E_I^0 &:= \setof{v^1}{v \in E^0\text{ and }v \neq w} \cup \{w^1,\ldots, w^n \} \\
E_I^1 &:= \setof{e^1}{e \in E^1\text{ and }s(e) \neq w} \cup \setof{e^1, \ldots, e^n}{e \in E^1\text{ and }s(e) = w} \\
r_{E_I} (e^i ) &:= 
\begin{cases}
r(e)^1 &\text{if }e \in E^1\text{ and }r(e) \neq w \\
r(e)^j &\text{if }e \in E^1\text{ and }r(e) = w\text{ with }e \in \mathcal{E}_j
\end{cases} \\
s_{E_I} (e^i ) &:= 
\begin{cases}
s(e)^1 &\text{if }e \in E^1\text{ and }s(e) \neq w \\
w^i &\text{if }e \in E^1\text{ and }s(e) = w.
\end{cases}
\end{align*}
We call $E_I$ the \emph{graph obtained by insplitting $E$ at $w$}, and say $E_I$ is formed by performing move \II to $E$.
\end{definition}

\begin{definition}[Move \CC: Cuntz splicing at a regular vertex supporting two return paths] \label{def:cuntzsplice}
Let $E = (E^0 , E^1 , r , s )$ be a graph and let $v \in E^0$ be a regular vertex that supports at least two return paths.
Let $E_C$ denote the graph $(E_C^0 , E_C^1 , r_C , s_C)$ defined by 
\begin{align*}
E_C^0 &:= E^0\sqcup\{u_1 , u_2 \} \\
E_C^1 &:= E^1\sqcup\{e_1 , e_2 , f_1 , f_2 , h_1 , h_2 \},
\end{align*}
where $r_{C}$ and $s_{C}$ extend $r$ and $s$, respectively, and satisfy
$$s_{C} (e_1 ) = v,\quad s_{C} (e_2 ) = u_1 ,\quad s_{C} (f_i ) = u_1 ,\quad s_{C} (h_i ) = u_2 ,$$
and
$$r_{C} (e_1 ) = u_1 ,\quad r_{C} (e_2 ) = v,\quad r_{C} (f_i ) = u_i ,\quad r_{C} (h_i ) = u_i . $$
We call $E_C$ the \emph{graph obtained by Cuntz splicing $E$ at $v$}, and say $E_C$ is formed by performing move \CC to $E$. 

We also use the notation $E_{v,-}$ for this graph --- even in the case where $v$ is not regular or not supporting two return paths. We can also Cuntz splice the vertex $u_1$ in $E_{v,-}$, and the resulting graph we denote $E_{v,--}$. 
See also Notation~\ref{notation:OnceAndTwice} and Example~\ref{example:cuntz-splice} for illustrations of the Cuntz splice.
\end{definition}

\begin{definition}
The equivalence relation generated by the moves \OO, \II, \RR, \SSS together with graph isomorphism is called \emph{move equivalence}, and denoted \Meq. 
The equivalence relation generated by the moves \OO, \II, \RR, \SSS, \CC together with graph isomorphism is called \emph{move prime equivalence}, and denoted \MCeq. 
\end{definition}

The following theorem follows from \cite[Propositions~3.1, 3.2 and 3.3 and Theorem~3.5]{MR3082546}.

\begin{theorem}[\cite{MR3082546}]\label{thm:moveimpliesstableisomorphism}
Let $E_1$ and $E_2$ be graphs such that $E_1\Meq E_2$. 
Then $C\sp*(E_1)\otimes \K\cong C\sp*(E_2)\otimes \K$. 
\end{theorem}

We also extend the notation of move equivalence to adjacency matrices. 

\begin{definition}
If $A,A'$ are square matrices with entries in $\N_0\sqcup\{\infty\}$ we define them to be \emph{move equivalent}, and write $A \Meq A'$ if $\Esf_A \Meq \Esf_{A'}$. 
We define \emph{move prime equivalence} similarly. 
\end{definition}


\section{Main result}
\label{sec:main}

\begin{theorem}\label{thm:main}
Let $E_1$ and $E_2$ be graphs with finitely many vertices satisfying Condition~(K).
Then the following are equivalent:
\begin{enumerate}
\item\label{thm:main-item-1} 
$E_{1} \MCeq E_{2}$, 
\item\label{thm:main-item-2} 
$C^{*} (E_{1} )\otimes\K \cong C^{*} ( E_{2} )\otimes\K$, and,
\item\label{thm:main-item-3} 
$\FKRplus( C^{*} (E_{1} ) ) \cong \FKRplus( C^{*} ( E_{2} ) )$.
\end{enumerate}
\end{theorem}


\subsection{Strategy of proof and structure of the paper}
The proof of the main theorem above, Theorem~\ref{thm:main}, is structured as follows. 

Section~\ref{sec:cuntzsplice} is devoted to show that the move \CC gives stable isomorphism. 
Thus, \eqref{thm:main-item-1} implies \eqref{thm:main-item-2} follows from Theorem~\ref{thm:moveimpliesstableisomorphism} and Proposition~\ref{prop:cuntzspliceinvariant} --- a variant for finite graphs is in \cite{Eilers-Ruiz-Sorensen}. 

That \eqref{thm:main-item-2} implies \eqref{thm:main-item-3} is clear. 

The rest of the paper is devoted to proving that \eqref{thm:main-item-3} implies \eqref{thm:main-item-1}.
Mainly we emulate the previous proofs that go from filtered K-theory data to stable isomorphism or flow equivalence, as in \cite{MR1990568, MR1907894, MR2270572}.
A key component of those proofs is manipulation of the matrix $\Bsf_E^{\bullet}$, in particular that we can perform basic row and column operations without changing stable isomorphism class or flow equivalence class, depending on context. 
We prove in Section~\ref{subsec:movesonmatrices} that these matrix manipulation are allowed.
Once we understand matrix manipulations our proof that \eqref{thm:main-item-3} implies \eqref{thm:main-item-1} goes through 5 steps.

\begin{enumerate}
 \item[Step 1] \label{mainthm-step-standardform}  First we find graphs $F_1$ and $F_2$ in a certain standard form such that $F_i \MCeq E_i$.
 This standard form will ensure that the adjacency matrices $\Bsf^{\bullet}_{F_i}$ have the same size and block structure, and that they satisfy certain additional technical conditions. This will be done in Section~\ref{sec:standardform}. 
 
 \item[Step 2] \label{mainthm-step-FKtoGL} In Section~\ref{sec:boylehuang} we generalize a result of Boyle and Huang (\cite{MR1990568}), to show that the isomorphism $\FKRplus( C^{*} (F_{1} ) ) \cong \FKRplus( C^{*} ( F_{2} ) )$ is induced by a \GLPEe from $\Bsf^{\bullet}_{F_1}$ to $\Bsf^{\bullet}_{F_2}$. 
 
 \item[Step 3] \label{mainthm-step-GLtoSL} In Section~\ref{sec:crelle-trick} we find graphs $G_1, G_2$ such that $G_i \MCeq F_i$ and $\Bsf^{\bullet}_{G_1}$ and $\Bsf^{\bullet}_{G_2}$ are \SLPE{}. 
 
 \item[Step 4] \label{mainthm-step-SLtoMoves} Then, in Section~\ref{sec:boyle}, we generalize Boyle's positive factorization result from \cite{MR1907894} to show that there exists a positive \SLPEe between $\Bsf^{\bullet}_{G_1}$ and $\Bsf^{\bullet}_{G_2}$.
 
 \item[Step 5] It now follows from the results of Section~\ref{subsec:movesonmatrices} that $G_1 \MCeq G_2$ and hence that $E_1 \MCeq E_2$. 
\end{enumerate}

In Section~\ref{sec:notation}, we introduce some notation and concepts about block matrices needed in the proof.
In Section~\ref{sec:proof}, we combine the results of the previous sections to prove the main theorem. 


\section{Derived moves}
\label{sec:derivedmoves}

\subsection{Moves on graphs}

Here we introduce the derived moves from \cite[Section 5]{MR3082546}.
These are shown not to change the move equivalence class, but using them simplifies working with $\Meq$. 

\begin{definition}[Collapse a regular vertex that does not support a loop] \label{def:collapse}
Let $E = (E^0 , E^1 , r , s )$ be a graph and let $v$ be a regular vertex in $E$ that does not support a loop. 
Define a graph $E_{COL}$ by 
\begin{align*}
E_{COL}^0 &= E^0 \setminus \{v\}, \\
E_{COL}^1 &= E^1 \setminus (r^{-1}(v) \cup s^{-1}(v))
\sqcup \setof{[ef ]}{e \in r^{-1}(v)\text{ and }f \in s^{-1}(v)},
\end{align*}
the range and source maps extends those of $E$, and satisfy $r_{E_{COL}}([ef ]) = r (f )$
and $s_{E_{COL}} ([ef ]) = s (e)$.
\end{definition}

According to \cite[Theorem~5.2]{MR3082546} $E\Meq E_{COL}$ --- in fact, the collapse move can be obtained using move \OO and move \RR.

We also introduce move \TT. 

\begin{definition}
Let $E = (E^0 , E^1 , r , s )$ be a graph and let $\alpha = \alpha_1 \alpha_2 \cdots \alpha_n$ be a path such that $\Asf_E(s(\alpha_1), r(\alpha_1)) = \infty$. 
Define a graph $E_{T}$ by 
\begin{align*}
E_{T}^0 &= E^0, \\
E_{T}^1 &= E^1 \cup \{ \alpha^m \mid m \in \N \}
\end{align*}
the range and source maps extends those of $E$, and satisfy $r_{E_{T}}(\alpha^m) = r(\alpha)$
and $s_{E_{T}}(\alpha^m) = s(\alpha)$.
\end{definition}

By \cite[Theorem 5.4]{MR3082546} $E\Meq E_{T}$. 

\subsection{Moves on matrices}
\label{subsec:movesonmatrices}

Let $E$ be a graph with finitely many vertices. 
In this section we perform row and column additions on $\Bsf_E$ without changing move equivalence class of the associated graphs. 
Our setup is slightly different from what was considered in \cite[Section 7]{MR3082546}, so we redo the proofs from there in our setting. 
There are no substantial changes in the proof technique. 

\begin{lemma} \label{lem:oneStepColumnAdd}
Let $E = (E^0, E^1, r,s)$ be a graph with finitely many vertices. 
Let $u,v \in E^0$ be distinct vertices. 
Suppose the $(u,v)$'th entry of $\Bsf_E$ is nonzero (\ie, there is an edge from $u$ to $v$), and that the sum of the entries in the $u$'th row of $\Bsf_E$ is strictly greater than 0 (\ie, $u$ emits at least two edges). 
If $B'$ is the matrix formed from $\Bsf_E$ by adding the $u$'th column into the $v$'th column, then 
$$\Asf_E \Meq B' + I. $$
\end{lemma}
\begin{proof}
Fix an edge $f$ from $u$ to $v$. 
Form a graph $G$ from $E$ by removing $f$ but adding for each edge $e \in r^{-1}(u)$ an edge $\bar{e}$ with $s(\bar{e}) = s(e)$ and $r(\bar{e}) = v$. 
We claim that $B' = \Bsf_G$. 
At any entry other than the $(u,v)$'th entry the two matrices have the same values, since we in both cases add entries into the $v$'th column that are exactly equal to the number of edges in $E$. 
At the $(u,v)$'th entry of $\Bsf_G$ we have
\begin{align*}
	(|s_E^{-1}(u) \cap r_E^{-1}(v)| - 1) + |s_E^{-1}(u) \cap r_E^{-1}(u)| &= \Bsf_E(u,v) + \Bsf_E(u,u) = B'(u,v).
\end{align*}
Thus to prove this lemma it suffices to show $E \Meq G$.

Partition $s^{-1}(u)$ as $\mathcal{E}_1 = \{ f \}$ and $\mathcal{E}_2 = s^{-1}(u) \setminus \{ f \}$. 
By assumption $\mathcal{E}_2$ is not empty, so we can use move \OO. 
Doing so yields a graph just as $E$ but where $u$ is replaced by two vertices, $u_1$ and $u_2$.
The vertex $u_1$ receives a copy of everything $u$ did and it emits only one edge.  
That edge has range $v$. 
The vertex $u_2$ also receives a copy of everything $u$ did, and it emits everything $u$ did, except $f$. 
Since $u_1$ is regular and not the base of a loop, we can collapse it.
The resulting graph is $G$ (after we relabel $u_2$ as $u$), so $G \Meq E$.
\end{proof}

We can also add columns along a path. 

\begin{proposition} \label{prop:columnAdd}
Let $E = (E^0, E^1, r,s)$ be a graph with finitely many vertices. 
Suppose $u,v \in E^0$ are distinct vertices with a path from $u$ to $v$ going through distinct vertices $u = u_0, u_1, u_2, \ldots, u_n = v$ (labelled so there is an edge from $u_i$ to $u_{i+1}$ for $i=0,1,2,\ldots,n-1$). 
Suppose further that for each $i = 0,1,2,\ldots, n-1$ the vertex $u_i$ emits at least two edges. 
If $B'$ is the matrix formed from $\Bsf_E$ by adding the $u$'th column into the $v$'th column, then 
\[
	\Asf_E \Meq B' + I. 
\]
\end{proposition}
\begin{proof}
By repeated applications of Lemma \ref{lem:oneStepColumnAdd}, we first add the $u_{n-1}$'th column into the $u_n$'th column, which we can since there is an edge from $u_{n-1}$ to $u_n$.
Then we add the $u_{n-2}$'th column into the $u_n$'th column, which we can since there now is an edge from $u_{n-2}$ to $u_n$. 
Continuing this way, we end up with a matrix $C$ which is formed from $\Bsf_E$ by adding all the columns $u_i$, for $i = 0,1,2,\ldots, n-1$, into the the $u_n$'th column. 
We have that $\Asf_E \Meq C + I$

Consider the matrix $D$ that is formed from $\Bsf_E$ by adding all the columns $u_0$ and $u_i$, for $i = 2,3\ldots, n-1$, into the the $u_n$'th column. 
Adding the $u_1$'th column in $D$ into the $u_n$'th column yields $C$. 
So by Lemma \ref{lem:oneStepColumnAdd}, which applies since in $\Esf_{D+I}$ there is an edge from $u_1$ to $u_n$, we get that $D + I \Meq C + I \Meq \Asf_E$. 
Similarly we see that $D + I$ is move equivalent to the matrix formed from $\Bsf_E$ by adding all the columns $u_0$ and $u_i$, for $i = 3\ldots, n-1$, into the the $u_n$'th column. 
Continuing to subtract columns in this fashion, we get that $\Asf_E \Meq B' + I$.
\end{proof}

\begin{remark} \label{rmk:columnAdd}
Similar to how we used Lemma \ref{lem:oneStepColumnAdd} in the above proof, we can use Proposition \ref{prop:columnAdd} ``backwards'' to subtract columns in $\Bsf_E$ as long as the addition that undoes the subtraction would be legal. 
\end{remark}

We now turn to row additions. 

\begin{lemma} \label{lem:oneStepRowAdd}
Let $E = (E^0, E^1, r,s)$ be a graph with finitely many vertices. 
Let $u,v \in E^0$ be distinct vertices. 
Suppose the $(v,u)$'th entry of $\Bsf_E$ is nonzero (\ie, there is an edge from $v$ to $u$), that the sum of the entries in the $u$'th column of $\Bsf_E$ is strictly greater than 0 (\ie, $u$ receives at least two edges), and that $u$ is a regular vertex. 
If $B'$ is the matrix formed from $\Bsf_E$ by adding the $u$'th row into the $v$'th row, then 
\[
	\Asf_E \Meq B' + I. 
\]
\end{lemma}
\begin{proof}
Fix an edge $f$ from $v$ to $u$.
Form a graph $G$ from $E$ by removing $f$ but adding for each edge $e \in s^{-1}(u)$ an edge $\bar{e}$ with $s(\bar{e}) = v$ and $r(\bar{e}) = r(e)$. 
We claim that $E \Meq G$. 
Arguing as in the proof of Lemma \ref{lem:oneStepColumnAdd} we see that this is equivalent to proving $\Asf_E \Meq B' + I$.

Partition $r^{-1}(u)$ as $\mathcal{E}_1 = \{ f \}$ and $\mathcal{E}_2 = r^{-1}(u) \setminus \{ f \}$. 
By our assumptions on $u$, $\mathcal{E}_2$ is nonempty, and $u$ is regular, so we can use move \II.  
Doing so replaces $u$ with two new vertices, $u_1$ and $u_2$. 
The vertex $u_1$ only receives one edge, and that edge comes from $v$, the vertex $u_2$ receives the edges $u$ received except $f$.
Since $u_1$ is regular and not the base of a loop of length one we can collapse it.
The resulting graph is $G$ (after we relabel $u_2$ as $u$), so $G \Meq E$.
\end{proof}

Naturally we can also add rows along a path. 

\begin{proposition} \label{prop:rowAdd}
Let $E = (E^0, E^1, r,s)$ be a graph with finitely many vertices. 
Suppose $u,v \in E^0$ are distinct vertices with a path from $v$ to $u$ going through distinct vertices $v = v_0, v_1, v_2, \ldots, v_n = u$ (labelled so there is an edge from $v_i$ to $v_{i+1}$ for $i=0,1,2,\ldots,n-1$). 
Suppose further that for each $i = 1,2,\ldots, n$ the vertex $v_i$ is regular and receives at least two edges. 
If $B'$ is the matrix formed from $\Bsf_E$ by adding the $u$'th row into the $v$'th row, then 
\[
	\Asf_E \Meq B' + I. 
\]
\end{proposition}
\begin{proof}
The proof is completely analogous to the proof of Proposition \ref{prop:columnAdd}.
\end{proof}

\begin{remark} \label{rmk:rowAdd}
We can also use Proposition \ref{prop:rowAdd} ``backwards'' to subtract rows in $\Bsf_E$ (\cf Remark \ref{rmk:columnAdd}). 
\end{remark}


\section{Cuntz splice implies stable isomorphism}
\label{sec:cuntzsplice}

In this section we prove that \eqref{thm:main-item-1} implies \eqref{thm:main-item-2} in Theorem~\ref{thm:main}. 
We know that the moves \OO, \II, \RR, \SSS imply stable isomorphism, \cf\ Theorem~\ref{thm:moveimpliesstableisomorphism}.
What is missing is to prove that if $E_1$ and $E_2$ are graphs with finitely many vertices satisfying Condition (K) and $E_1$ is the Cuntz splice of $E_2$ on a vertex that supports at least two distinct return paths then $C^*(E_1)\otimes \K\cong C^*(E_2)\otimes \K$, which is what we prove in Proposition~\ref{prop:cuntzspliceinvariant}. 
This is also an important result needed in Section~\ref{sec:crelle-trick}.

First we reduce to the case where we perform a Cuntz splice on a regular vertex that supports at least two loops. 

\begin{proposition} \label{prop:cuntzSpliceSetup}
Let $E$ be a graph with finitely many vertices, and let $u \in E^0$ be a vertex that supports at least two distinct return paths. 
Then there exists a graph $F$ and a regular vertex $v \in F^0$ such that 
\begin{enumerate}
	\item $E \Meq F$,
	\item $E_{u,-} \Meq F_{v,-}$ and $E_{u,--} \Meq F_{v,--}$, 
	\item $v$ supports at least two loops, and, 
	\item for all $w \in F^0$ with $v \geq w \geq v$ we have that $w$ supports at least one loop, there is a path from $v$ to $w$ through regular vertices, and there is a path from $w$ to $v$ through regular vertices (we say that a path $e_1e_2\cdots e_n$ goes through regular vertices if $s(e_i)$ is regular for all $i=2,3,\ldots,n$). 
\end{enumerate}
\end{proposition}
\begin{proof}
Let $w \in E^0\setminus\{u\}$ be a regular vertex such that $u \geq w \geq u$. 
If $w$ does not support a loop we can use the collapse move (Definition \ref{def:collapse}) to remove it. 
The resulting graph will be move equivalent to $E$ and have fewer regular vertices $z$ with $u \geq z \geq u$ that do not support a loop. 
So by repeatedly collapsing regular vertices that do not support loops we arrive at a graph $E_1$ such that $E \Meq E_1$, and since the Cuntz splice has no bearing on the collapse move we also see that $E_{u,-} \Meq (E_1)_{u,-}$ and $E_{u,--} \Meq (E_1)_{u,--}$.

For each infinite emitter in $w \in E_1^0\setminus\{u\}$ with $u \geq w \geq u$, we can apply move \TT to assure that there is at least one loop based at $w$. 
Call the resulting graph $E_2$. 
Again we have that $E \Meq E_2$ and again the Cuntz splice is irrelevant for our move so $E_{u,-} \Meq (E_2)_{u,-}$ and $E_{u,--} \Meq (E_2)_{u,--}$. 
Thus we have now found a graph where every vertex $w \neq u$ with $w \geq u \geq w$ supports at least one loop. 

We will now modify $E_2$ to get the desired paths to and from $u$ through regular vertices. 
Let $w\in E_2^0\setminus\{u\}$ be a vertex with $u \geq w \geq u$. 
Suppose every path from $u$ to $w$ goes through an infinite emitter and pick a path $e_1e_2 \cdots e_n$ from $u$ to $w$ of minimal length (in particular it does not contain any loops nor does it visit $u$ again). 
Let $l$ be the first index such that $s(e_l)$ is an infinite emitter, and note that $e_l$ is not a loop. 
Partition $s^{-1}(s(e_l))$ into two sets, one of them $\{e_l\}$, and then outsplit according to this partition. 
After the outsplit we can collapse the vertex that emits $e_l$, since $e_l$ is the only edge it emits. 
Notice that in the post-collapse graph, the singular vertices are the same, and all the paths that were in the graph are still present, and each vertex $z\neq u$ with $u \geq z \geq u$ still supports at least one loop. 
We now have an edge from $s(e_{l-1})$ to $r(e_l)$, so we can change our path to avoid $s(e_l)$. 
Continuing in this fashion we eventual modify $E_2$ in such a way that there is a path from $u$ to $w$ using only regular vertices. 
Now we continue to do this for every such vertex $w$.

Exactly the same strategy lets us assure that there is a path from $w$ to $u$ through regular vertices when $u \geq w \geq u$. 
Call the graph that emerges after all these moves $E_3$.

Since we only did outsplits and collapses on vertices in $E_2^0 \setminus \{u\}$, we see that these moves are unaffected by the Cuntz splice. 
Thus we have $E \Meq E_3$, $E_{u,-} \Meq (E_3)_{u,-}$ and $E_{u,--} \Meq (E_3)_{u,--}$.

Now we want to modify $E_3$ such that $u$ has at least two loops. 
If not, then since $u$ supports two distinct return paths there exists some vertex $w\neq u$ such that $w \geq u$ and $|s^{-1}(u) \cap r^{-1}(w)| \geq 1$. 
As every vertex $z\neq u$ with $u \geq z \geq u$ supports a loop, we can use Proposition~\ref{prop:columnAdd} to add the $w$'th column of $\Bsf_{E_3}$ into the $u$'th column twice. 
Call the resulting matrix $B'$, and let $E_4 = \Esf_{B' + I}$.
In $E_4$, $u$ will support (at least) two loops and all the other properties are preserved, since $w$ supports a loop. 
The column addition is also valid in $(E_3)_{u,-}$ and $(E_3)_{u,--}$, so we have $E \Meq E_4$, $E_{u,-} \Meq (E_4)_{u,-}$ and $E_{u,--} \Meq (E_4)_{u,--}$.

We will do the proof in cases. 

{\textbf{Case 1:}} If $u$ is regular, then we can end Case~1 by letting $F = E_4$ and $v = u$. 

{\textbf{Case 2:}} $u$ is an infinite emitter and there exists $w_0 \in E^0$ such that $w_0 \geq u$ and $|s^{-1}(u) \cap r^{-1}(w_0)| = \infty$.

Doing what we did above and using move $\TT$ we can find a graph $E_5$ such that 
\begin{enumerate}[(i)]
	\item $E \Meq E_5$,
	\item $E_{u,-} \Meq (E_5)_{u,-}$ and $E_{u,--} \Meq (E_5)_{u,--}$, 
	\item $u$ supports infinitely many loops,
	\item if $u \geq w \geq u$ then there are infinitely many edges from $u$ to $w$, and,
	\item for all $w \in E_5^0$ with $u \geq w \geq u$ we have that $w$ supports at least one loop, there is a path from $u$ to $w$ through regular vertices, and there is a path from $w$ to $u$ through regular vertices. 
\end{enumerate}

Pick two edges $e_1,e_2 \in s^{-1}(u) \cap r^{-1}(u)$, and pick for each $u \geq w \geq u $, $w \neq u$, one edge $e_w \in s^{-1}(u) \cap r^{-1}(w)$.
Partition $s^{-1}(u)$ as into two sets, one which is 
\[
	\mathcal{E}_1 = \{ e_1, e_2 \} \cup \{ e_w \mid u \geq w \geq u, w \neq u \}.
\]
Out-splitting according to this partition we get a graph $F$ with $E \Meq F$.
We will show that $F_{u_1,-} \Meq (E_5)_{u,-}$ and $F_{u_1,--} \Meq (E_5)_{u,--}$. 
Hence putting $v = u_1$ will complete the proof of this case. 

In $(E_5)_{u,-}$ we call the two vertices in the Cuntz splice $v_1$ and $v_2$, and let $f$ be the edge from $u$ to $v_1$. 
If we outsplit at $u$ by partitioning $s^{-1}(u)$ into two sets, one of which is 
\[
	\mathcal{F}_1 = \{ e_1, e_2, f \} \cup \{ e_w \mid u \geq w \geq u, w \neq u \}
\]
we get a graph $F_1 \Meq (E_5)_{u,-}$, which is just like $F_{u_1,-}$, except that in $F_1$, there is an edge from $v_1$ to $u_2$, while there is no such edge in $F_{u_1,-}$. 
But Proposition~\ref{prop:columnAdd} lets us add the $v_2$'th column in $\Bsf_{F_{u_1,-}}$ to the $u_2$'th, to show that $F_1 \Meq F_{u_1,-}$. 

A completely analogue argument shows that $F_{u_1,--} \Meq (E_5)_{u,--}$. Letting $v=u_1$ finishes Case~2. 

{\textbf{Case 3:}} $u$ is an infinite emitter and for all $w \in E_4^0$ with $|s^{-1}(u) \cap r^{-1}(w)| = \infty$ we have $w \ngeq u$.

We will perform an outsplit at $u$, by partitioning $s^{-1}(u)$ into two sets, one of which is 
\[
	\mathcal{E}_1 = \{ e \in s^{-1}(u) \mid r(e) \geq u \}. 
\]
Similarly to Case 2, we see that the only difference between outsplitting according to this partition before or after we perform the Cuntz splice is as edge from $v_1$ to $u_2 $ (notation as above).
Hence, we see as above that if we let $F$ be the outsplit graph coming from $E_4$, then $E_4 \Meq F$, $(E_4)_{u,-} \Meq F_{u_1,-}$ and $(E_4)_{u,--} \Meq F_{u_1,--}$. 
Letting $v=u_1$ finishes Case 3. 
\end{proof}

We now show that performing the Cuntz splice twice is a legal move. 

\begin{proposition} \label{prop:cuntzsplicetwice}
Let $E$ be a graph with finitely many vertices, and let $v$ be a vertex that supports at least two distinct return paths. 
Then $E \Meq E_{v,--}$.  
\end{proposition}
\begin{proof}
According to Proposition~\ref{prop:cuntzSpliceSetup}, we can assume that $E$ satisfies the conditions of that proposition --- so we assume that $v$ is a regular vertex that supports at least two loops. Moreover, for convenience, we let $n$ be the number of vertices in $E$ and we label the vertices by the numbers $1,2,\ldots,n$ in such a way that $v$ gets the label $n$. 

For a given matrix size $N$ and $i,j\in\{1,2,\ldots,N\}$, we let $E_{(i,j)}$ denote the $N\times N$ matrix that is equal to the identity matrix everywhere except for the $(i,j)$'th entry, that is $1$. 
If $B$ is a $N\times N$ matrix, then $E_{(i,j)}B$ is the matrix obtained from $B$ by adding $j$'th row into the $i$'th row, and $BE_{(i,j)}$ is the matrix obtained from $B$ by adding $i$'th column into the $j$'th column. 
Using $E_{(i,j)}^{-1}$ instead will yield subtraction. 
In what follows we will make extensive use of Propositions \ref{prop:columnAdd} and \ref{prop:rowAdd} and Remarks \ref{rmk:columnAdd} and \ref{rmk:rowAdd}, we feel it will only muddle the exposition if we add all the references in.

Note that $\Bsf_{E_{v,--}}$ can be written as 
$$B_1=
\begin{pmatrix} 
\Bsf_E & 
\begin{pmatrix}
0 & 0 & 0 & 0 \\
\vdots &\vdots &\vdots &\vdots \\
0 & 0 & 0 & 0 \\
0 & 1 & 0 & 0 
\end{pmatrix} \\ 
\begin{pmatrix} 
0 & \cdots & 0 & 0 \\
0 & \cdots & 0 & 1 \\
0 & \cdots & 0 & 0 \\
0 & \cdots & 0 & 0 
\end{pmatrix} & 
\begin{pmatrix}
0 & 1 & 0 & 0 \\
1 & 0 & 1 & 0 \\
0 & 1 & 0 & 1 \\
0 & 0 & 1 & 0 
\end{pmatrix} 
\end{pmatrix}.$$
Now let 
$B_2=E_{(n+2,n+3)}B_1$ and 
$B_3=B_2E_{(n+3,n+4)}^{-1}$. 
Then $B_1 + I \Meq B_2 + I \Meq B_3 + I$. 
We have that 
$$B_3=
\begin{pmatrix} 
\Bsf_E & 
\begin{pmatrix}
0 & 0 & 0 & 0 \\
\vdots &\vdots &\vdots &\vdots \\
0 & 0 & 0 & 0 \\
0 & 1 & 0 & 0 
\end{pmatrix} \\ 
\begin{pmatrix} 
0 & \cdots & 0 & 0 \\
0 & \cdots & 0 & 1 \\
0 & \cdots & 0 & 0 \\
0 & \cdots & 0 & 0 
\end{pmatrix} & 
\begin{pmatrix}
0 & 1 & 0 & 0 \\
1 & 1 & 1 & 0 \\
0 & 1 & 0 & 1 \\
0 & 0 & 1 & -1
\end{pmatrix} 
\end{pmatrix}$$
The $n+4$'th vertex in $\Esf_{B_3 + I}$ does not support a loop, so it can be collapsed yielding
$$B_4=
\begin{pmatrix} 
\Bsf_E & 
\begin{pmatrix}
0 & 0 & 0 \\
\vdots &\vdots &\vdots \\
0 & 0 & 0 \\
0 & 1 & 0  
\end{pmatrix} \\ 
\begin{pmatrix} 
0 & \cdots & 0 & 0 \\
0 & \cdots & 0 & 1 \\
0 & \cdots & 0 & 0 
\end{pmatrix} & 
\begin{pmatrix}
0 & 1 & 0 \\
1 & 1 & 1 \\
0 & 1 & 1 
\end{pmatrix} 
\end{pmatrix}.$$
With $B_4 + I \Meq B_3 + I$. 
Now we let 
$B_5=E_{(n+2,n+1)}^{-1}B_4$, 
$B_6=E_{(n,n+3)}B_5$, 
$B_7=E_{(n,n+1)}^{-1}E_{(n,n+1)}^{-1}B_6$, 
$B_8=E_{(n+3,n+2)}B_7$ and 
$B_9=B_8E_{(n+2,n+1)}^{-1}$. 
We then have 
$B_4 + I \Meq B_5 + I \Meq B_6 + I \Meq B_7 + I \Meq B_8 + I\Meq B_9 + I$. 
We have that 
$$B_9=
\begin{pmatrix} 
\Bsf_E & 
\begin{pmatrix}
0 & 0 & 0 \\
\vdots &\vdots &\vdots \\
0 & 0 & 0 \\
0 & 0 & 1  
\end{pmatrix} \\ 
\begin{pmatrix} 
0 & \cdots & 0 & 0 \\
0 & \cdots & 0 & 1 \\
0 & \cdots & 0 & 1 
\end{pmatrix} & 
\begin{pmatrix}
-1 & 1 & 0 \\
1 & 0 & 1 \\
0 & 1 & 2 
\end{pmatrix} 
\end{pmatrix}$$
In $\Esf_{B_9 + I}$ the $n+1$'th vertex does not support a loop, so it can be collapsed to yield
$$B_{10}=
\begin{pmatrix} 
\Bsf_E & 
\begin{pmatrix}
0 & 0 \\
\vdots &\vdots \\
0 & 0 \\
0 & 1  
\end{pmatrix} \\ 
\begin{pmatrix} 
0 & \cdots & 0 & 1 \\
0 & \cdots & 0 & 1 
\end{pmatrix} & 
\begin{pmatrix}
1 & 1 \\
1 & 2 
\end{pmatrix} 
\end{pmatrix}.$$
with $B_9 + I \Meq B_{10} + I$. 

Now we look at the graph $E$ again, and and let $\Bsf_E = (b_{ij})$. 
Since the vertex $v$ (number $n$) has at least two loops, we have $b_{nn}\geq 1$. 
Now we can insplit by partitioning $r^{-1}(v)$ into two sets, one with a single set consisting of a loop based at $v$, and the other the rest. 
In the resulting graph, $v$ is split into two vertices $v^1$ and $v^2$, and let $E'$ denote the rest of the graph. The vertex $v^1$ has the same edges in and out of $E'$ as $v$ had, but it has only $b_{nn}$ loops. There is one edge from $v^1$ to $v^2$ and $v^2$ has one loop and there are $b_{nn}$ edges from $v^2$ to $v^1$ as well as all the same edges going from $v^2$ into $E'$ as originally from $v$. 
Use the inverse collapse move to add a new vertex $u$ to the middle of the edge from $v^1$ to $v^2$ and call the resulting graph $F$. 
Label the vertices such that  $v^1$, $u$ and $v^2$ are the $n$'th, $n+1$'st and $n+2$'nd vertex, then $\Bsf_F$ is:
$$\Bsf_F =
\begin{pmatrix} 
\widetilde{B} & 
\begin{pmatrix}
0 & 0 \\
\vdots &\vdots \\
0 & 0 \\
1 & 0  
\end{pmatrix} \\ 
\begin{pmatrix} 
0 & \cdots & 0 & 0 \\
b_{n1} & \cdots & b_{n,n-1} & b_{nn} 
\end{pmatrix} & 
\begin{pmatrix}
-1 & 1 \\
0 & 0 
\end{pmatrix} 
\end{pmatrix},$$
where $\widetilde{B}$ is $\Bsf_E$ except for on the $(n,n)$'th entry, which is $b_{nn}-1$. 
Note that $b_{nn}-1\geq 0$, so that there is still a loop based at the $n$'th vertex.
This is important since it allows us to do the following matrix manipulations. 
Let $C_2=\Bsf_F E_{(n+2,n+1)}E_{(n+2,n+1)}$, 
$C_3=E_{(n+2,n+1)}C_2$, 
$C_4=E_{(n+2,n)}^{-1}C_3$,
$C_5=C_4E_{(n+1,n)}$ and
$C_6=C_5E_{(n+2,n+1)}$. 
We have that
$C_1 + I \Meq C_2 + I \Meq C_3 + I \Meq C_4 + I \Meq C_5 + I \Meq C_6 + I$. 
The matrix $C_6$ is in fact equivalent to $B_{10}$ upon relabelling of the last two vertices, thus it follows, that $E\Meq E_{v,--}$. 
\end{proof}

We now show that Cuntz splicing once and twice yields isomorphic graph $C^{*}$-algebras.  To do this, we first set up some notation.  

\begin{notation}\label{notation:OnceAndTwice}
Let $\mathbf{E}_*$ and $\mathbf{E}_{**}$ denote the graphs: 
\begin{align*}
\mathbf{E}_* \  = \ \ \ \ \xymatrix{
  \bullet^{v_1} \ar@(ul,ur)[]^{e_{1}} \ar@/^/[r]^{e_{2}} & \bullet^{v_2} \ar@(ul,ur)[]^{e_{4}} \ar@/^/[l]^{e_{3}}
}
\end{align*}
\begin{align*}
\mathbf{E}_{**} \  =  \ \ \ \ \xymatrix{
	\bullet^{ w_{4} } \ar@(ul,ur)[]^{f_{10}}  \ar@/^/[r]^{ f_{9} } & \bullet^{ w_{3} } \ar@(ul,ur)[]^{f_{7}} \ar@/^/[r]^{ f_{6} }  \ar@/^/[l]^{f_{8}} &  \bullet^{w_1} 				\ar@(ul,ur)[]^{f_{1}} \ar@/^/[r]^{f_{2}} \ar@/^/[l]^{f_{5}}
	& \bullet^{w_2} \ar@(ul,ur)[]^{f_{4}} \ar@/^/[l]^{f_{3}}
	}
\end{align*}
The graph $\mathbf{E}_*$ is what we attach when we Cuntz splice, if we instead attach the graph $\mathbf{E}_{**}$ it is like we Cuntz spliced twice. 

Let $E = ( E^{0}, E^{1} , r_{E}, s_{E} )$ be a graph and let $u$ be a vertex of $E$.
Then $E_{u, -}$ can be described as follows (up to canonical isomorphism):
\begin{align*}
E_{u,-}^{0} &= E^{0} \sqcup \mathbf{E}_{*}^{0} \\
E_{u,-}^{1} &= E^{1} \sqcup \mathbf{E}_{*}^{1} \sqcup \{ d_1, d_2 \}
\end{align*}
with $r_{E_{u,-}} \vert_{E^{1}} = r_{E}$, $s_{E_{u,-}} \vert_{ E^{1} } = s_{E}$, $r_{E_{u,-}} \vert_{\mathbf{E}_{*}^{1}} = r_{\mathbf{E}_{*}}$, $s_{E_{u,-}} \vert_{\mathbf{E}_{*}^{1}} = s_{\mathbf{E}_{*}}$, and
\begin{align*}
	s_{E_{u,-}}(d_1) &= u	& r_{E_{u,-}}(d_1) &= v_{1} \\
	s_{E_{u,-}}(d_2) &= v_1	& r_{E_{u,-}}(d_2) &= u.
\end{align*}
Moreover, $E_{u,--}$ can be described as follows (up to canonical isomorphism):
\begin{align*}
E_{u,--}^{0} &= E^{0} \sqcup \mathbf{E}_{**}^{0} \\
E_{u,--}^{1} &= E^{1} \sqcup \mathbf{E}_{**}^{1} \sqcup \{ d_1, d_2 \}
\end{align*}
with $r_{E_{u,--}} \vert_{E^{1}} = r_{E}$, $s_{E_{u,--}} \vert_{ E^{1} } = s_{E}$, $r_{E_{u,--}} \vert_{\mathbf{E}_{**}^{1}} = r_{\mathbf{E}_{**}}$, $s_{E_{u,--}} \vert_{\mathbf{E}_{**}^{1}} = s_{\mathbf{E}_{**}}$, and
\begin{align*}
	s_{E_{u,--}}(d_1) &= u		& r_{E_{u,--}}(d_1) &= w_{1} \\
	s_{E_{u,--}}(d_2) &= w_1	& r_{E_{u,--}}(d_2) &= u.
\end{align*}
\end{notation}

\begin{example}\label{example:cuntz-splice}
Consider the graph 
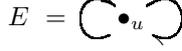
\begin{align*}
E \  = \ \ \ \ \xymatrix{
  \bullet_{u} \ar@(dl,ul) \ar@(ur,dr)
}
\end{align*}
Then 
\begin{align*}
E_{u,-} \  = \ \ \ \ 
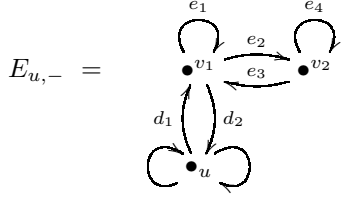
\xymatrix{
  \bullet^{v_1} \ar@/^/[d]^{d_2} \ar@(ul,ur)[]^{e_{1}} \ar@/^/[r]^{e_{2}} & \bullet^{v_2} \ar@(ul,ur)[]^{e_{4}} \ar@/^/[l]_{e_{3}} \\
  \bullet_{u} \ar@/^/[u]^{d_1} \ar@(dl,ul) \ar@(ur,dr) & 
}
\end{align*}
and 
\begin{align*}
E_{u,--} \  = \ \ \ \ 
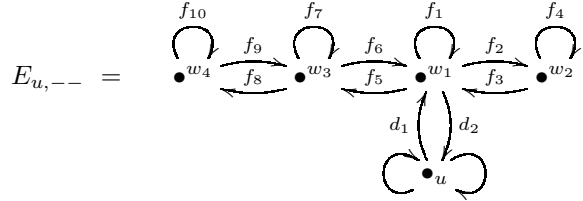
\xymatrix{
	\bullet^{ w_{4} } \ar@(ul,ur)[]^{f_{10}}  \ar@/^/[r]^{ f_{9} } &
	\bullet^{ w_{3} } \ar@(ul,ur)[]^{f_{7}} \ar@/^/[r]^{ f_{6} }  \ar@/^/[l]_{f_{8}} &
	\bullet^{w_1} \ar@/^/[d]^{d_2} \ar@(ul,ur)[]^{f_{1}} \ar@/^/[r]^{f_{2}} \ar@/^/[l]_{f_{5}} &
	\bullet^{w_2} \ar@(ul,ur)[]^{f_{4}} \ar@/^/[l]_{f_{3}} \\
	& & \bullet_{u} \ar@/^/[u]^{d_1} \ar@(dl,ul) \ar@(ur,dr) & 
}
\end{align*}
\end{example}

By classification of simple purely infinite graph \cas, \ie, by Kirchberg-Phillips classification, the graph \cas $C^*(\mathbf{E}_*)$ and $C^*(\mathbf{E}_{**})$ are isomorphic (this important case is actually due to R{\o}rdam, \cf\ \cite{MR1340839}). 
To show that $C^*(E_{u,-})$ is isomorphic to $C^*(E_{u,--})$ we would like to know that $C^*(\mathbf{E}_*)$ and $C^*(\mathbf{E}_{**})$ are still isomorphic if we do not enforce the summation relation at $v_1$ and $w_1$ respectively. 

\begin{proposition} \label{prop:csdouble}
The relative graph \cas (in the sense of Muhly-Tomforde \cite{MR2054981}) $C^*(\mathbf{E}_{*}, \{v_2\})$ and $C^*(\mathbf{E}_{**}, \{ w_2,w_3,w_4 \})$ are isomorphic. 
\end{proposition}

\begin{proof}
Following \cite[Definition~3.6]{MR2054981} we define a graph 
\begin{align*}
(\mathbf{E}_*)_{\{v_2\}} \ = \ \ \ \ \xymatrix{
  \bullet^{v_1} \ar[d]_{e_{1}'} \ar@(ul,ur)[]^{e_{1}} \ar@/^/[r]^{e_{2}} & \bullet^{v_2} \ar@(ul,ur)[]^{e_{4}} \ar@/^/[l]_{e_{3}} \ar@/^/[dl]^{e_{3}'} \\
  \bullet_{v_1'} & 
}
\end{align*} 
Then by \cite[Theorem~3.7]{MR2054981} we have that $C^*(\mathbf{E}_{*}, \{v_2\}) \cong C^*((\mathbf{E}_*)_{\{v_2\}})$.
Similarly we define a graph 
\begin{align*}
(\mathbf{E}_{**})_{\{w_2,w_3,w_4\}} \ = \ \ \ \ \xymatrix{
	\bullet^{w_4} \ar@(ul,ur)[]^{f_{10}}  \ar@/^/[r]^{ f_{9} } &
	\bullet^{w_3} \ar@/_/[dr]_{f_6'} \ar@(ul,ur)[]^{f_{7}} \ar@/^/[r]^{ f_{6} }  \ar@/^/[l]_{f_{8}} &
	\bullet^{w_1} \ar[d]_{f_1'} \ar@(ul,ur)[]^{f_{1}} \ar@/^/[r]^{f_{2}} \ar@/^/[l]_{f_{5}} &
	\bullet^{w_2} \ar@/^/[dl]^{f_3'} \ar@(ul,ur)[]^{f_{4}} \ar@/^/[l]_{f_{3}} \\
	& & \bullet_{w_1'} & 
}
\end{align*}
Using \cite[Theorem~3.7]{MR2054981} again, we have that $C^*(\mathbf{E}_{**}, \{ w_2,w_3,w_4 \})$ is isomorphic to $C^*((\mathbf{E}_{**})_{\{w_2,w_3,w_4\}})$.

Both the graphs $(\mathbf{E}_*)_{\{v_2\}}$ and $(\mathbf{E}_{**})_{\{w_2,w_3,w_4\}}$ satisfy Condition~(K). 
Using the well developed theory of ideal structure and $K$-theory for graph \cas, we see that both have exactly one nontrivial ideal, that this ideal is the compact operators, and that their six-term exact sequences are 
\begin{align*}
\xymatrix{
	\Z \langle v_1' \rangle \ar[r] & \Z \ar[r] & 0 \ar[d] \\
	0 \ar[u] & \ar[l] 0 & \ar[l] 0
}
\ \ \ \ &  \ \ \ \
\xymatrix{
	\Z \langle w_1' \rangle \ar[r] & \Z \ar[r] & 0 \ar[d] \\
	0 \ar[u] & \ar[l] 0 & \ar[l] 0
}
\end{align*}

Furthermore, in $K_0(C^*((\mathbf{E}_*)_{\{v_2\}}))$ we have 
\begin{align*}
	[p_{v_1}] &= -[p_{v_1'}] = [p_{v_2}],
\end{align*}
and in $K_0(C^*((\mathbf{E}_{**})_{\{w_2,w_3,w_4\}}))$ we have 
\begin{align*}
	[p_{w_1}] &= -[p_{w_1'}] = [p_{w_2}], \\
	[p_{w_3}] &= 0 = [p_{w_4}].
\end{align*}
Therefore the class of the unit is $-[p_{v_1'}]$ and $-[p_{w_1'}]$, respectively. 
It now follows from \cite[Theorem~2]{MR1396721} (see also \cite[Corollary~4.20]{arXiv:1301.7695v1}) that $C^*((\mathbf{E}_*)_{\{v_2\}}) \cong C^*((\mathbf{E}_{**})_{\{w_2,w_3,w_4\}})$ and hence that $C^*(\mathbf{E}_{*}, \{v_2\}) \cong C^*(\mathbf{E}_{**}, \{ w_2,w_3,w_4 \})$.
\end{proof}

We also need a technical result about the projections in $\mathcal{E} = C^*(\mathbf{E}_{*}, \{v_2\})$.

\begin{lemma} \label{lem:csmurray}
Let $\mathcal{E} = C^*(\mathbf{E}_{*}, \{v_2\})$ and choose an isomorphism between $\mathcal{E}$ and $C^*(\mathbf{E}_{**}, \{ w_2,w_3,w_4 \})$, which exists according to the previous proposition. 
Let $p_{v_1}$, $p_{v_2}$, $s_{e_1}$, $s_{e_2}$, $s_{e_3}$, $s_{e_4}$ be the canonical generators of $C^*(\mathbf{E}_{*}, \{v_2\})=\mathcal{E}$ and let $p_{w_1}$, $p_{w_2}$, $p_{w_3}$, $p_{w_4}$, $s_{f_1}$, $s_{f_2}, \ldots, s_{f_{10}}$ denote the image of the canonical generators of $C^*(\mathbf{E}_{**}, \{ w_2,w_3,w_4 \})$ in $\mathcal{E}$ under the chosen isomorphism. 
Then 
\begin{align*}
	s_{e_1} s_{e_1}^* + s_{e_2} s_{e_2}^* &\sim s_{f_1} s_{f_1}^* + s_{f_2} s_{f_2}^* +s_{f_5} s_{f_5}^*, \\
	p_{v_1} - \left( s_{e_1} s_{e_1}^* + s_{e_2} s_{e_2}^* \right) &\sim p_{w_1} - \left( s_{f_1} s_{f_1}^* + s_{f_2} s_{f_2}^* +s_{f_5} s_{f_5}^* \right),
\end{align*}
in $\mathcal{E}$, where $\sim$ denotes Murray-von Neumann equivalence. 
\end{lemma}

\begin{proof}
By \cite[Corollary~7.2]{MR2310414}, row-finite graph \cas have stable weak cancellation, so by \cite[Theorem~3.7]{MR2054981}, $\mathcal{E}$ has stable weak cancellation. 
Hence any two projections in $\mathcal{E}$ are Murray-von Neumann equivalent if they generate the same ideal and have the same $K$-theory class. 

As in the proof of Proposition \ref{prop:csdouble}, we will use \cite[Theorem~3.7]{MR2054981} to realize our relative graph \cas as graph \cas of the graphs $(\mathbf{E}_*)_{\{v_2\}}$ and $(\mathbf{E}_{**})_{\{w_2,w_3,w_4\}}$. 
Denote the image of the vertex projections of $C^*((\mathbf{E}_*)_{\{v_2\}})$ inside $\mathcal{E}$ under this isomorphism by $q_{v_1}, q_{v_2}, q_{v_1'}$ and denote the image of the vertex projections of $(\mathbf{E}_{**})_{\{w_2,w_3,w_4\}}$ inside $\mathcal{E}$ under the isomorphisms $(\mathbf{E}_{**})_{\{w_2,w_3,w_4\}}\cong C^*(\mathbf{E}_{**}, \{ w_2,w_3,w_4 \}) \cong\mathcal{E}$ by $q_{w_1}, q_{w_2}, q_{w_3}, q_{w_4}, q_{v_1'}$. 
Using the description of the isomorphism in \cite[Theorem~3.7]{MR2054981}, we see that we need to show that $q_{v_1} \sim q_{w_1}$ and $q_{v_1'} \sim q_{w_1'}$.

Since $(\mathbf{E}_*)_{\{v_2\}}^0$ satisfies Condition~(K) and the smallest hereditary and saturated subset containing $v_1$ is all of $(\mathbf{E}_*)_{\{v_2\}}^0$ we have that $q_{v_1}$ is a full projection (\cite[Theorem~4.4]{MR1988256}).
Similarly $q_{w_1}$ is full. 
In $K_0(\mathcal{E})$ we have, using our calculations from the proof of Proposition \ref{prop:csdouble}, that 
\begin{align*}
	[q_{v_1}] = [1] = [q_{w_1}].
\end{align*}
So by weak stable cancellation $q_{v_1} \sim q_{w_1}$.

Both $q_{v_1'}$ and $q_{w_1'}$ generate the only nontrivial ideal $\mathfrak{I}$ of $\mathcal{E}$ (\cite[Theorem~4.4]{MR1988256}). 
Since that ideal is isomorphic to the compact operators and both $[q_{v_1'}]$ and $[q_{w_1'}]$ are positive generators of  $K_0(\mathfrak{I})\cong K_0(\mathbb{K})\cong \Z$, they must both represent the same class in $K_0(\mathfrak{I})$, and thus also in $K_0(\mathcal{E})$. 
Therefore $q_{v_1'} \sim q_{w_1'}$.
\end{proof}

If $E$ is a graph and we have a set of mutually orthogonal projections $\setof{ p_v }{ v \in E^0 }$ and a set $\setof{ s_e }{ e \in E^1 }$ of partial isometries in a \ca satisfying the relations of Definition~\ref{def:graphca}, then we call these elements a \emph{Cuntz-Krieger $E$-family}. 
In a graph $E$, we call a cycle $e_1e_2\cdots e_n$ a \emph{vertex-simple cycle} if $r(e_i)\neq r(e_j)$ for all $i\neq j$. A vertex-simple cycle $e_1e_2\cdots e_n$ is said to have an \emph{exit} if there exists an edge $f$ such that $s(f)=s(e_k)$ for some $k=1,2,\ldots,n$ with $e_k\neq f$. 
Note that in \cite{MR1914564}, the author uses the term \emph{loop} where we use \emph{cycle}.

\begin{theorem}\label{t:cuntz-splice-1}
Let $E$ be a graph with finitely many vertices and let $u$ be a vertex of $E$.  Then $C^{*}(E_{u,-}) \cong C^{*}(E_{u,--})$.
\end{theorem}

\begin{proof}
As above, we let $\mathcal{E}$ denote the \ca $C^*(\mathbf{E}_{*}, \{v_2\})$, and we choose an isomorphism between $\mathcal{E}$ and $C^*(\mathbf{E}_{**}, \{ w_2,w_3,w_4 \})$, which exists according to Proposition~\ref{prop:csdouble}.

Since $C^*(E_{u,-})$ and $\mathcal{E}$ are unital, separable, nuclear \cas, it follows from Kirchberg's embedding theorem that there exists a unital embedding 
\[
	C^*(E_{u,-}) \oplus \mathcal{E} \hookrightarrow \mathcal{O}_2. 	
\]
We will suppress this embedding in our notation. 
In $\mathcal{O}_2$, we denote the vertex projections and the partial isometries coming from $C^*(E_{u,-})$ by $p_v, v\in E_{u,-}^0$ and $s_e,e\in E_{u,-}^1$, respectively, and we denote the vertex projections and the partial isometries coming from $\mathcal{E}=C^*(\mathbf{E}_*, \{v_2\})$ by $p_1,p_2$ and $s_1, s_2, s_3, s_4$, respectively.
Since we are dealing with an embedding, it follows from Szyma{\'n}ski's General Cuntz-Krieger Uniqueness Theorem (\cite[Theorem~1.2]{MR1914564}) that for any vertex-simple cycle $\alpha_1 \alpha_2 \cdots \alpha_n$ in $E_{u,-}$ without any exit, we have that the spectrum of $s_{\alpha_1} s_{\alpha_2} \cdots s_{\alpha_n}$ contains the entire unit circle.

We will define a new Cuntz-Krieger $E_{u,-}$-family. 
For each vertex $v \in E^0$ we let $q_v = p_v$, we let $q_{v_1} = p_1$ and $q_{v_2} = p_2$.
Since any two nonzero projections in $\mathcal{O}_2$ are Murray-von Neumann equivalent, we can choose partial isometries $x_1, x_2 \in \mathcal{O}_2$ such that 
\begin{align*}
	x_1 x_1^* &= s_{d_1} s_{d_1}^*				& x_1^* x_1 &= p_1 \\
	x_2 x_2^* &= p_1 - (s_1 s_1^* + s_2 s_2^*)	& x_2^* x_2 &= p_u.
\end{align*} 
We let $t_{d_1} = x_1$ and $t_{d_2} = x_2$. 
Finally we let $t_e = s_e$ for $e \in E^1$ and put $t_{e_i} = s_i$ for $i=1,2,3,4$. 

By construction $\{ q_v \mid v \in E_{u,-}^0\}$ is a set of orthogonal projections, and $\{ t_e \mid e \in E_{u,-}^1\}$ a set of partial isometries.
Furthermore, by choice of $\{ t_e \mid e \neq d_1,d_2 \}$ the relations are clearly satisfied at all vertices other than $v_1$ and $u$. 
The choice of $x_1, x_2$ ensures that the relations hold at $u$ and $v_1$ as well. 
Hence $\{q_v, t_e\}$ does indeed form a Cuntz-Krieger $E_{u,-}$ family. 
Denote this family by $\mathcal{S}$.

Using the universal property of graph \cas, we get a $*$-homomorphism from $C^*(E_{u,-})$ onto $C^*(\mathcal{S}) \subseteq \mathcal{O}_2$.
Let $\alpha_1 \alpha_2 \cdots \alpha_n$ be a vertex-simple cycle in $E_{u,-}$ without any exit. 
Since $u$ is where the Cuntz splice is glued on, no vertex-simple cycle without any exit uses edges connected to $u, v_1$ or $v_2$. 
Hence $t_{\alpha_1} t_{\alpha_2} \cdots t_{\alpha_n} = s_{\alpha_1} s_{\alpha_2} \cdots s_{\alpha_n}$ and so its spectrum contains the entire unit circle. 
It now follows from Szyma{\'n}ski's General Cuntz-Krieger Uniqueness Theorem (\cite[Theorem~1.2]{MR1914564}) that $C^*(E_{u,-}) \cong C^*(\mathcal{S})$.

Let $\mathfrak{A}$ be the subalgebra of $\mathcal{O}_2$ generated by $\{ p_v \mid v \in E^0 \}$ and $\mathcal{E}$.
Note that $\mathfrak{A}$ has a unit, and although it does not coincide with the unit of $\mathcal{O}_2$ it does coincide with the unit of $C^*(\mathcal{S})$. 
In fact $\mathfrak{A}$ is  a unital subalgebra of $C^*(\mathcal{S})$.
Let us denote by $\{ r_{w_i}, y_{f_j} \mid i=1,2,3,4, j = 1,2,\ldots, 10 \}$ the image of the canonical generators of $C^*(\mathbf{E}_{**}, \{ w_2,w_3,w_4 \})$ in $\mathcal{O}_2$ under the chosen isomorphism between $C^*(\mathbf{E}_{**}, \{ w_2,w_3,w_4 \})$ and $\mathcal{E}$ composed with the embedding into $\mathcal{O}_2$. 
By Lemma \ref{lem:csmurray}, certain projections in $\mathcal{E}$ are Murray-von Neumann equivalent, hence we can find a unitary $z \in \mathfrak{A}$ such that 
\begin{align*}
	z q_v z^* &= q_v, \text{ for all } v \in E^0, \\
	z \left( t_{e_1} t_{e_1}^* + t_{e_2} t_{e_2}^* \right) z^* &= y_{f_1} y_{f_1}^* + y_{f_2} y_{f_2}^* +y_{f_5} y_{f_5}^*, \\
	z \left( q_{v_1} - \left( t_{e_1} t_{e_1}^* + t_{e_2} t_{e_2}^* \right) \right) z^* &= r_{w_1} - \left( y_{f_1} y_{f_1}^* + y_{f_2} y_{f_2}^* +y_{f_5} y_{f_5}^* \right).
\end{align*}
Note that this implies that $z q_{v_1} z^* = r_{w_1}$.

We will now define a Cuntz-Krieger $E_{u,--}$-family in $\mathcal{O}_2$.
For $v \in E^0$, we let $P_v = q_v$, and we let $P_{w_i} = r_{w_i}$, for $i = 1,2,3,4$.
For $e \in E^1 \cup \{ d_1, d_2 \}$, we let $S_e = z t_e z^*$, and we let $S_{f_i} = y_{f_i}$ for $i=1,2,\ldots,10$.
Denote this family by $\mathcal{T}$.

By construction $\{ P_v \mid v \in E_{u,--}^0\}$ is a set of orthogonal projections, and $\{ S_e \mid e \in E_{u,--}^1\}$ a set of partial isometries.
Since $z$ is a unitary in $C^*(\mathcal{S})$ and since $\mathcal{S}$ is a Cuntz-Krieger $E_{u,-}$-family, $\mathcal{T}$ will satisfy the Cuntz-Krieger relations at all vertices in $E^0$.
Similarly, we see that since $\{ r_{w_i}, y_{f_j} \mid i=1,2,3,4, j = 1,2,\ldots, 10 \}$ is a Cuntz-Krieger $(\mathbf{E}_{**}, \{ w_2,w_3,w_4 \})$-family, $\mathcal{T}$ will satisfy the relations at the vertices $w_2, w_3, w_4$. 
It only remains to check the summation relation at $w_1$, for that we compute
\begin{align*}
	\smash{\sum_{s_{E_{u,--}}(e) = w_1} S_e S_e^*}	&= S_{f_1} S_{f_1}^* + S_{f_2} S_{f_2}^* + S_{f_5} S_{f_5}^* + S_{d_2} S_{d_2}^* \\
											&= y_{f_1} y_{f_1}^* + y_{f_2} y_{f_2}^* +y_{f_5} y_{f_5}^* + z t_{d_2} t_{d_2}^* z^* \\
											&= z \left( t_{e_1} t_{e_1}^* + t_{e_2} t_{e_2}^* \right) z^* + z t_{d_2} t_{d_2}^* z^* \\
											&= z \left( t_{e_1} t_{e_1}^* + t_{e_2} t_{e_2}^* + t_{d_2} t_{d_2}^* \right) z^* \\
											&= z q_{v_1} z^* = r_{w_1} = P_{w_1}.
\end{align*}
Hence $\mathcal{T}$ is a Cuntz-Krieger $E_{u,--}$-family. 

The universal property of $C^*(E_{u,--})$ provides a surjective $*$-homomorphism from $C^*(E_{u,--})$ to $C^*(\mathcal{T}) \subseteq \mathcal{O}_2$. 
Let $\alpha_1 \alpha_2 \cdots \alpha_n$ be a vertex-simple cycle in $E_{u,--}$ without any exit. 
We see that all the edges $\alpha_i$ must be in $E^1$, and hence we have
\begin{align*}
	S_{\alpha_1} S_{\alpha_2} \cdots S_{\alpha_n} = z t_{\alpha_1} z^* z t_{\alpha_2} z^* \cdots z t_{\alpha_n} z^* = z s_{\alpha_1} s_{\alpha_2} \cdots s_{\alpha_n} z^* 
\end{align*} 
and so its spectrum contain the entire unit circle. 
It now follows from Szyma{\'n}ski's General Cuntz-Krieger Uniqueness Theorem (\cite[Theorem~1.2]{MR1914564}) that $C^*(E_{u,--})$ is isomorphic to $C^*(\mathcal{T})$.

Since $\mathfrak{A} \subseteq C^*(\mathcal{S})$ and since $\{ r_{w_i}, y_{f_j} \mid i=1,2,3,4, j = 1,2,\ldots, 10 \} \subseteq \mathcal{E} \subseteq C^*(\mathcal{S})$, we have that $\mathcal{T} \subseteq C^*(\mathcal{S})$. 
So $C^*(\mathcal{T}) \subseteq C^*(\mathcal{S})$. 
But since $\mathfrak{A}$ is also contained in $C^*(\mathcal{T})$ and $\mathcal{E} \subseteq C^*(\mathcal{T})$, we have that $\mathcal{S} \subseteq C^*(\mathcal{T})$, and hence $C^*(\mathcal{S}) \subseteq C^*(\mathcal{T})$. 
Therefore 
\[
	C^*(E_{u,-}) \cong C^*(\mathcal{S}) = C^*(\mathcal{T}) \cong C^*(E_{u,--}).
\qedhere\]
\end{proof}

Thus we have the following fundamental result.

\begin{proposition}\label{prop:cuntzspliceinvariant}
Let $E$ be a graph with finitely many vertices, and let $v$ be a vertex that supports at least two distinct return paths. 
Then $C^*(E)\otimes\K\cong C^*(E_{v,-})\otimes\K$.
\end{proposition}

\begin{proof}
By Theorem~\ref{t:cuntz-splice-1}, $C^{*} ( E_{v, -} ) \otimes \K \cong C^{*} ( E_{v,- -} ) \otimes \K)$.  By Proposition~\ref{prop:cuntzsplicetwice} and Theorem~\ref{thm:moveimpliesstableisomorphism}, $C^{*} (E) \otimes \K \cong C^{*} ( E_{v, - - } ) \otimes \K$.  Thus, $C^*(E)\otimes\K\cong C^*(E_{v,-})\otimes\K$.
\end{proof}


\section{Notation needed for the proof}
\label{sec:notation}

\subsection{Block matrices and equivalences}

\begin{notation}
For $m,n\in\N_0$, we let $\MZ[m\times n]$ denote the set of group homomorphisms from $\Z^n$ to $\Z^m$. When $m,n\geq 1$, we can equivalently view this as the $m\times n$ matrices over $\Z$, where composition of group homomorphisms corresponds to matrix multiplication --- the (zero) group homomorphisms for $m=0$ or $n=0$ we will also call empty matrices with zero rows or columns, respectively. 

For $m,n\in\N$, we let \Mplus denote the subset of $\MZ[m\times n]$, where all entries in the corresponding matrix are positive. For a $m\times n$ matrix, we will also write $B>0$ whenever $B\in\Mplus$.

For a $m\times n$ matrix $B$, where $m,n\in\N$, we let $B(i,j)$ denote the $(i,j)$'th entry of the corresponding matrix, \ie, the entry in the $i$'th row and $j$'th column.\end{notation}

\begin{definition} 
Let $m,n\in\N$. 
For a $m\times n$ matrix $B$ over \Z, we let $\gcd B$ be the greatest common divisor of the entries $B(i,j)$, for $i=1,\ldots,m$, $j=1,\ldots,n$, if $B$ is nonzero, and zero otherwise. 
\end{definition}

\begin{assumption} \label{ass:preorder}
Let $N\in\N$. 
For the rest of the paper, we let $\calP=\{1,2,\ldots,N\}$ denote a partially ordered set with order $\preceq$ satisfying
$$i\preceq j\Rightarrow i\leq j,$$
for all $i,j\in\calP$, where $\leq$ denotes the usual order on \N. 
We denote the corresponding irreflexive order by $\prec$.
\end{assumption}

\begin{definition}
Let $\mathbf{m}=(m_i)_{i=1}^{N},\mathbf{n}=(n_i)_{i=1}^{N}\in\N_{0}^N$ be \emph{multiindices}. 
We write $\mathbf{m}\leq\mathbf{n}$ if $m_i\leq n_i$ for all $i=1,2,\ldots,N$, 
and in that case, we let $\mathbf{n}-\mathbf{m}$ be $(n_i-m_i)_{i=1}^N$. 

We let $\MZ$ denote the set of group homomorphisms from $\Z^{n_1}\oplus\Z^{n_2}\oplus\cdots\oplus\Z^{n_N}$ to $\Z^{m_1}\oplus\Z^{m_2}\oplus\cdots\oplus\Z^{m_N}$, and for such a homomorphism $B$, we let $B\{ i,j\}$ denote the component of $B$ from the $j$'th direct summand to the $i$'th direct summand. 
We also use the notation $B\{i\}$ for $B\{i,i\}$. 
Using composition of homomorphisms we get in a natural way a category $\mathfrak{M}_N$ with objects $\N_0^N$ and with the morphisms from $\mathbf{n}$ to $\mathbf{m}$ being $\MZ$. 
Moreover, 
$$(BC)\{ i,j\}=\sum_{k=1}^N B \{ i,k\} C\{ k,j\},$$
whenever $B\in\MZ$ and $C\in\MZ[\mathbf{n}\times\mathbf{r}]$ for a multiindex $\mathbf{r}$. 

A morphism $B\in\MZ$ is said to be in $\MPZ$, if 
$$B\{i,j\}\neq 0\Longrightarrow i\preceq j,$$
for all $i,j\in\calP$. 
It is easy to verify, that this gives a subcategory $\mathfrak{M}_\calP$ with the same objects but $\MPZ$ as morphisms. 

Moreover, for a subset $s$ of \calP, we let --- with a slight misuse of notation --- $B\{s\}\in\mathfrak{M}_s((m_i)_{i\in s}\times (n_i)_{i\in s},\Z)$ denote the component of $B$ from $\bigoplus_{i\in s}\Z^{n_i}$ to $\bigoplus_{i\in s}\Z^{m_i}$.

We let $\MZ[\mathbf{n}]$ denote $\MZ[\mathbf{n}\times\mathbf{n}]$, and $\MPZ[\mathbf{n}]$ denote $\MPZ[\mathbf{n}\times\mathbf{n}]$.

For $\mathbf{n}$, we let $\GLPZ$ denote the automorphisms in $\MPZ[\mathbf{n}]$. 
Then $U\in\GLPZ$ if and only if $U\in\MPZ[\mathbf{n}]$ and $U\{ i\}$ is a group automorphism (meaning that the determinant as a matrix is $\pm 1$ whenever $n_i\neq 0$, for every $i\in\calP$). 

An automorphism $U\in\GLPZ$ is in $\SLPZ$ if the determinant of $U\{ i\}$ is $1$ for all $i\in\calP$ with $n_i\neq 0$. 
\end{definition}

\begin{remark}
Let $\mathbf{m},\mathbf{n}\in\N_{0}^N$ be \emph{multiindices}. 
Set $k_1 = m_1 + \cdots + m_N$ and $k_2 = n_1 + \cdots + n_N$. 
If $k_1\neq 0$ and $k_2\neq 0$, we can equivalently view the elements $B\in\MZ$ as block matrices
$$B = 
\begin{pmatrix}
B\{1,1\} & \dots & B\{1,N\} \\
\vdots &  & \vdots \\
B\{N,1\} & \dots & B\{N,N\}
\end{pmatrix}$$
where $B\{i,j\} \in \MZ[m_i\times n_j]$ with $B\{i,j\}$ the empty matrix if $m_{i} = 0$ or $n_{j} = 0$.

Note that from this point of view, the matrices in \MPZ are upper triangular matrices with a certain zero block structure dictated by the order on \calP, and the matrices in \GLPZ (respectively \SLPZ) are matrices in \MPZ with all nonempty diagonal blocks having determinant $\pm 1$ (respectively $1$). 

Note that if $B\in\MZ$ and $C\in\MZ[\mathbf{n}\times\mathbf{r}]$ for a multiindex $\mathbf{r}$, 
then the matrix product makes sense, and --- as matrices --- we have that 
$$(BC)\{ i,j\}=\sum_{k\in\calP, n_k\neq 0}^N B \{ i,k\} C\{ k,j\},$$
for all $i,j\in\calP$ with $m_i\neq 0$ and $r_j\neq 0$. 

We will therefore also allow ourselves to talk about matrices with zero rows or columns (by considering it as an element of $\MZ[m\times n]$ ); and then $B\{s\}$ for a subset $s$ of \calP as defined above is just the principal submatrix corresponding to indices in $s$ (remembering the block structure). 
\end{remark}

\begin{definition}
Let $\mathbf{m}$ and $\mathbf{n}$ be multiindices. 
Two matrices $B$ and $B'$ in \MPZ are said to be \emph{\GLPE} (respectively \emph{\SLPE}) if there exist 
$U \in\GLPZ[\mathbf{m}]$ and $V \in\GLPZ$ (respectively $U \in\SLPZ[\mathbf{m}]$ and $V \in\SLPZ$) such that 
$$U B V = B'.$$
\end{definition}

Note that this is a generalization of the definitions in \cite{MR1907894,MR1990568} (in the finite matrix case) to the cases with rectangular diagonal blocks or vacuous blocks. 

\subsection{\texorpdfstring{$K$}{K}-web and induced isomorphisms}
We define the $K$-web, $K(B)$, of a matrix $B \in \MPZ$ and describe how a \GLPEe	 $\ftn{(U,V)}{B}{B'}$ induces an isomorphism $\ftn{ \kappa_{(U,V)} }{ K(B) }{K(B')}$.

For an element $B\in\MZ[m\times n]$ (\ie, a group homomorphism $\ftn{B}{\Z^n}{\Z^m}$), we define as usual $\cok B$ to be the abelian group $\Z^m/B\Z^n$ and $\ker B$ to be the abelian group $\setof{x\in\Z^n}{Bx=0}$. Note, that if $m=0$, then $\cok B=\{0\}$ and $\ker B=\Z^n$, and if $n=0$, then $\cok B=\Z^m$ and $\ker B=\{0\}$. 

For $m,n\in\N_0$, $B,B'\in\MZ[m\times n]$, $U\in\GLZ[m]$ and $V\in\GLZ$ with $UBV=B'$, it is now clear that this equivalence induces isomorphisms
$$
\xymatrix{\cok B \ar[rr]_{\xi_{(U,V)}}^{ [x] \mapsto [Ux] } & &\cok B'} \quad \text{and} \quad \xymatrix{\ker B \ar[rr]_{\delta_{(U,V)}}^{ [x] \mapsto [V^{-1}x] } & & \ker B'.}
$$

\begin{lemma}\label{lem: Kweb 2 components}
Let $\calP=\calP_2 = \{1,2\}$ be a partially ordered set and let $B\in\MPZ$.
Then the following sequence 
\[
\xymatrix@C=40pt{
\cok B\{1\} \ar[r]^-{[v] \mapsto \left[ \begin{smallpmatrix} v \\ 0 \end{smallpmatrix} \right] } & 
\cok B \ar[r]^-{\left[ \begin{smallpmatrix} v \\ w \end{smallpmatrix} \right] \mapsto [w] } & 
\cok B\{2\} \ar[d]^0 \\ 
\ker B\{ 2 \} \ar[u]^-{ v \mapsto [ A\{1,2\}v ] } & 
\ker B \ar[l]^-{ w \mapsfrom \begin{smallpmatrix}v \\ w\end{smallpmatrix} } & 
\ker B\{ 1 \} \ar[l]^-{ \begin{smallpmatrix}v \\ 0\end{smallpmatrix}\mapsfrom v  } }
\]
is exact.

Moreover, if $B$ and $B'$ are elements of \MPZ and $\ftn{ (U,V) }{ B }{ B' }$ is a \GLPEe, then $(U,V)$ induces an isomorphism 
$$(\xi_{(U\{1\},V\{1\})},\xi_{(U,V)},\xi_{(U\{2\},V\{2\})},\delta_{(U\{1\},V\{1\})},\delta_{(U,V)},\delta_{(U\{2\},V\{2\})})$$
of (cyclic six-term) exact sequences. 
\end{lemma}

\begin{proof}
The first part of the lemma follows directly from the Snake lemma applied to the diagram
$$\xymatrix{ 
0\ar[r] & \Z^{n_1}\ar[r]\ar[d]^{B\{1\}} & \Z^{n_1}\oplus\Z^{n_2}\ar[r]\ar[d]^{B} & \Z^{n_2}\ar[d]^{B\{2\} }\ar[r] & 0 \\ 
0\ar[r] & \Z^{m_1}\ar[r] & \Z^{m_1}\oplus\Z^{m_2}\ar[r] & \Z^{m_2}\ar[r] & 0
}$$
The second part of the proof is a straightforward verification.
\end{proof}

Completely analogous to \cite{MR1990568}, we make the following definitions.

\begin{definition}
A subset $c$ of $\calP$ is called \emph{convex} if $c$ is nonempty and for all $k \in \calP$, 
$$\text{$\{i,j\} \subseteq c$ and $i \preceq k \preceq j \ \implies \ k \in c$.}$$

A subset $d$ of $\calP$ is called a \emph{difference set} if $d$ is convex and there are convex sets $r$ and $s$ in $\calP$ with $r \subseteq s$ such that $d =s \setminus r$ and 
\[
\text{$i \in r$ and $j \in d \ \implies \ j \npreceq i$.}
\]
Whenever we have such set $r$, $s$ and $d=s\setminus r$, we get a canonical functor from $\mathfrak{M}_\calP$ to $\mathfrak{M}_{\calP_2}$, where $\calP_2=\{1,2\}$ with the usual order if there exist $i\in r$ and $j\in d$ such that $i\preceq j$, and the trivial order otherwise. 
Thus such sets will also give a canonical (cyclic six-term) exact sequence as above. 
\end{definition}

\begin{definition}
Let $B \in \MPZ$. 
The \emph{(reduced) $K$-web} of $B$, $K(B)$, consists of a family of abelian groups together with families of group homomorphisms between these, as described below. 

For each $i \in \calP$, let $r_i = \setof{j\in\calP}{j \prec i}$ and $s_{i} = \setof{ j\in\calP }{ j \preceq i }$. 
Note that if $r_{i}$ in the above definition is nonempty, then $\{ i \} = s_{i} \setminus r_{i}$ is a difference set.  
We let $\mathrm{Imm}(i)$ denote the set of immediate predecessors of $i$ (we say that $j$ is an \emph{immediate predecessor of $i$} if $j \prec i$ and there is no $k$ such that $j \prec k \prec i$).  

For each $i \in \calP$ with $r_i\neq \emptyset$, we get an exact sequence from Lemma~\ref{lem: Kweb 2 components},
\begin{equation}\label{eq:exact-seq-Kweb}
\ker B\{i\}\rightarrow \cok B\{r_i\}\rightarrow \cok B\{s_i\}\rightarrow \cok B\{i\}
\end{equation}
Moreover, for every pair $(i, j ) \in \calP \times \calP$ satisfying $j \in \mathrm{Imm}(i)$ and $\mathrm{Imm}(i) \setminus \{ j \} \neq \emptyset$ is $s_{j} \subsetneq r_{i}$; consequently we have a homomorphism 
\begin{equation}\label{eq: Kweb hom}
\cok B\{s_j\} \to \cok B\{r_i\} 
\end{equation}
originating from the exact sequence above (\cf\ Lemma~\ref{lem: Kweb 2 components} used on the division into the sets $r_i$, $s_j$ and $r_i\setminus s_j$). 

Set 
\begin{align*}
I_0^{\calP} &= 
\setof{ r_i }{ i \in \calP \text{ and }r_{i} \neq \emptyset } \cup \setof{ s_i }{ i \in \calP } \cup \setof{ \{ i \} }{ i \in \calP }, \\
I_1^{\calP} &= \setof{i\in\calP}{r_i\neq \emptyset}.
\end{align*}

The \emph{$K$-web of $B$}, denoted by $K(B)$, consists of the families $\left( \cok B\{c\} \right)_{ c \in I_0^{\calP}}$ and $\left( \ker B\{i\} \right)_{ i \in I_1^\calP}$ together with all the homomorphisms from the sequences \eqref{eq:exact-seq-Kweb} and \eqref{eq: Kweb hom}.  Let $B'$ be an element of \MPZ[\mathbf{m}'\times\mathbf{n}'].  By a \emph{$K$-web isomorphism}, $\ftn{\kappa}{ K(A) }{ K(B) }$, we mean families 
$$\left( \ftn{ \phi_{c} }{ \cok B\{c\} }{ \cok B'\{c\} } \right)_{ c \in I_0^\calP}$$ 
and $$\left( \ftn{ \psi_i }{ \ker B\{i\} }{ \ker B'\{i\} } \right)_{ i \in I_1^\calP }$$ of isomorphisms satisfying that the ladders coming from the sequences in $K(B)$ and $K(B')$ commute.

By Lemma~\ref{lem: Kweb 2 components}, any \GLPEe $\ftn{ (U,V) }{ B }{ B' }$ induces a $K$-web isomorphism from $B$ to $B'$.  We denote this induced isomorphism by $\kappa_{(U,V)}$. 
\end{definition}

\begin{remark}
The definitions above are completely analogous to the definitions in \cite{MR1990568}, and are the same in the case $m_{i} = n_{i} \neq 0$ for all $i\in\calP$. 
Note that the last homomorphism in \eqref{eq:exact-seq-Kweb} is really not needed, because commutativity with this map is automatic.
\end{remark}


\section{Standard form}
\label{sec:standardform}

In this section, we prove that every graph with finitely many vertices is move equivalent to a graph in canonical form (see Definition~\ref{def: canonical form}).
This will allow us to reduce the proof of our classification result to graphs in canonical form.  
In fact, we will do even better.
We will reduce the proof of our classification result to graphs whose adjacency matrices are in the same block form.
  
The first result of this type is the following that allows us to remove breaking vertices (see \cite{MR1988256} for a definition) and regular vertices that do not support a loop. 

\begin{lemma}\label{lem:  removing breaking vertices}
Let $E$ be a graph with finitely many vertices.
Then $E \sim_{M} E'$, where $E'$ is a graph with finitely many vertices such that every vertex of $E'$ is either a regular vertex that is the base point of a loop or a singular vertex $v$ satisfying the property that if there exists a path of positive length from $v$ to $w$, then $| s^{-1}(v) \cap r^{-1} (w) | = \infty$. 
\end{lemma}
\begin{proof}
First we show how to modify $E$ to get a graph with the property that if $v$ is an infinite emitter, then $v$ emits infinitely many edges to any vertex it emits any edges to. 
Let $v \in E^0$ be an infinite emitter. 
If there exists a vertex $u \in E^0$ such that $v$ emits only finitely many edges to $u$, we partition $s^{-1}(v)$ into two sets, $\mathcal{E}_1 = \{ e \in s^{-1}(v) \mid |s^{-1}(v) \cap r^{-1}(r(e))| < \infty  \}$ and $\mathcal{E}_2 = \{ e \in s^{-1}(v) \mid |s^{-1}(v) \cap r^{-1}(r(e))| = \infty  \}$, \ie $\mathcal{E}_1$ consists of the edges out of $v$ that only have finitely many parallel edges. 
Note that since $E^0$ is finite, $\mathcal{E}_1$ is a finite set. 
Hence we can perform move \OO according to this partition, resulting in a graph $F'$ that is move equivalent to $E$. 
Call the vertices $v$ got split into $v_1$ and $v_2$. 
In $F'$, $v_2$ is an infinite emitter with the property that it emits infinitely many edges to any vertex it emits any edges to, and any infinite emitter in $E$ that already had that property keeps it. 
On the other hand $v_1$ is a finite emitter. 

Since $E^0$ is finite, we can do the above process a finite number of times, ending with a graph $F$ that is move equivalent to $E$, and with the property that if $v$ is an infinite emitter, then $v$ emits infinitely many edges to any vertex it emits any edges to. 
Now we can use move \TT a finite number of times to get a graph $G$ that is move equivalent to $F$ and satisfies that for every infinite emitter $v \in G^0$ and every $w \in G^0$ for which there exists a path of positive length from $v$ to $w$ we have $| s^{-1}(v) \cap r^{-1} (w) | = \infty$.  
Finally we use the collapse move (Definition \ref{def:collapse}) on each regular vertex of $F$ that does not support a loop to produce a new graph, $E'$ say, with $E' \Meq G \Meq E$ and such that every regular vertex in $E'$ supports a loop. 
Because of the way the collapse move adds edges this process maintains the property that $| s^{-1}(v) \cap r^{-1} (w) | = \infty$ for any infinite emitter $v$ and any vertex $w$ with a path of positive length from $v$ to $w$.
\end{proof}

Assume $E$ satisfies Condition~(K) and satisfies the conclusion of Lemma~\ref{lem:  removing breaking vertices}.  Then every hereditary subset of $E^{0}$ is saturated and $E$ has no breaking vertices.  Moreover, every ideal of $C^{*} (E)$ is gauge invariant.  In particular, there is a lattice isomorphism from the ideal lattice of $C^{*} (E)$ to the lattice of hereditary subsets of $E^{0}$ with ordering given by set containment. 

Therefore, $\Bsf_{E}^{\bullet} \in \MPZ[\mathbf{m}_{E} \times \mathbf{n}_{E}]$ (in a canonical way) for a partially ordered set $\calP = ( \{1, \dots, N \} , \preceq )$, where $N$ is the number of points in $\mathrm{Prim}(C^*(E))$, and ``$\preceq$'' is chosen so that it satisfies Assumption \ref{ass:preorder}. 
More formally we have:

\begin{lemma}\label{lem: block structure of adjacency matrix}
Let $E$ be a graph with finitely many vertices such that every vertex of $E$ is either a regular vertex that is the base point of a loop or a singular vertex $v$ satisfying the property that if there exists a path of positive length from $v$ to $w$, then $| s^{-1}(v) \cap r^{-1} (w) | = \infty$.  Suppose $E$ satisfies Condition~(K).  Then $\Bsf_{E}^{\bullet} \in \MPZ[\mathbf{m}_{E} \times \mathbf{n}_{E}]$, where 
\[
n_{E, i} = | H_{i,1}^{E} \setminus H_{i,0}^{E} | 
\]
with $I_{H_{i,1}^{E}}$ the prime ideal corresponding to $i$ and $I_{H_{i,0}^{E}}$ the maximal proper ideal of $I_{H_{i,1}^{E}}$, and 
\[
m_{E, i} = n_{ E, i } - | \setof{ v \in H_{i,1}^{E} \setminus H_{i,0}^{E} }{ \text{$v$ is a singular vertex in $H_{i,1}^{E} \setminus H_{i,0}^{E}$}} |. 
\]
\end{lemma}
Note that the hereditary subsets of vertices --- as usually defined for graphs, when we consider graph \cas{} --- correspond to subsets $S$ of $\calP$ satisfying that $i\preceq j$ implies that $j\in S$ whenever $i\in S$. 
This is due to that fact that we generally do not work with the transposed matrix in this paper, since we find it more convenient to work with the non-transposed matrix (see also the proof of Theorem~\ref{thm:putting-it-all-together}). 

We now expand on the conditions we can put on graphs. 
To turn $K$-theory isomorphisms into \GLPEe{s} or \SLPEe{s}, the matrices $\Bsf^{\bullet}_E$ and $\Bsf^{\bullet}_{E'}$ must have sufficiently big diagonal blocks, this requirement is captured in (\ref{thm:canonical-item-size3}) and (\ref{thm:canonical-item-rowrank}) below. 
The positivity condition, (\ref{thm:canonical-item-positive}), is also critical when dealing with matrix manipulations. 
Condition (\ref{thm:canonical-item-paths}) ensures that we can apply Propositions \ref{prop:columnAdd} and \ref{prop:rowAdd} to actually do matrix manipulations.

\begin{theorem}\label{thm: canonical form}
Let $E$ be a graph with finitely many vertices that satisfies Condition~(K). 
Then there exists a graph $E'$ with finitely many vertices such that $E \sim_{M} E'$ and $E'$ satisfies the following properties:
\begin{enumerate}
\item \label{thm:canonical-item-loops-inf} every vertex of $E'$ is either a regular vertex that is the base point of a loop or a singular vertex $v$ satisfying the property that if there exists a path of positive length from $v$ to $w$, then $| s^{-1}(v) \cap r^{-1} (w) | = \infty$;
\item \label{thm:canonical-item-paths} for all regular vertices $v, w$ of $E'$ with $v \geq w$, there exists a path in $E'$ from $v$ to $w$ through regular vertices in $E$;
\item \label{thm:canonical-item-size3} $m_{E',i} \geq 3$ whenever there exists a cycle in the graph
\[
\left( H_{i,1}^{E'} \setminus H_{i, 0}^{E'} , r^{-1}( H_{i,1}^{E'} \setminus H_{i, 0}^{E'} ) \cap s^{-1} ( H_{i,1}^{E'} ), r, s \right)\text{;} 
\]
\item \label{thm:canonical-item-positive} if $i \preceq j$ and $\Bsf^{\bullet}_{E'} \{ i, j \}$ is not the empty matrix, then $\Bsf^{\bullet}_{E'} \{ i , j \} > 0$; and
\item \label{thm:canonical-item-rowrank}  if $\Bsf^{\bullet}_{E'} \{ i \}$ is not the empty matrix, then the Smith normal form  of $\Bsf^{\bullet}_{E'} \{ i \}$ has at least two 1's.
\end{enumerate}
\end{theorem}
\begin{proof}
Lemma \ref{lem:  removing breaking vertices} lets us find a graph $F$ such that $F \Meq E$ and $F$ satisfies (\ref{thm:canonical-item-loops-inf}). 
Using the same technique as described in the proof Proposition \ref{prop:cuntzSpliceSetup}, we can guarantee that $F$ also satisfies (\ref{thm:canonical-item-paths}). 

Suppose now $i$ is such that $m_{F,i} < 3$ and there exist a cycle in 
\[
  \left( H_{i,1}^{F} \setminus H_{i, 0}^{F} , r^{-1}( H_{i,1}^{F} \setminus H_{i, 0}^{F} ) \cap s^{-1} ( H_{i,1}^{F} ), r, s \right),
\]
We want to reduce to the case where $m_{F,i} = 2$. 

If $m_{F,i} = 0$ then all the vertices in $H_{i,1}^{F} \setminus H_{i, 0}^{F}$ are infinite emitters.
Since the subgraph has a cycle and $F$ satisfies (\ref{thm:canonical-item-loops-inf}) each of the vertices in $H_{i,1}^{F} \setminus H_{i, 0}^{F}$ supports an infinite number of loops. 
By using move \OO to split two loops of an infinite emitter, we get a graph $F'$ that is move equivalent to $F$, satisfies (\ref{thm:canonical-item-loops-inf}) and  (\ref{thm:canonical-item-paths}) and where $m_{F', i} > 0$.
Hence we may assume that $1 \leq m_{F,i} < 3$.

If $m_{F,i} = 1$ there are two cases. 
Case one is that $H_{i,1}^{F} \setminus H_{i, 0}^{F}$ only consists of one vertex. 
In this case, that vertex must support at least two loops (since $F$ satisfies Condition~(K)) and we can use move \OO to split the vertex in to two, thus giving us a move equivalent graph, $F'$, that satisfies (\ref{thm:canonical-item-loops-inf}) and  (\ref{thm:canonical-item-paths}) and where $m_{F', i} = 2$.
The other case is that $H_{i,1}^{F} \setminus H_{i, 0}^{F}$  also contains an infinite emitter. 
The regular vertex $v$ has to emit at least one edge to one such infinite emitter $w$. 
By the construction of $H_{i,1}^{F} \setminus H_{i, 0}^{F}$ $w$ must emit an edge to $v$, and therefore we can use column addition (Proposition \ref{prop:columnAdd}) to add the $w$'th column of $\Bsf_F$ into the $v$'th column. 
The result will be a graph $F'$ that satisfies (\ref{thm:canonical-item-loops-inf}) and  (\ref{thm:canonical-item-paths}) and where $v$ supports at least two loops. 
Outsplitting, as in case one, we reduce to the case where $m_{F,i} = 2$.

Suppose now that $m_{F,i} = 2$. 
Then there are two regular vertices $u,v \in H_{i,1}^{F} \setminus H_{i, 0}^{F}$ and there is at least one edge from $u$ to $v$ and at least one from $v$ to $u$. 
Hence we can add the $v$'th column of $\Bsf_F$ into the $u$'th, using Proposition \ref{prop:columnAdd}, to ensure that $u$ supports at least $2$ loops.
We can now use move \OO to outsplit $u$, by dividing the outgoing edges into two in such a way that each partition has a loop, to yield a graph move equivalent to $F$ that satisfies (\ref{thm:canonical-item-loops-inf}), (\ref{thm:canonical-item-paths}) and (\ref{thm:canonical-item-size3}). 
Hence we can assume that $F$ also satisfies (\ref{thm:canonical-item-size3}).

By (\ref{thm:canonical-item-size3}) each nonempty diagonal block of $\Bsf_{F}^{\bullet}$ will have a nonzero entry.
Hence we may use row and column additions (Propositions \ref{prop:rowAdd} and \ref{prop:columnAdd}), which are legal because of (\ref{thm:canonical-item-paths}), to make sure that all entries in the diagonal blocks are nonzero. 
Then we can use column addition to guarantee that all offdiagonal blocks (that are not forced to be zero by the block structure) are strictly positive. 
Since adding rows and columns together will keep conditions (\ref{thm:canonical-item-loops-inf}), (\ref{thm:canonical-item-paths}) and (\ref{thm:canonical-item-size3}), we may assume that $F$ also satisfies (\ref{thm:canonical-item-positive}).

Note, that by the above reasoning, we can assume that any entry in $\Bsf^{\bullet}_E$ is not only positive, but greater than or equal to  any  natural number we see fit. 
Hence, for each $i$ with $\Bsf^{\bullet}_F \{i\}$ nonempty, we can find a regular vertex, $v$ say, in $H_{i,1}^{F}$ such that $v$ emits at least $4$ edges to each vertex $v$ reaches. 
Partition the outgoing edges of $v$ into two sets in such a way that each partition contains at least one edges to each vertex $v$ can reach, and at least two loops.  Let $d_1$ be the number of loops in the first partition, and let $d_2$ be the number in the second (then $d_1 + d_2 = d$). 
Outsplitting according to this partition will yield a graph $F'$ such that $F \Meq F'$ and $F$ satisfies (\ref{thm:canonical-item-loops-inf}), (\ref{thm:canonical-item-paths}), (\ref{thm:canonical-item-size3}) and (\ref{thm:canonical-item-positive}). $\Bsf^{\bullet}_{F'} \{i\}$ will contain the following two rows (corresponding to the vertices $v$ was split into)
\[
\begin{pmatrix}
 d_1 - 1 & d_1 & * & * & \cdots \\
 d_2  & d_2 -1 & * & * & \cdots
\end{pmatrix},
\]
where the asterisks can be any positive numbers, with $m_{F',i} = m_{F,i} +1$. Repeating this process we can increase the size of the relevant block so much that the Smith normal form must contain at least two 1's.

Continuing in this fashion for each diagonal block we can construct $E'$ such that $E \Meq E'$ and $E'$ satisfies (\ref{thm:canonical-item-loops-inf}), (\ref{thm:canonical-item-paths}), (\ref{thm:canonical-item-size3}), (\ref{thm:canonical-item-positive}) and (\ref{thm:canonical-item-rowrank}). 
\end{proof}

\begin{remark}
Suppose that $E$ is a graph with finitely many vertices that satisfies (\ref{thm:canonical-item-loops-inf}) and (\ref{thm:canonical-item-size3}) of Theorem~\ref{thm: canonical form}.  Then $E$ satisfies Condition~(K). 
Moreover, if $H_{i,1}^{E} \setminus H_{i, 0}^{E} = \{ v_{i} \}$, then either $v_{i}$ is an infinite emitter that does not support a cycle or is a sink. 
\end{remark}

\begin{remark} \label{rmk:  canonical form and SL equivalence}
Let $E$ be a graph with finitely many vertices that satisfies Condition~(K).
It follows from the proof of Theorem \ref{thm: canonical form} that if $E$ satisfies (\ref{thm:canonical-item-loops-inf}), (\ref{thm:canonical-item-paths}) and (\ref{thm:canonical-item-size3}) from the theorem, then there exists a graph $E'$ that is move equivalent to $E$ and satisfies (\ref{thm:canonical-item-loops-inf}), (\ref{thm:canonical-item-paths}),  (\ref{thm:canonical-item-size3}) and (\ref{thm:canonical-item-positive}), furthermore $\Bsf^{\bullet}_E$ and $\Bsf^{\bullet}_{E'}$ have the same block form and are \SLPE.

Since the Smith normal form of a matrix is invariant under \SL-equivalence $E'$ will satisfy condition (\ref{thm:canonical-item-rowrank}) if $E$ does. 
\end{remark}

\begin{definition}\label{def:  canonical form}
A graph $E$ with finitely many vertices is in \emph{canonical form} if $E$ satisfies the properties (\ref{thm:canonical-item-loops-inf}), (\ref{thm:canonical-item-paths}),  (\ref{thm:canonical-item-size3}), (\ref{thm:canonical-item-positive}), and (\ref{thm:canonical-item-rowrank}) of Theorem~\ref{thm: canonical form}. 
A pair of graphs $(E,F)$ with finitely many vertices are in \emph{standard form} if $E$ and $F$ are in canonical form with $\mathbf{m}_{E} = \mathbf{m}_{F}$, and $\mathbf{n}_{E} = \mathbf{n}_{F}$.
\end{definition}

The notion of a standard form is of course only useful if we can assume that our graphs have the standard form, the next proposition shows that we can indeed assume that, if the corresponding \cas have isomorphic ordered reduced filtered $K$-theory. 

\begin{proposition}\label{prop:  standard form}
Let there be given graphs $E_{1}$ and $E_{2}$ with finitely many vertices. 
If $\FKRplus ( C^{*} (E_{1}) ) \cong \FKRplus ( C^{*} (E_{2}) )$, then there exists a pair of graphs $( F_{1} , F_{2} )$ with finitely many vertices such that the pair $( F_{1} , F_{2} )$ is in standard form and $E_{i} \sim_{M} F_{i}$.
\end{proposition}
\begin{proof}
It follows from Theorem \ref{thm: canonical form} that we can find graphs $G_1, G_2$ such that $G_i \Meq E_i$ and $G_i$ are in canonical form, $i = 1,2$. 
The $K$-theory condition gives a specific isomorphism between the primitive ideal spaces of $C^{*} (E_{1})$ and $C^{*} (E_{2})$, hence $\Bsf^{\bullet}_{G_1}$ and $\Bsf^{\bullet}_{G_2}$ can be chosen to have the same same block structure according to this isomorphism. 
Furthermore, the number of singular vertices in $H_{i,1}^{E_1} \setminus H_{i,0}^{E_1}$ is determined by its $K$-theory, since $C^*(H_{i,1}^{E_1} \setminus H_{i,0}^{E_1})$ is simple (see \cite[Lemma 9.2]{MR3082546}). 
The same holds for $E_2$ so $n_{E_1,i} - m_{E_1,i} = n_{E_2,i} - m_{E_2,i}$ for all $i$, and therefore $n_{G_1,i} - m_{G_1,i} = n_{G_2,i} - m_{G_2,i}$ for all $i$.

The only potential problem is now that the we may not have $n_{G_1,i} = n_{G_2, i}$ for all $i$. 
Since all the entries in $\Bsf^{\bullet}_{G_1}$ are positive, unless forced to be zero by the block structure, we may use row and column additions to ensure that all nonzero entries in $\Bsf^{\bullet}_{G_1}$ are at least $4$.
Similarly we can assume that all nonzero entries of $\Bsf^{\bullet}_{G_2}$ are at least $4$. 
So if $n_{G_1,i} < n_{G_2,i}$ for some $i$, we can use an outsplit (similar to what is described at the end of the proof of Theorem \ref{thm: canonical form}) to grow $n_{G_1,i}$ by $1$ while keeping it in canonical form. 
Proceeding this way, we construct graphs $F_1, F_2$ in canonical form such that $F_1 \Meq E_1$, $F_2 \Meq E_2$ and $n_{F_1,i} = n_{F_2,i}$ for all $i$.
Since we also have $n_{F_1,i} - m_{F_1,i} = n_{F_2,i} - m_{F_2,i}$ we must have $m_{F_1,i} = m_{F_2,i}$ for all $i$. 
\end{proof}

When $E$ is in canonical form, the rows of $\Bsf_E$ that are removed to form $\Bsf^{\bullet}_E$ either have all entries equal to $0$ except on which is $-1$, this is in case the corresponding vertex is a sink, or it only contains $0$ and $\infty$, in case it is an infinite emitter. 
It therefore follows from Proposition \ref{prop:columnAdd} and Remark \ref{rmk:columnAdd} that adding one column in $\Bsf^{\bullet}_E$ into another will preserve move equivalence, so long as it maintains the block structure and similarly for rows by Proposition \ref{prop:rowAdd} and Remark \ref{rmk:rowAdd}. 
Hence we have:

\begin{corollary} \label{cor:rc-in-bullet}
Let $E$ be a graph with finitely many vertices and suppose that $E$ is in canonical form. 
In $\Bsf^{\bullet}_E$ we can add column $l$ into column $k$ without changing the move equivalence class of the associated graph if the diagonal entry of column $l$ is in block $i$, the diagonal entry of column $k$ is in block $j$ and $i \preceq j$. 
Similarly can add row $l$ into row $k$ without changing move equivalence class if the diagonal entry of row $l$ is in block $i$, the diagonal entry of row $k$ is in block $j$ and $j \preceq i$. 
\end{corollary}

For the results in Section \ref{sec:boylehuang} we need a final refinement of our standard form, we also need the diagonal blocks of $\Bsf^{\bullet}_E$ to have has greatest common divisor $1$.
We achieve this in the next proposition by making sure that each diagonal block has an entry that is equal to $1$. 

\begin{proposition}\label{prop:  standard form plus gcd}
Let there be given graphs $E_{1}$ and $E_{2}$ with finitely many vertices. 
If $\FKRplus( C^{*} (E_{1}) ) \cong \FKRplus ( C^{*} (E_{2}) )$, then there exists a pair of graphs $( F_{1} , F_{2} )$ with finitely many vertices such that the pair $( F_{1} , F_{2} )$ is in standard form, $E_{i} \sim_{M} F_{i}$ and each nonempty diagonal block $\Bsf^{\bullet}_{F_i}$ contains a $1$.
\end{proposition}
\begin{proof}
By Proposition \ref{prop:  standard form} we can find $F_1, F_2$ satisfying the conclusion of the proposition, except for the last condition. 
As in the proof of Proposition \ref{prop:  standard form} we may use row and column operations to ensure that all nonzero entries of $\Bsf^{\bullet}_{G_1}$ and $\Bsf^{\bullet}_{G_2}$ are at least $4$. 
For each nonzero diagonal block, we will now do an outsplit similar to what is described at the end of the proof of Theorem \ref{thm: canonical form}, where we find a regular vertex, $v$ say, that supports at least two loops. 
Partition the outgoing edges of $v$ into two sets in such a way that each partition contains at least one edges to each vertex $v$ can reach, but we also insist that one partition only contains two loops. 
Let $d_1$ be the number of loops in the first partition, and let $d_2$ be the number in the second (then $d_1 + d_2 = d$).
Our assumption forces either $d_1$ or $d_2$ to be $2$, for simplicity let us say that $d_1 = 2$.
As noted in the proof of Theorem \ref{thm: canonical form} in the resulting graph, the diagonal block will contain the rows 
\[
\begin{pmatrix}
 d_1 - 1 & d_1 & * & * & \cdots \\
 d_2  & d_2 -1 & * & * & \cdots
\end{pmatrix}
= 
\begin{pmatrix}
 1 & 2 & * & * & \cdots \\
 d_2  & d_2 -1 & * & * & \cdots
\end{pmatrix}.
\]
Hence it contains a $1$. 
Doing this for all nonzero diagonal blocks yields the desired graphs. 
\end{proof}


\section{Generalization of Boyle-Huang's lifting result}
\label{sec:boylehuang}

We aim to prove Theorem~\ref{thm:mainBH} which says that --- in certain cases --- every $K$-web isomorphism is induced by a \GLPEe. This is the main result of this section. 
To prove Theorem~\ref{thm:mainBH}, we first strengthen \cite[Theorem~4.5]{MR1990568}.
The following theorem is a classical well-known theorem, \cf\ \cite[Section~II.15]{MR0340283}. 
\begin{theorem}[Smith normal form]
\label{thm-smith-normal-form}
Suppose $B$ is an $m\times n$ matrix over \Z. 
Then there exist matrices $U\in\GLZ[m]$ and $V\in\GLZ$ such that the matrix 
$D=UBV$ satisfies the following
\begin{itemize}
\item $D(i,j)=0$ for all $i\neq j$,
\item the $\min(m,n)\times\min(m,n)$ principal submatrix of $D$ 
is a diagonal matrix $$\operatorname{diag}(d_1,d_2,\ldots,d_r,0,0,\ldots,0),$$
where $r\in\{0,1,\ldots,\min(m,n)\}$ is the rank of $B$ and $d_1,d_2,\ldots,d_r$ are positive integers such that $d_i|d_{i+1}$ for $i=1,\ldots,r-1$. 
\end{itemize}
For each matrix $B$, the matrix $D$ is unique and is called the \emph{Smith normal form} of $B$. 
\end{theorem}

We now recall some terminology that was introduced in \cite{MR1990568}.
\begin{definition}
Let $B$ be an element of $\MZ[m\times n]$.  A \emph{\GL self-equivalence} of $B$ is a \GL-equivalence $\ftn{ (U,V) }{B}{B}$.
We say that an automorphism $\phi$ of $\cok B$ is \emph{\GL-allowable} if there exists a \GL self-equivalence, $(U,V)$, of $B$ such that the isomorphism $\kappa_{(U,V)}$ induces $\phi$.
\end{definition}

\begin{lemma}\label{lem-gcd}
Let $B$ be an $m\times n$ matrix over \Z, and let $U\in\GLZ[m]$ and $V\in\GLZ[n]$ be given invertible matrices. 
Then $\gcd B=\gcd (UBV)$. In particular, if $D$ is the Smith normal form of $B$, then $\gcd B=D(1,1)=d_1$. 
\end{lemma}
\begin{proof}
We may assume that $B\neq 0$. 
Let $d$ be a positive integer. Then
\begin{align*}
d\text{ divides all entries of }B & \Leftrightarrow \forall i,j\colon d\,|\,e_i^TBe_j \\
& \Leftrightarrow \forall x\in\Z^m,y\in\Z^n\colon d\,|\,x^TBy \\
& \Leftrightarrow \forall x\in\Z^m,y\in\Z^n\colon d\,|\,x^TUBVy \\
& \Leftrightarrow \forall i,j\colon d\,|\,e_i^TUBVe_j \\
& \Leftrightarrow d\text{ divides all entries of }UBV.
\end{align*}
Now the lemma follows.
\end{proof}

\begin{remark}
Let $B$ be an $m\times n$ matrix over \Z. 
Then it follows from the above, that $m$ is greater than the number of generators of $\cok B$ according to the decomposition from the Smith normal form into direct sums of nonzero cyclic groups if and only if $\gcd B=1$. 
\end{remark}

Boyle and Huang show in their paper \cite{MR1990568} the following fundamental theorem. 

\begin{theorem}[{\cite[Theorem~4.4]{MR1990568}}] \label{thm-BH-4.4}
Let $B$ be a $n\times n$ (square) matrix over a PID $\mathcal{R}$, and let $\delta=\gcd B$. 
Let $\phi$ be an automorphism of $\cok B$, and let $M$ be any $n\times n$ matrix over $\mathcal{R}$ defining $\phi$, \ie, $\phi([x])=[Mx]$ for all $x\in\Z^n$. 

Then $\det(M)\equiv 1 \pmod \delta$ if and only if there exist $n\times n$ matrices $U$ and $V$ over $\mathcal{R}$ with determinants $1$ such that $UBV=B$ and $U$ is defining $\phi$.

Then $\det(M)\equiv u \pmod \delta$ for some unit $u$ in $\mathcal{R}$ if and only if there exist $n\times n$ invertible (GL) matrices $U$ and $V$ over $\mathcal{R}$ such that $UBV=B$ and $U$ is defining $\phi$.
\end{theorem}

\begin{remark}
As we will see, it is possible to generalize the part about \GL-allowance in this theorem to rectangular matrices, the analogous statement to the part about \SL-allowance in  Theorem \ref{thm-BH-4.4} does not hold in general (for rectangular matrices). 
If we consider the matrix
$$B=\begin{pmatrix}
3 \\ 0
\end{pmatrix}$$ 
and the automorphism $-\id$ on $\cok B\cong \Z/3\oplus\Z$ induced by the matrix 
$$M=\begin{pmatrix}
-1 & 0 \\ 0 & -1
\end{pmatrix}$$
it is easy to see that we get a counterexample. 

Although it can be done, we do not investigate this further, since for our purposes we do not need to know when automorphisms can be lifted to \SL-equivalences. 
\end{remark}

In \cite{MR1990568} there is the following useful theorem. 
Note that all $n_i$'s are assumed to be nonzero in \cite{MR1990568}. 

\begin{theorem}[{\cite[Theorem~4.5]{MR1990568}}] \label{thm-BH-4.5}
Suppose $B$ and $B'$ are matrices in $\MPZ[\mathbf{n}]$ with corresponding diagonal
blocks equal, and $\kappa\colon K(B)\rightarrow K(B')$ is a $K$-web isomorphism. 
Then there exist matrices $U,V\in\GLZ[\mathbf{n}]$ such that we have a \GLPEe $(U, V )\colon B\rightarrow B'$ satisfying $\kappa_{(U,V )} = \kappa$ if and only if each of the
automorphisms $d_i\colon \cok B\{i\}\rightarrow\cok B'\{i\}$ defined by $\kappa$ is \GL-allowable.
\end{theorem}

Together with \cite[Theorem~4.4]{MR1990568} (see Theorem~\ref{thm-BH-4.4} above), this gives us the following useful corollary. 

\begin{corollary}[{\cite[Corollary~4.7]{MR1990568}}] \label{cor-BH-4.6}
Let $B$ and $B'$ be matrices in $\MPZ[\mathbf{n}]$ with $\gcd B\{i\}=1=\gcd B'\{i\}$ for all $i\in\calP$. Then for any $K$-web isomorphism $\kappa\colon K(B)\rightarrow K(B')$ there exist matrices $U,V\in\GLZ[\mathbf{n}]$ such that we have a \GLPEe $(U, V )\colon B\rightarrow B'$ satisfying $\kappa_{(U,V )} = \kappa$.
\end{corollary}

But even more is true.  
We can generalize \cite[Theorem~4.4 (and Proposition~4.1)]{MR1990568} (\cf\ Theorem~\ref{thm-BH-4.4}) as follows (we here only consider the case $\mathcal{R}=\Z$). 

\begin{theorem}\label{help}
Let $B$ be a $n\times n$ (square) matrix over \Z, and let $\delta=\gcd B$. 
Let $\phi$ be an automorphism of $\cok B$, let $\psi$ be an automorphism of $\ker B$, and let $M$ be any $n\times n$ matrix over \Z defining $\phi$, \ie, $\phi([x])=[Mx]$ for all $x\in\Z^n$. 

Then $\det(M)\equiv\pm 1 \pmod \delta$ if and only if there exist $n\times n$ invertible (\GL) matrices $U$ and $V$ over \Z such that $UBV=B$ and $U$ is defining $\phi$ and $V^{-1}$ is defining $\psi$.
\end{theorem}
\begin{proof}
The only thing that does not follow from \cite[Theorem~4.4]{MR1990568} is that we can choose the \GL-equivalence $(U,V)$ such that it also induces the right automorphism on $\ker B$. For this, it is clear that we may assume that $B$ is its own Smith normal form (just like in the proof of \cite[Theorem~4.4]{MR1990568}). 
We use \cite[Theorem~4.4]{MR1990568} to get a \GL-equivalence $(U,V)\colon B\rightarrow B$ that induces $\phi$ on $\cok B$. 
The matrix $V^{-1}$ induces an automorphism $\psi'$ of $\ker B$. 
Now we will find a \GL-equivalence $(I,V')\colon B\rightarrow B$ that induces $\psi\circ\psi'^{-1}$ on $\ker B$ --- then $(U,VV')$ is a \GL-equivalence that induces $\phi$ on $\cok B$ and $\psi$ on $\ker B$. Now, the automorphism $\psi\circ\psi'^{-1}$ on $\ker B$ uniquely determines what $V'^{-1}$ should be on the lower right block matrix (where we write the matrices as $2\times 2$ block matrices according to the nonzero respectively zero part of the diagonal of $B$). Let $V'^{-1}$ be the block diagonal matrix that has this matrix as lower right block matrix and the identity as the upper left block matrix.
\end{proof}

Now we let
$$\calP_{\min}=\{i\in\calP\colon j\prec i\Rightarrow i=j\}.$$
Using the above result, we get the following stronger version of Theorem~\ref{thm-BH-4.5}:

\begin{theorem}[{Strengthening of \cite[Theorem~4.5]{MR1990568}}] \label{thm-BH-4.5-B}
Let $\mathbf{n}=(n_i)_{i\in\calP}$ be a multiindex with $n_i\neq 0$, for all $i\in\calP$. 
Suppose $B$ and $B'$ are matrices in \MPZ[\mathbf{n}] with corresponding diagonal
blocks equal, and $\kappa\colon K(B)\rightarrow K(B')$ is a $K$-web isomorphism. 
Suppose that for each $i\in\calP_{\min}$, we have an automorphism $\psi_i\colon\ker B\{i\}\rightarrow\ker B\{i\}$. 
Then there exist matrices $U,V\in\GLZ[\mathbf{n}]$ such that we have a \GLPEe $(U, V )\colon B\rightarrow B'$ satisfying $\kappa_{(U,V )} = \kappa$ if and only if each of the
automorphisms $d_i\colon \cok B\{i\}\rightarrow \cok B'\{i\}$ defined by $\kappa$ are \GL-allowable --- 
moreover, the \GLPEe can always be chosen such that $V^{-1}\{i\}$ induces $\psi_i$ for each $i\in\calP_{\min}$.
\end{theorem}
\begin{proof}
The only thing that does not follow from \cite[Theorem~4.4]{MR1990568} (\cf\ Theorem~\ref{thm-BH-4.5}), is that we can choose the \GL-equivalence $(U,V)$ such that it also induces the right automorphisms on $\ker B\{i\}$, $i\in\calP_{\min}$. 
We choose a \GLPEe $(U,V)$ according to Theorem~\ref{thm-BH-4.5}, so that it induces the given $K$-web isomorphism. 
For each $i\in\calP_{\min}$, this gives an automorphism $\psi_i'$ of $\ker B\{i\}$. 
Now choose \GL-eqivalences $(I,V_i')$ of $B\{i\}$ according to (the proof of) Theorem~\ref{help} so that $V_i'^{-1}$ induces $\psi_i\circ\psi_i'^{-1}$ for each $i\in\calP_{\min}$.
Let $V$ be the block matrix that is the identity matrix everywhere except that 
$V\{i\}=V_i'$ for every $i\in\calP_{\min}$. 
It is straight forward to verify that $(I,V')$ is a \GLPEe from $B'$ to $B'$, and that 
$(U,VV')$ induces exactly what we want. 
\end{proof}

Together with \cite[Theorem~4.4 (and Proposition~4.1)]{MR1990568} (see Theorem~\ref{thm-BH-4.4} above), this gives us the following stronger version of Corollary~\ref{cor-BH-4.6}:

\begin{corollary}[{Strengthening of \cite[Corollary~4.7]{MR1990568}}] \label{cor-BH-4.6-B}
Let $\mathbf{n}=(n_i)_{i\in\calP}$ be a multiindex with $n_i\neq 0$, for all $i\in\calP$. 
Suppose $B$ and $B'$ are matrices in $\MPZ[\mathbf{n}]$ with $\gcd B\{i\}=1=\gcd B'\{i\}$ for all $i\in\calP$.
Then for any $K$-web isomorphism $\kappa\colon K(B)\rightarrow K(B')$ together with automorphisms $\psi_i\colon\ker B\{i\}\rightarrow\ker B\{i\}$, for $i\in\calP_{\min}$, 
there exist matrices $U,V\in\GLZ[\mathbf{n}]$ such that we have a \GLPEe $(U, V )\colon B\rightarrow B'$ satisfying $\kappa_{(U,V )} = \kappa$ and $V^{-1}\{i\}$ induces $\psi_i$ for each $i\in\calP_{\min}$.
\end{corollary}

The following theorem is the main result of this section, and allows us --- in certain cases --- to lift $K$-web isomorphisms to \GLPEe{s} for rectangular cases. 
Although it is possible to prove this directly, imitating the proof in \cite{MR1990568}, the present proof is much shorter and reduces the rectangular case to the square case and uses  the results from \cite{MR1990568}. 

\begin{theorem}\label{thm:mainBH}
Let $\mathbf{m}=(m_i)_{i\in\calP},\mathbf{n}=(n_i)_{i\in\calP}\in(\N_0)^N$ be multiindices. 
Suppose $B$ and $B'$ are matrices in \MPZ with $\gcd B\{i\}=1=\gcd B'\{i\}$ for all $i\in\calP$ with $m_i\neq 0$ and $n_i\neq 0$. 

Then for any $K$-web isomorphism $\kappa\colon K(B)\rightarrow K(B')$ there exist matrices $U\in\GLZ[\mathbf{m}]$ and $V\in\GLZ[\mathbf{n}]$ such that we have a \GLPEe $(U, V )\colon B\rightarrow B'$ satisfying $\kappa_{(U,V )} = \kappa$. 

If, moreover, we have given an isomorphism $\psi_i\colon\ker B\{i\}\to\ker B'\{i\}$, for every $i\in\calP_{\min}$, then we can choose the above \GLPEe $(U, V )$ such that --- in addition to the above --- also $V^{-1}\{i\}$ induces the $\psi_i$, for all $i\in\calP_{\min}$.
\end{theorem}
\begin{proof}
For each $i\in\calP$, choose $U_i,U_i'\in\GLZ[m_i]$ and $V_i,V_i'\in\GLZ[n_i]$ such that $D_i=U_iBV_i$ and $D_i'=U_i'B'V_i'$ are the Smith normal forms of $B$ and $B'$, respectively (\cf\ Theorem~\ref{thm-smith-normal-form}). 
Let $U,U'\in\GLZ[\mathbf{m}]$ and $V,V'\in\GLZ[\mathbf{n}]$ be the block diagonal matrices with $U_i,U_i',V_i$ and $V_i'$ in the diagonals, respectively.

Then $UBV$ and $U'B'V'$ are in $\MPZ$ and $(U,V)\colon B\rightarrow UBV$ and $(U,V)\colon B'\rightarrow U'B'V'$ are \GLPEe{s} inducing $K$-web isomorphisms $\kappa_{(U,V)}$ from $K(B)$ to $K(UBV)$ and $\kappa_{(U',V')}$ from $K(B')$ to $K(U'B'V')$, respectively. 
Moreover, Lemma~\ref{lem-gcd} ensures that we still have $\gcd(UBV)\{i\}=1=\gcd(U'B'V')\{i\}$. 
Thus we can without loss of generality assume that each diagonal block is equal to its Smith normal form. 
Also note, that because we have a $K$-web isomorphism from $K(B)$ to $K(B')$, now the diagonal blocks are necessarily identical. 

Let $\mathbf{r}$ be such that $r_i=\max(m_i,n_i)$ for all $i\in\calP$. 
Let, moreover, $C,C'\in\MPZ[\mathbf{r}]$ denote the matrices $B$ and $B'$ enlarged by putting zeros outside the original matrices. 
Define $\mathbf{r}^c$ and $\mathbf{r}^k$ by $r_i^c=\max(n_i-m_i,0)$ and $r_i^k=\max(m_i-n_i,0)$ for all $i\in\calP$. 
In the (reduced) $K$-web we are considering the modules $C_c(B)=\cok B(c)$ where $c$ is $\{i\}$, $\{j\in\calP\colon j\prec i\}\neq \emptyset$ or $\{j\in\calP\colon j\preceq i\}$ for $i\in\calP$ --- and similarly for $B'$. 
It is clear that when we consider $C$ and $C'$ we just add onto these cokernels 
$$\bigoplus_{j\in c}\Z^{r_j^c},$$
and that the maps between the modules are the obvious ones. 
Similarly for the modules $K_d(B)=\ker B(d)$ where $d$ is $\{i\}$ where $\{j\in\calP\colon j\prec i\}\neq \emptyset$ --- and similarly for $B'$. 
It is clear that when we consider $C$ and $C'$ we just add onto these kernels 
$$\Z^{r_i^k},$$
where $d=\{i\}$. And connecting homomorphism will be the zero maps. 

Thus we can extend the isomorphism $\kappa$ to an isomorphism $\widetilde{\kappa}\colon K(C)\rightarrow K(C')$ by setting it to be the identity on the new groups. 
By Corollary~\ref{cor-BH-4.6}, we see that there exist matrices $U,V\in\GLZ[\mathbf{r}]$ such that we have a \GLPEe $(U, V )\colon C\rightarrow C'$ satisfying $\kappa_{(U,V )} = \widetilde{\kappa}$. We may (according to Theorem~\ref{cor-BH-4.6-B}) actually assume that $(U,V)$ induces $\psi_i$ plus the identity on the new summands of $\ker B\{i\}$ for $i\in\calP_{\min}$ as well. 

Now let us look at the $i$'th diagonal block. We now want to cut $U$ and $V$ down to match the original structure. Naturally there are three cases to consider. The first one, $m_i=n_i$ is trivial. 

Now consider the case $m_i<n_i$. 
In this case, $\cok C\{i\}=\cok B\{i\}\oplus \Z^{n_i-m_i}$ and $\ker C\{i\}=\ker B\{i\}$ --- and similarly for $B'$ and $C'$. 
We write $C\{i\}=C'\{i\}$ as
$$\begin{pmatrix}C_{00} & 0 & 0 \\ 0 & 0 & 0 \\ 0 & 0 & 0\end{pmatrix},$$
where $C_{00}$ is an invertible matrix over \Q and the last diagonal block has size $(n_i-m_i)\times(n_i-m_i)$. 
We write $U$ and $V$ as
$$\begin{pmatrix}U_{11} & U_{12} & U_{13} \\ U_{21} & U_{22} & U_{23} \\ U_{31} & U_{32} & U_{33}\end{pmatrix} \quad\text{and}\quad
\begin{pmatrix}V_{11} & V_{12} & V_{13} \\ V_{21} & V_{22} & V_{23} \\ V_{31} & V_{32} & V_{33}\end{pmatrix},$$
according to the block structure of $C\{i\}$ (and $C'\{i\}$). 

The condition 
$$UCV=C'$$
implies that 
$$U_{11}C_{00}V_{11}=C_{00},\quad
U_{i1}C_{00}V_{1j}=0,\text{ for all }(i,j)\neq(1,1).$$
Since $C_{00}$ is invertible as a matrix over \Q, we see that also 
$U_{11}$ and $V_{11}$ have to be invertible over \Q. 
Thus $V_{12}=0$, $V_{13}=0$, $U_{21}=0$, and $U_{31}=0$. 
Moreover, since we have to get the identity homomorphism on the new direct summand, we need to have $U_{33}=I$, $U_{23}=0$ and $U_{32}=0$. 
So now let $U_0$ be the block matrix where we erase the rows and columns corresponding to change the size of the $i$'th diagonal block from $r_i\times r_i$ to $m_i\times m_i$ --- call the new size $\mathbf{r}'$. 
Moreover, we let $C_0$ and $C_0'$ be the block matrices where we erase the rows corresponding to change the size of the $i$'th diagonal block from $r_i\times r_i$ to $m_i\times n_i$. 
Note that the $i$'th diagonal block now is the matrix $\begin{smallpmatrix} U_{11} & U_{12} \\ 0 & U_{22} \end{smallpmatrix}$. 
This is a \GL matrix that induces the right automorphism of $\cok B\{i\}$. 
Moreover, clearly $U_0\{i\}B\{i\}V\{i\}=B\{i\}$. 
But more is true. 
We have that $U_0$ is a $\GLZ[\mathbf{r}']$ matrix and 
that $U_0C_0V=C_0'$ and the induced $K$-web isomorphism agrees with the original on all parts except for the direct summands we cut out.

Now consider instead the case $m_i>n_i$. 
In this case, $\cok C\{i\}=\cok B\{i\}$ and $\ker C\{i\}=\ker B\{i\}\oplus \Z^{m_i-n_i}$ --- and similarly for $B'$ and $C'$. 
We write $C\{i\}=C'\{i\}$ as
$$\begin{pmatrix}C_{00} & 0 & 0 \\ 0 & 0 & 0 \\ 0 & 0 & 0\end{pmatrix},$$
where $C_{00}$ is an invertible matrix over \Q and the last diagonal block has size $(m_i-n_i)\times(m_i-n_i)$. 
We write $U$ and $V$ as
$$\begin{pmatrix}U_{11} & U_{12} & U_{13} \\ U_{21} & U_{22} & U_{23} \\ U_{31} & U_{32} & U_{33}\end{pmatrix} \quad\text{and}\quad
\begin{pmatrix}V_{11} & V_{12} & V_{13} \\ V_{21} & V_{22} & V_{23} \\ V_{31} & V_{32} & V_{33}\end{pmatrix},$$
according to the block structure of $C\{i\}$ (and $C'\{i\}$). 

The condition 
$$UCV=C'$$
implies that 
$$U_{11}C_{00}V_{11}=C_{00},\quad
U_{i1}C_{00}V_{1j}=0,\text{ for all }(i,j)\neq(1,1).$$
Since $C_{00}$ is invertible as a matrix over \Q, we see that also 
$U_{11}$ and $V_{11}$ have to be invertible over \Q. 
Thus $V_{12}=0$, $V_{13}=0$, $U_{21}=0$, and $U_{31}=0$. 

Moreover, since we have to get the identity homomorphism on the new direct summand, we need to have $V_{33}=I$, $V_{23}=0$ and $V_{32}=0$. 
So now let $V_0$ be the block matrix where we erase the rows and columns corresponding to change the size of the $i$'th diagonal block from $r_i\times r_i$ to $n_i\times n_i$ --- call the new size $\mathbf{r}'$. 
Moreover, we let $C_0$ and $C_0'$ be the block matrices where we erase the rows corresponding to change the size of the $i$'th diagonal block from $r_i\times r_i$ to $m_i\times n_i$. 
Note that the $i$'th diagonal block now is the matrix $\begin{smallpmatrix} V_{11} & 0 \\ V_{21} & V_{22} \end{smallpmatrix}$. 
This is a \GL matrix that induces the right automorphism of $\ker B\{i\}$. 
Moreover, clearly $U\{i\}B\{i\}V_0\{i\}=B\{i\}$. 
But more is true. 
We have that $V_0$ is a $\GLZ[\mathbf{r}']$ matrix and 
that $UC_0V_0=C_0'$ and the induced $K$-web isomorphism agrees with the original on all parts except for the direct summands we cut out.

Induction finishes the proof. 
\end{proof}


\section{\texorpdfstring{$\GLP$}{GLP}-equivalence to \texorpdfstring{$\SLP$}{SLP}-equivalence}
\label{sec:crelle-trick}
 
In this section we are concerned with Step 3 in our proof outline in Section~\ref{sec:main} of the proof of \eqref{thm:main-item-3} implies \eqref{thm:main-item-1} in Theorem \ref{thm:main}. 
It is of course not true in general that any two \GLPE matrices will be \SLPE, so we will need to alter our matrices.
Our first step in that direction is to create a little more room.

\begin{lemma}\label{lem:SLEqexpansion}
Let $E$ be a graph with finitely many vertices and suppose 
\begin{align*}
\Bsf_{E} = \begin{pmatrix} A & X & Y \\ 0 & B & Z \\ 0 & 0 & C \end{pmatrix}
\end{align*}
where $B$ is an $n\times n$ matrix with entries from $\N_0\sqcup\{\infty\}$ for some $n\geq 2$ and the entries of rows $n-1$ and $n$ of $B$ are positive integers and the vertices corresponding to these two rows are regular vertices of $E$.  

Then there exists a graph $E'$ such that $E \sim_{M} E'$, and
\begin{align*}
\Bsf_{E'} = \begin{pmatrix} A & X' & Y' \\ 0 & B' & Z' \\ 0 & 0 & C \end{pmatrix}
\end{align*}
with $B'$ an $(n+2)\times(n+2)$ matrix with entries from $\N_0 \sqcup \{ \infty \}$ and there exists $V \in \MZ[n+2]$ with $\det( V ) = 1$ such that  
\begin{align*}
\begin{pmatrix} A_0 & X_0' & Y_0' \\ 0 & B_0' & Z_0' \\ 0 & 0 & C_0 \end{pmatrix}  \begin{pmatrix} I & 0 & 0 \\ 0 & V & 0 \\ 0 & 0 & I \end{pmatrix} = \begin{pmatrix} A_0 & X_0'' & Y_0 \\ 0 & B_0'' & Z_0'' \\ 0 & 0 & C_0 \end{pmatrix} 
\end{align*}
where 
$$\Bsf_E^\bullet=\begin{pmatrix} A_0 & X_0 & Y_0 \\ 0 & B_0 & Z_0 \\ 0 & 0 & C_0 \end{pmatrix} ,\qquad\Bsf_{E'}^\bullet=\begin{pmatrix} A_0 & X_0' & Y_0' \\ 0 & B_0' & Z_0' \\ 0 & 0 & C_0 \end{pmatrix},$$
$$X_0'' = \begin{pmatrix} X_0 & 0 \end{pmatrix},\qquad 
B_0'' = \begin{pmatrix} B_0 & 0 \\ 0 & I_{2}  \end{pmatrix},\quad \text{and}\quad 
Z_0'' = \begin{pmatrix} Z_0 \\ 0 \end{pmatrix}.$$

Moreover, if $E$ satisfies the property that for all $v, w \in E^{0}_{\mathrm{reg}}$ with $v \geq w$, there exists a path in $E$ from $v$ to $w$ through regular vertices in $E$, then $E'$ also satisfies the same property.
\end{lemma} 

\begin{proof}
Let $v$ be the vertex in $E$ corresponding to the entry $B(n-1,n-1) + 1$ and $w$ be the vertex in $E$ corresponding to the entry $B(n,n) + 1$. 
Outsplitting the vertices $v$ and $w$ with respect to the partitions $s^{-1} (v) = \{ e \} \sqcup \left( s^{-1} (v) \setminus \{e \} \right)$ where $r( e) = v$ and $s^{-1} (w) = \{ f \} \sqcup \left( s^{-1} (w) \setminus \{ f \} \right)$ where $r(f) = w$, we get a graph $F$ such that $E \sim_{M} F$. 

Let $X_{1}$ and $B_{1}$ be the column vector of $X$ and $B$, respectively, that corresponds to the $v$th-column of $\Bsf_{E}$ and let $X_{2}$ and $B_{2}$ be the column vector of $X$ and $B$, respectively, that corresponds to the $w$th-column of $\Bsf_{E}$. Then 
\begin{align*}
\Bsf_{F} = \begin{pmatrix} A & X' & Y \\ 0 & \widehat{B} & Z' \\ 0 & 0 & C \end{pmatrix}
\end{align*}
where $X' = \begin{pmatrix} X & X_{1} & X_{2} \end{pmatrix}$, $Z' = \begin{pmatrix} Z \\ 0 \end{pmatrix}$, and 
\begin{align*}
\widehat{B} = \begin{pmatrix}
B - J & B_{1} & B_{2} \\
\begin{matrix} 0 & \cdots  & 0 & 1 & 0 \end{matrix}   & 0 & 0 \\
\begin{matrix} 0 & \cdots & 0 & 0 & 1 \end{matrix}   & 0 & 0
\end{pmatrix},    
\end{align*}
where $J$ is the matrix that is zero in all entries except the last two diagonal entries which are $1$.
These account for the loops $v$ and $w$ lost when we did the outsplit. 
We can now use row additions (Corollary \ref{lem:oneStepRowAdd}), adding the last row of $\widehat{B}$ into the third last and the second last into the fourth last, to get a graph $E'$ such that $E' \Meq F \Meq E$ and where  
\begin{align*}
\Bsf_{E'} = \begin{pmatrix} A & X' & Y \\ 0 & B' & Z' \\ 0 & 0 & C \end{pmatrix},
\end{align*}
where 
\begin{align*}
B' = \begin{pmatrix}
B  & B_{1} & B_{2} \\
\begin{matrix} 0 & \cdots  & 0 & 1 & 0 \end{matrix}   & 0 & 0 \\
\begin{matrix} 0 & \cdots & 0 & 0 & 1 \end{matrix}   & 0 & 0
\end{pmatrix}.
\end{align*}

Let $\widetilde{X}_0$ be the matrix obtained from $X_0'$ by replacing the $(n-1)$'st column of $X_0'$ and the $n$'th column of $X_0'$ by the zero column and let $\widetilde{B}_0$ be the matrix obtained from $B_0'$ by replacing the $(n-1)$'st and the $n$'th column of $B_0$ by the zero column but keeping that last two rows intact.
Consider the matrix 
\begin{align*}
\begin{pmatrix}
A_0 & \widetilde{X}_0 & Y_0 \\
0 & \widetilde{B}_0 & Z_0' \\
0 & 0 & C_0
\end{pmatrix}.
\end{align*}
Let $V_{1} \in \MZ[n+2]$ be the matrix obtained from $I_{n+2}$ by switching the $(n-1)$'st and the $(n+1)$'st columns and let $V_{2} \in \MZ[n+2]$ be the matrix obtained from $I_{n+2}$ by switching the $n$'th and the $(n+2)$'nd columns.  Then $\det( V_{1} ) = \det( V_{2} ) = -1$ and 
\begin{align*}
\begin{pmatrix}
A_0 & \widetilde{X}_0 & Y_0 \\
0 & \widetilde{B}_0 & Z_0' \\
0 & 0 & C_0
\end{pmatrix}
\begin{pmatrix} I & 0 &  0 \\ 0 & V_{1} V_{2} & 0 \\ 0 & 0 & I \end{pmatrix} = \begin{pmatrix} A_0 & X_0'' & Y_0 \\ 0 & B_0'' & Z_0'' \\ 0 & 0 & C_0 \end{pmatrix}. 
\end{align*}

As in the proof of Proposition \ref{prop:cuntzsplicetwice} we let $E_{(i,j)}$ denote the elementary matrix that is equal to the identity matrix everywhere except for the $(i,j)$'th entry, that is $1$. 
Let $V_{3} = E_{(n+1,n-1)}\in \MZ[n+2]$, and let $V_{4} = E_{(n+2,n)} \in \MZ[n+2]$.
Then $\det( V_{3} ) = \det( V_{4} )  =1$ and
\begin{align*}
\begin{pmatrix}
A_0 & \widetilde{X}_0 & Y_0 \\
0 & \widetilde{B}_0 & Z_0' \\
0 & 0 & C_0
\end{pmatrix}
 \begin{pmatrix} I & 0 &  0 \\ 0 & V_{3} V_{4} & 0 \\ 0 & 0 & I \end{pmatrix} = \begin{pmatrix} A_0 & X_0' & Y_0 \\ 0 & B_0' & Z_0' \\ 0 & 0 & C_0 \end{pmatrix}.
 \end{align*} 
 
Set $V = V_{4}^{-1} V_{3}^{-1} V_{1} V_{2}$.  Then
\begin{align*}
\begin{pmatrix} A_0 & X_0' & Y_0 \\ 0 & B_0' & Z_0' \\ 0 & 0 & C_0 \end{pmatrix} \begin{pmatrix} I & 0 & 0 \\ 0 & V & 0 \\ 0 & 0 & I \end{pmatrix} =  \begin{pmatrix} A_0 & X_0'' & Y_0 \\ 0 & B_0'' & Z_0'' \\ 0 & 0 & C_0 \end{pmatrix}.
\end{align*}
Since $\det( V _{1}) = \det( V_{2} ) = -1$ and $\det( V _{3}) = \det( V_{4} ) = 1$, we have that $\det( V ) = 1$.

For the last part of the lemma, let $v_{1}$ and $v_{2}$ be the two additional vertices obtained from the outsplitting.  It is clear that if $v , w \in E^{0}_{\mathrm{reg}}$ such that $v \geq w$ in $E'$, then $v \geq w$ in $E$.  Thus, there exists a path in $E'$ from $v$ to $w$ through regular vertices of $E'$.  Suppose $v \in E^{0}_{\mathrm{reg}}$ and $v \geq v_{i}$.  Then by the definition of outsplitting and by the assumption on $E$, there exists an edge $e$ in $E'$ such that $r(e) = v_{i}$, $s( e ) \in E^{0}_{\mathrm{reg}}$, and $v \geq s(e)$.  It is now clear that there exists a path in $E'$ from $v$ to $v_{i}$ through regular vertices of $E'$.  Suppose $v_{i} \geq v$ with $v \in E^{0}_{ \mathrm{reg}}$.  By the definition of the outsplitting, there exists a path $\alpha = \alpha_{1} \cdots \alpha_{m}$ in $E'$ such that $s( \alpha_{1} ) = v_{i}$, $r( \alpha_{m} ) = v$, $r( \alpha_{1} ) \in E^{0}_{\mathrm{reg}}$, and $r( \alpha_{1} ) \geq v_{i}$.  Since $r( \alpha_{1} ) \geq v_{i}$ and $v_{i} \geq v$, we have that $r( \alpha_{1} ) \geq v$.  By the previous cases, we have that there exists a path in $E'$ from $r( \alpha_{1} )$ to $v$ in $E'$ through regular vertices in $E'$.  Hence, there exists a path in $E'$ from $v_{i}$ to $v$ through regular vertices in $E'$. 
For the pair $(v_1,v_2)$ this is clear by construction.
\end{proof}

We now connect the space we have created to our method of changing signs, \ie the Cuntz splice. 

\begin{lemma}\label{lem:GLeqCS}
Let $E$ be a graph with finitely many vertices such that 
 \[
 \Bsf_{E} = \begin{pmatrix} A & X & Y \\ 0 & B & Z \\ 0 & 0 & C \end{pmatrix}
\]
where $B$ is an $n\times n$ matrix with entries from $\N_0\sqcup\{\infty\}$ for some $n\geq 1$ and the entries of row $n$ of $B$ are positive integers and the vertex $v$ corresponding to this row is a regular vertex of $E$.  
Let $E_{v,-}$ be the Cuntz splice of $E$ at the vertex $v$.

Then $\det( U ) = 1$, $\det( V )  = -1$, and 
\begin{align*}
\begin{pmatrix} I & 0 & 0 \\ 0 & U & 0 \\ 0 & 0 & I \end{pmatrix}
 \overbrace{\begin{pmatrix}
A_0 & (X_{-})_0 &  Y_0 \\
0 & (B_{-})_0 & (Z_{-})_0 \\
0 & 0 & C_0
\end{pmatrix}}^{\Bsf^{\bullet}_{E_{v,-}} }
\begin{pmatrix} I & 0 & 0 \\ 0 & V & 0 \\ 0 & 0 & I \end{pmatrix} = \begin{pmatrix} A_0 & X_0'' & Y_0 \\ 0 & B_0'' & Z_0'' \\ 0 & 0 & C_0 \end{pmatrix} 
\end{align*}
where 
$$\Bsf_E^\bullet=\begin{pmatrix} A_0 & X_0 & Y_0 \\ 0 & B_0 & Z_0 \\ 0 & 0 & C_0 \end{pmatrix},$$
$$X_0'' = \begin{pmatrix} X_0 & 0 \end{pmatrix},\qquad 
B_0'' = \begin{pmatrix} B_0 & 0 \\ 0 & I_{2}  \end{pmatrix},\quad \text{and}\quad 
Z_0'' = \begin{pmatrix} Z_0 \\ 0 \end{pmatrix},$$
$V = V_{1} V_{2}$ where $V_{1}$ is the matrix obtained from $I_{n+2}$ by subtracting the $(n+2)$'nd column from the $n$'th column and $V_{2}$ is the matrix obtained from $I_{n+2}$ by switching the $(n+1)$'st and $(n+2)$'nd columns, and $U$ is the matrix obtained from $I_{m+2}$ by subtracting the $(m+2)$'nd row from the $m$'th row, with $m$ being the number of rows of $B_0$. 
\end{lemma}

\begin{proof}
The proof of the lemma is just a simple matrix computation using that the row operations to get $U$ only involve regular vertices of $E_{v,-}$, and is left for the reader.
\end{proof}

The next proposition will be our first step in going from a \GLPEe to an \SLPEe. 
The idea is to alter our graphs in such a way that their $\Bsf^{\bullet}$ matrices are \GLPE, say by $(U,V)$, but where all the diagonal blocks of $U$ have determinant one. 
Thus moving all our problems to $V$.

\begin{proposition}\label{p:GLtoSL1}
Let $E_{1}$ and $E_{2}$ be graphs with finitely many vertices such that $(E_{1} , E_{2} )$ is in standard form.
Suppose $\ftn{ (U_{1},V_{1}) }{ \Bsf^{\bullet}_{E_{1}} }{ \Bsf^{\bullet}_{E_{2}} }$ is a \GLPEe, where $U_{1} \in \GLPZ[\mathbf{m}]$, $V_1 \in \GLPZ[\mathbf{n}]$, $\mathbf{m} = ( m_{1} , \dots, m_{N} )$ and $\mathbf{n} = ( n_{1} , \dots, n_{N} )$. 
If, for some $i$, $m_i\neq 0$, $\det ( U_{1} \{ i\} ) = -1$, then there exist graphs $F_{1}$ and $F_{2}$ with finitely many vertices and there exist $U_{2} \in \GLPZ[\mathbf{m}'], V_2 \in \GLPZ[\mathbf{n}']$, where $\mathbf{m}' = (m_{1} , \dots, m_{i-1} , m_{i}+2 , m_{i+1} , \dots, m_{N} )$ and $\mathbf{n}' = (n_{1} , \dots, n_{i-1} , n_{i}+2 , n_{i+1} , \dots, n_{N} )$, such that
\begin{itemize}
 \item $E_{k} \Meq F_{k}$, $k=1,2$;
 \item $(F_{1} , F_{2} )$ is in standard form;
 \item $U_{2}\Bsf^{\bullet}_{F_{1}} V_{2} = \Bsf^{\bullet}_{F_{2}}$;
 \item $\det ( U_{2} \{ i\}  )= 1$, $\det ( V_{2} \{ i\} ) = - \det( V_{1} \{ i \} )$; and
 \item $\det( U_{2} \{ j\} ) = \det( U_{1} \{ j\} )$ and $\det( V_{2} \{ j\} ) = \det( V_{1} \{ j\} )$ for all $j \neq i$.
\end{itemize}
\end{proposition}

\begin{proof}
Write $\Bsf^{\bullet}_{E_{1}}$ as $\begin{pmatrix} A_{1} & X_{1} & Y_{1} \\ 0 & B_{1} & Z_{1} \\ 0 & 0 & C_{1} \end{pmatrix}$ and write $\Bsf^{\bullet}_{E_{2}}$ as $\begin{pmatrix} A_{2} & X_{2} & Y_{2} \\ 0 & B_{2} & Z_{2} \\ 0 & 0 & C_{2} \end{pmatrix}$, where $B_{k} = \Bsf^{\bullet}_{E_{k}} \{ i \}$.  
Apply Lemma~\ref{lem:SLEqexpansion} to both $E_{k}$'s to yield graphs $E_{k}'$ and matrices $V_{k}'$.
Define $\widetilde{U} \in \GLPZ[\mathbf{m}]$ by 
\begin{align*}
\widetilde{ U }\{ r, s \} = 
\begin{cases}
\begin{pmatrix}
U_{1}\{ i\} & 0 & 0\\
 0				& 0 & 1 \\
 0				& 1 & 0
\end{pmatrix} &\text{if $(r,s) = ( i, i )$} \\
\\
\begin{pmatrix}
U_{1} \{ r,i \} & 0 & 0 
\end{pmatrix} &\text{if $( r, s) = ( r, i), r\neq i$} \\
\\
\begin{pmatrix}
U_{1} \{ i, s \}  \\
0 \\
0
\end{pmatrix} &\text{if $(r,s) = ( i, s ),s\neq i$}  \\
\\
U_{1}\{ r,s \} &\text{otherwise}  
\end{cases}
\end{align*}
and set $\widetilde{V} =  \begin{pmatrix} I & 0 & 0 \\ 0 & V_{1}' & 0 \\ 0 & 0 & I \end{pmatrix} \overline{V} \begin{pmatrix} I & 0 & 0 \\ 0 & (V_{2}')^{-1} & 0 \\ 0 & 0 & I \end{pmatrix} $ where $\overline{V}$ is the matrix defined by 
\begin{align*}
\overline{ V } \{ r, s \} = 
\begin{cases}
\begin{pmatrix}
V_{1}\{ i\} & 0 & 0\\
 0				& 0 & 1 \\
 0				& 1 & 0
\end{pmatrix} &\text{if $(r,s) = ( i, i )$} \\
\\
\begin{pmatrix}
V_{1} \{ r,i \} & 0 & 0 
\end{pmatrix} &\text{if $( r, s) = ( r, i),r\neq i$} \\
\\
\begin{pmatrix}
V_{1} \{ i, s \}  \\
0 \\
0
\end{pmatrix} &\text{if $(r,s) = ( i, s ),s\neq i$}  \\
\\
V_{1}\{ r,s \} &\text{otherwise}  
\end{cases}
\end{align*}
Note that $\widetilde{V} \in \GLPZ[\mathbf{n}]$. 

By Lemma~\ref{lem:SLEqexpansion}, we have that 
\[
\widetilde{U} \Bsf^{\bullet}_{E_{1}'} \widetilde{V} = \Bsf^{\bullet}_{E_{2}'}.
\]
Note that $\det ( \widetilde{V} \{ i \} ) = - \det ( V_{1} \{ i \} )$, $\det ( \widetilde{U}\{i\} ) = 1$, and $\det ( \widetilde{U} \{ j \} ) = \det ( U_{1} \{ j \} )$ and $\det( \widetilde{V} \{ j \} ) = \det ( V_{1} \{ j \} )$ for all $j \neq i$. 
Since $E_{k}$ is in canonical form, we have that for every vertex $v, w \in (E_{k})^{0}_{\mathrm{reg}}$ with $v \geq w$, there exists a path in $E_{k}$ from $v$ to $w$ through regular vertices in $E_{k}$.
Hence, by Lemma~\ref{lem:SLEqexpansion}, for every $v , w \in (E_{k}')^{0}_{\mathrm{reg}}$ with $v \geq w$, there exists a path in $E_{k}'$ from $v$ to $w$ through regular vertices in $E_{i}'$.  
Since $E_{k}$ is in canonical form and by definition of the outsplitting graph, $E_{k}'$ satisfies (\ref{thm:canonical-item-loops-inf}), (\ref{thm:canonical-item-paths}), and (\ref{thm:canonical-item-size3}) of Theorem~\ref{thm: canonical form}. 
Furthermore, the fact that $E_{k}$ is in canonical form implies that all the diagonal blocks of $\Bsf^{\bullet}_{E_k}$ have Smith normal form with at least two 1's, so it follows from the constructions of Lemma \ref{lem:SLEqexpansion} that the diagonal blocks of $\Bsf^{\bullet}_{E_{k}'}$ also have this property. Therefore $E_{k}'$ also satisfies (\ref{thm:canonical-item-rowrank}) of Theorem~\ref{thm: canonical form}. 
By Remark~\ref{rmk:  canonical form and SL equivalence} we get a graph $F_{i}$ in canonical form such that $\mathbf{m}_{F_{i}} = \mathbf{m}_{E_{i}'} = \mathbf{m}'$, $\mathbf{n}_{F_{i}} = \mathbf{n}_{ E_{i}'} = \mathbf{n}'$, and $E_{i}' \sim_{M'} F_{i}$.
Also, we get an \SLPEe $\ftn{ ( W_{i} , Z_{i} ) }{ \Bsf^{\bullet}_{F_{i}} }{ \Bsf^{\bullet}_{E_{i}'} }$.  

Set $U_{2} = W_{2}^{-1} \widetilde{U} W_{1}$ and $V_{2} = Z_{1} \widetilde{V} Z_{2}^{-1}$.  Since $W_{2}^{-1} , W_{1} \in \SLPZ[\mathbf{m}']$ and since $Z_{1} , Z_{2}^{-1} \in \SLPZ[\mathbf{n}']$, we have that $\det ( U_{2} \{ i \} ) = 1$, $\det ( V_{2}\{i\} ) = - \det ( V_{1} \{ i \} )$, and $\det ( U_{2} \{ j \} ) = \det ( U_{1} \{ j \} )$ and $\det( V_{2}\{ j \} ) = \det ( V_{1} \{ j \} )$ for all $j \neq i$.  By construction, the pair $(F_{1} , F_{2} )$ is in standard form with $\Bsf^{\bullet}_{F_{i}} \in \MPZ[\mathbf{m}' \times \mathbf{n}']$, $U_{2} \Bsf^{\bullet}_{F_{1}} V_{2} = \Bsf^{\bullet}_{F_{2}}$, and $F_{i} \sim_{M} E_{i}$.
\end{proof}

We now use the Cuntz splice to fix potential sign problems on $V$.

\begin{proposition}\label{p:GLtoSL2}
Let $E_{1}$ and $E_{2}$ be graphs with finitely many vertices such that $(E_{1} , E_{2} )$ is in standard form.
Suppose $\ftn{ (U_{1},V_{1}) }{ \Bsf^{\bullet}_{E_{1}} }{ \Bsf^{\bullet}_{E_{2}} }$ is a \GLPEe, where $U_{1} \in \GLPZ[\mathbf{m}]$ and $V_{1} \in \GLPZ[\mathbf{n}]$,
and $\mathbf{m} = ( m_{1} , \dots, m_{N} )$ and $\mathbf{n} = ( n_{1} , \dots, n_{N} )$.
If, for some $i$, $m_i\neq 0$, $\det ( U_{1} \{ i\} ) = 1$ and $\det( V_{1} \{i\} ) = -1$, there exist graphs $F_{1}$ and $F_{2}$ with finitely many vertices and there exist $U_{2} \in \GLPZ[\mathbf{m}']$ and $V_{2} \in \GLPZ[\mathbf{n}']$, where $\mathbf{m}' = (m_{1} , \dots, m_{i-1} , m_{i}+2 , m_{i+1} , \dots, m_{N} )$ and $\mathbf{n}' = (n_{1} , \dots, n_{i-1} , n_{i}+2 , n_{i+1} , \dots, n_{N} )$ such that
\begin{itemize}
 \item $E_{k} \MCeq F_{k}$, for $k=1,2$;
 \item $(F_{1} , F_{2} )$ is in standard form;
 \item $U_{2}\Bsf^{\bullet}_{F_{1}} V_{2} = \Bsf^{\bullet}_{F_{2}}$;
 \item $\det ( U_{2} \{ i \} )  = \det ( V_{2} \{ i \} ) = 1$; and
 \item $\det( U_{2} \{ j\} ) = \det( U_{1} \{ j\} )$ and $\det( V_{2} \{ j\} ) = \det( V_{1} \{ j\} )$ for all $j \neq i$.
\end{itemize}
\end{proposition}

\begin{proof}
Write $\Bsf^{\bullet}_{E_{1}}$ as $\begin{pmatrix} A_{1} & X_{1} & Y_{1} \\ 0 & B_{1} & Z_{1} \\ 0 & 0 & C_{1} \end{pmatrix}$ and write $\Bsf^{\bullet}_{E_{2}}$ as $\begin{pmatrix} A_{2} & X_{2} & Y_{2} \\ 0 & B_{2} & Z_{2} \\ 0 & 0 & C_{2} \end{pmatrix}$, where $B_{k} = \Bsf^{\bullet}_{E_{k}} \{ i \}$ and the entries of the last two rows of $B_{k}$ are positive integers, and the corresponding vertices of $E_k$ are regular. 
Apply Lemma~\ref{lem:SLEqexpansion} to $E_{2}$ to yield a graph $E_{2}'$ and a matrix $V_{2}'$.
Let $U_{-}$ and $V_{-}$ be the matrices that one obtain when applying Lemma~\ref{lem:GLeqCS} to the graph $E_{1}$. 

Set $\widetilde{U} = \overline{U} \begin{pmatrix} I & 0 & 0 \\ 0 & U_{-} & 0 \\ 0 & 0 & I \end{pmatrix}$ where $\overline{U}$ is the matrix defined by 
\begin{align*}
\overline{ U } \{ r, s \} = 
\begin{cases}
\begin{pmatrix}
U_{1}\{ i\} & 0 & 0\\
 0				& 1 & 0 \\
 0				& 0 & 1
\end{pmatrix} &\text{if $(r,s) = ( i, i )$} \\
\\
\begin{pmatrix}
U_{1} \{ r,i \} & 0 & 0 
\end{pmatrix} &\text{if $( r, s) = ( r, i),r\neq i$} \\
\\
\begin{pmatrix}
U_{1} \{ i, s \}  \\
0 \\
0
\end{pmatrix} &\text{if $(r,s) = ( i, s ),s\neq i$}  \\
\\
U_{1}\{ r,s \} &\text{otherwise}  
\end{cases}
\end{align*}
Note that $\widetilde{U} \in \GLPZ[\mathbf{m}']$.

Set $\widetilde{V}= \begin{pmatrix} I & 0 & 0 \\ 0 & V_{-} & 0 \\ 0 & 0 & I \end{pmatrix} \overline{V} \begin{pmatrix} I & 0 & 0 \\ 0 & ( V_{2}' )^{-1} & 0 \\ 0 & 0 & I \end{pmatrix}$ where $\overline{V}$ is the matrix defined by 
\begin{align*}
\overline{ V } \{ r, s \} = 
\begin{cases}
\begin{pmatrix}
V_{1}\{ i\} & 0 & 0\\
 0				& 1 & 0 \\
 0				& 0 & 1
\end{pmatrix} &\text{if $(r,s) = ( i, i )$} \\
\\
\begin{pmatrix}
V_{1} \{ r,i \} & 0 & 0 
\end{pmatrix} &\text{if $( r, s) = ( r, i),r\neq i$} \\
\\
\begin{pmatrix}
V_{1} \{ i, s \}  \\
0 \\
0
\end{pmatrix} &\text{if $(r,s) = ( i, s ),s\neq i$}  \\
\\
V_{1}\{ r,s \} &\text{otherwise}  
\end{cases}
\end{align*}
Note that $\widetilde{V} \in \GLPZ[\mathbf{n}']$.

By Lemma~\ref{lem:SLEqexpansion} and Lemma~\ref{lem:GLeqCS}, we have that 
\[
\widetilde{U} \Bsf^{\bullet}_{(E_{1})_{-}} \widetilde{V} = \Bsf^{\bullet}_{E_{2}'}.
\]
Note that $\det( \widetilde{U} \{ i \} ) = 1$ and $\det( \widetilde{V} \{ i \} ) = 1$; moreover, $\det( \widetilde{U}\{j\} ) = \det ( U_{1} \{ j \} )$ and $\det( \widetilde{V}\{j\}) = \det( V_{1} \{ j \} )$, for all $j \neq i$. 
Since $E_{2}$ is in canonical form, by Lemma~\ref{lem:SLEqexpansion}, $E_{2}'$ has the property that for every vertex $v, w \in (E_{2}')^{0}_{\mathrm{reg}}$ with $v \geq w$, there exists a path in $E_{2}'$ from $v$ to $w$ through regular vertices in $E_{2}'$. 
It is clear from the construction of $(E_{1})_{-}$ that for all regular vertices $v, w$ of $(E_{1})_{-}$ satisfying $v \geq w$, we have that there exists a path in $(E_{1})_{-}$ from $v$ to $w$ through regular vertices of $( E_{1})_{-}$ (since $E_{1}$ has this property).  

Since $E_{2}$ is in canonical form and by the definition of the outsplitting graph, we have that $E_{2}'$ satisfies (\ref{thm:canonical-item-loops-inf}), (\ref{thm:canonical-item-paths}) and (\ref{thm:canonical-item-size3}) of Theorem~\ref{thm:  canonical form}.
Similarly, $(E_{1})_{-}$ will satisfying the same properties since $E_{1}$ is in canonical form. 
Furthermore, the canonical form of $E_{k}$ and the construction of Lemma \ref{lem:SLEqexpansion} implies that the diagonal blocks of $\Bsf^{\bullet}_{E_{k}'}$ have a Smith normal form with at least two 1's, so $E_{k}'$ also satisfies (\ref{thm:canonical-item-rowrank}) of Theorem~\ref{thm: canonical form}. 
By Remark~\ref{rmk:  canonical form and SL equivalence}, there exist graphs $F_{1}$, $F_{2}$ in canonical form such that $\mathbf{m}_{ F_{1} } = \mathbf{m}_{ ( E_{1} )_{-} } = \mathbf{m}'$, $\mathbf{n}_{ F_{1} } = \mathbf{n}_{ ( E_{1} )_{-} } = \mathbf{n}'$, $F_{1} \sim_{M} (E_{1})_{-}$, and $F_{2} \sim_{M} E_{2}'$.
Moreover, there exist \SLPEe{s} $\ftn{ ( W_{1} , Z_{1} ) }{ \Bsf^{\bullet}_{F_{1}} }{ \Bsf^{\bullet}_{(E_{1})_{-}} }$ and $\ftn{ ( W_{2} , Z_{2} ) }{ \Bsf^{\bullet}_{F_{2}} }{ \Bsf^{\bullet}_{E_{2}'} }$.

Set $U_{2} = W_{2}^{-1} \widetilde{U} W_{1}$ and $V_{2} = Z_{1} \widetilde{V} Z_{2}^{-1}$.  Since $W_{2}^{-1} , W_{1} \in \SLPZ[\mathbf{m}']$ and since $Z_{1} , Z_{2}^{-1} \in \SLPZ[\mathbf{n}']$, we have that $\det ( U_{2} \{ i \} ) = \det ( V_{2}\{i\} ) = 1$, and $\det ( U_{2} \{ j \} ) = \det ( U_{1} \{ j \} )$ and $\det( V_{2}\{ j \} ) = \det ( V_{1} \{ j \} )$ for all $j \neq i$.  By construction, the pair $(F_{1} , F_{2} )$ is in standard form with $\Bsf^{\bullet}_{F_{k}} \in \MPZ[\mathbf{m}' \times \mathbf{n}']$, $U_{2} \Bsf^{\bullet}_{F_{1}} V_{2} = \Bsf^{\bullet}_{F_{2}}$, and $F_{k} \sim_{M'} E_{k}$.
\end{proof}

We now have all we need to modify a \GLPEe to an \SLPEe. 

\begin{theorem}\label{thm:GLtoSL}
Let $E_{1}$ and $E_{2}$ be graphs with finitely many vertices such that the pair $(E_{1} , E_{2} )$ is in standard form.  
Suppose $(U,V)$ is a \GLPEe from $\Bsf^{\bullet}_{E_{1}}$ to $\Bsf^{\bullet}_{E_{2}}$ satisfying that $V\{i\}=1$ whenever $n_i=1$.
Then there exist graphs $F_{1}$ and $F_{2}$ such that $E_{i} \sim_{M'} F_{i}$, the pair $(F_{1}, F_{2} )$ is in standard form, and $\Bsf^{\bullet}_{F_{1}}$ is \SLPE to $\Bsf^{\bullet}_{F_{2}}$.
\end{theorem} 

\begin{proof}
The theorem follows from an argument similar to the argument in \cite[Theorem~6.8]{MR2270572} with Propositions~\ref{p:GLtoSL1} and~\ref{p:GLtoSL2} in place of \cite[Lemma~6.7]{MR2270572}. 

Briefly, the idea is that we are given a \GLPEe, say $(U,V)$. 
We go down the diagonal blocks and for each of them use Proposition~\ref{p:GLtoSL1} if necessary to make sure the $U$ has positive determinant. 
Then we go down the diagonal blocks again this time using Proposition~\ref{p:GLtoSL2} to fix the determinant of the diagonal blocks of $V$ when necessary. 
\end{proof}


\section{Generalization of Boyle's positive factorization method}
\label{sec:boyle}

In \cite{MR1907894}, Boyle proved several factorization theorems for square matrices.  These theorems are the key components to go from \SLPEe to flow equivalence.  In this section, we prove similar factorization theorems for rectangular matrices.  This is our key technical result to go from \SLPEe to move equivalence.  Although the assumptions might seem restrictive, every unital graph \ca is move equivalent to another unital graph \cas whose adjacency matrix satisfy the assumptions of the factorization theorem.  The proof for rectangular matrices will closely follow the proof in \cite{MR1907894} for square matrices.  

First we introduce a new equivalence called ``positive equivalence'' of two matrices in \MPplusZ (see Definition~\ref{def: positive matrices}) and show that if $n_{i} \neq 0$ for all $i$, then two matrices in \MPplusZ that are \SLPE are positive equivalent.

\begin{definition}\label{def: positive matrices}
Define \MPplusZ to be the set of all $B \in \MPZ$ satisfying the following:
\begin{enumerate}[(i)]
\item If $i \preceq j$ and $B \{ i, j \}$ is not the empty matrix, then $B \{ i , j \} > 0$.

\item If $B \{ i \}$ is not the empty matrix, then $B \{ i \} > 0$, the Smith normal form of  $B \{ i \}$ has at least two 1's, and $n_{i} , m_{i} \geq 3$.
\end{enumerate}

Note that condition (ii) implies that the row rank of every non-empty diagonal block  is at least 2. In most of what follows, this will suffice for our purposes, but the stronger condition is needed to apply Theorem \ref{thm:BoyleTheorem5.1} below.

Let $B, B' \in \MPplusZ$.  An \SLPEe $\ftn{ (U,V) }{B}{B'}$ is said to be a \emph{positive equivalence} if $U$ has a factorization of basic elementary matrices in \SLPZ[\mathbf{m}] and $V$ has a factorization of basic elementary matrices in \SLPZ such that when applying these basic elementary matrices at each step we get matrices in \MPplusZ (recall from \cite{MR1907894} that a basic elementary matrix is a matrix that is equal to the identity matrix except for on one offdiagonal entry, where it is either $1$ or $-1$).  We denote a positive equivalence by $\xymatrix{B \ar[r]^-{(U,V)}_-{+} & B'}$.
\end{definition}

Note that every element $U \in \SLPZ$ has a factorization of basic elementary matrices in \SLPZ.
Therefore, a positive equivalence $\ftn{ (U,V) }{B}{B'}$ is an \SL-equivalence that allows one to stay in 
\MPplusZ for some factorization of $U$ and $V$.  

\subsection{Factorization: Positive case}

In this section, we prove a factorization theorem similar to that of \cite[Theorem~5.1]{MR1907894} for positive rectangular matrices.  The proof is imitating the proof in \cite{MR1907894} for square matrices.   

\begin{definition}
By a \emph{signed transposition matrix}, we mean a matrix which is the matrix of a transposition, but with one of the offdiagonal $1$'s replaced by $-1$.  By a \emph{signed permutation matrix} we mean a product of signed transposition matrices.
\end{definition}

Note that for $K > 1$, any $K \times K$ permutation matrix with determinant 1 is a signed permutation matrix.  A $K \times K$ matrix $S$ is a signed permutation matrix if and only if $\det (S) = 1$ and the matrix $| S |$ is a permutation matrix (where $| S | (i,j) := | S(i,j) |$).

For $B, B' \in \Mplus[m\times n]$, we say an equivalence $\ftn{ (U,V) }{ B }{ B' }$ is a \emph{positive equivalence} through \Mplus[m\times n] if it can be given as a chain of positive elementary equivalences 
\begin{align*}
B = B_{0} \to B_{1} \to B_{2} \to \cdots \to B_{k} = B'
\end{align*}
in which every $B_{i}$ is in \Mplus[m\times n] (recall from \cite{MR1907894} that an equivalence $(U,V)$ is an elementary equivalence if one of $U$ and $V$ is a basic elementary matrix and the other is the identity matrix).

Investigating the proof of \cite[Lemma~5.3 and Lemma~5.4]{MR1907894} one can see the proofs also hold for rectangular matrices.  Thus, we have the following lemmas.

\begin{lemma}[{\cf \cite[Lemma~5.3]{MR1907894}}]\label{lem:BoyleLemma5.3}
Suppose $B \in \Mplus[m\times n]$, $E$ is a basic elementary matrix with nonzero offdiagonal entry $E(i,j)$, and the $i$th row of $EB$ is not the zero row.  Then there exists $Q \in \SLZ[n]$ that is a product of nonnegative basic elementary matrices and there exists a signed permutation matrix $S \in \SLZ[m]$ such that $\ftn{ (SE, Q ) }{ B }{ SE B Q }$ is a positive equivalence through \Mplus[m\times n].
\end{lemma}

\begin{lemma}[\cf\ {\cite[Lemma~5.4]{MR1907894}}]\label{lem:BoyleLemma5.4}
Let $B$ be an element of $\MZ[K_1\times K_2]$ for $K_{1} ,K_{2} \geq 3$ such that the row rank of $B$ is at least 2.  Suppose $U \in \SLZ[K_{1}]$ such that no row of $B$ and $U B$ is the zero row.  Then $U$ is the product of elementary matrices $U = E_{k} \cdots E_{1}$ such that for $1 \leq j \leq k$ the matrix $E_{j} E_{j-1} \cdots E_{1} B$ has no zero rows. 
\end{lemma}

The following lemma is inspired by the reduction step in the proof of \cite[Lemma~5.5]{MR1907894}.  We give the entire proof for the convenience of the reader. 

\begin{lemma}\label{lem:BoyleLemma5.5a}
Let $B \in \Mplus[K_{1}\times K_{2}]$ with $K_{1} ,K_{2} \geq 3$.  Suppose the row rank of $B$ is at least 2 and there exists $U \in \SLZ[K_{1}]$ such that $U B > 0$.  Then the equivalence $\ftn{ (U , I_{K_{2}} ) }{ B }{ UB }$ is a positive equivalence through \Mplus[K_{1}\times K_{2}]. 
\end{lemma}

\begin{proof}
By Lemma~\ref{lem:BoyleLemma5.4}, we can write $U$ as a product of basic elementary matrices $U = E_{k} E_{k-1} \cdots E_{1}$, such that for $1 \leq j \leq k$, the matrix $E_{j} \cdots E_{1}B$ has no zero row.  By Lemma~\ref{lem:BoyleLemma5.3}, given the pair $( E_{1} , B)$, there is a nonnegative $Q_{1}$ which is a product of nonnegative basic elementary matrices and a signed permutation $S_{1}$ such that $\ftn{ (S_1E_{1} , Q_{1} ) }{ B }{ S_{1} E_{1} B Q_{1} }$ is a positive equivalence through \Mplus[K_{1}\times K_{2}].  Note that 
\begin{align*}
UBQ_{1} = S_{1}^{-1} [ S_{1} E_{k} S_{1}^{-1} ] \cdots [ S_{1} E_{2} S_{1}^{-1} ] [ S_{1} E_{1} ] B Q_{1}.
\end{align*}  
Now, for $2 \leq j \leq k$, the matrix $S_{1} E_{j} S_{1}^{-1}$ is again a basic elementary matrix $E_{j}'$.  Since $E_{j}' \cdots E_{2}' ( S_{1} E_{1} B Q_{1} ) = S_{1} E_{j} \cdots E_{2} E_{1} B Q_{1}$ for $2 \leq j \leq k$  and since $E_{j} \cdots E_{2} E_{1} B Q_{1}$ has no zero rows, and $S_{1}$ is a signed permutation, we have that $E_{j}' \cdots E_{2}' ( S_{1} E_{1} B Q_{1} )$ has no zero rows for all $2 \leq j \leq k$.  

Using Lemma~\ref{lem:BoyleLemma5.3}, for the pair $( S_{1} E_{2} S_{1}^{-1} , S_{1} E_{1} B Q_{1} )$, we get a signed permutation matrix $S_{2}$ and a nonnegative $Q_{2}$ which is a product of nonnegative basic elementary matrices such that 
\begin{align*}
\ftn{ ( S_{2} [ S_{1} E_{2} S_{1}^{-1} ] , Q_{2} ) }{ S _{1} E_{1} B Q_{1} }{ S_{2} [ S_{1} E_{2} S_{1} ]^{-1} S_{1} E_{1} B Q_{1} Q_{2} }
\end{align*}
is a positive equivalence through \Mplus[K_{1}\times K_{2}].  Thus, we get a positive equivalence through \Mplus[K_{1}\times K_{2}] 
\begin{align*}
\ftn{ ( [S_{2} S_{1} E_{2} S_{1}^{-1} ] [ S_{1} E_{1} ] , Q_{1}Q_{2} ) }{ B }{ S_{2} S_{1} E_{2}  E_{1} B Q_{1} Q_{2} }
\end{align*}
and we observe that 
\begin{align*}
U B Q_{1} Q_{2} &= S_{1}^{-1} S_{2}^{-1} [ S_{2} S_{1} E_{k} S_{1}^{-1} S_{2}^{-1} ] \cdots \\ 
& \qquad\cdots [ S_{2} S_{1} E_{3} S_{1}^{-1} S_{2}^{-1} ][ S_{2} S_{1} E_{2} S_{1}^{-1} ] [ S_{1} E_{1} ] B Q_{1} Q_{2}.
\end{align*}
Continue this, to obtain a signed permutation matrix $S = S_{k} \cdots S_{1}$ and a nonnegative matrix $Q = Q_{1} Q_{2} \cdots Q_{k}$ that is a product of nonnegative basic elementary matrices such that 
\begin{align*}
U B Q = S^{-1} [ S_{k} \cdots S_{1} E_{k} S_{1}^{-1} \cdots S_{k-1}^{-1} ] \cdots [ S_{2} S_{1} E_{2} S_{1}^{-1} ][ S_{1} E_{1} ] BQ = S^{-1} ( SU B Q )
\end{align*}
and $\ftn{ ( SU, Q ) }{ B }{ SU B Q }$ is a positive equivalence through \Mplus[K_{1}\times K_{2}].  

We claim that the equivalence $\ftn{ ( S, I_{K_{2}} ) }{ U B Q }{ SU B Q }$ is a positive equivalence through \Mplus[K_{1}\times K_{2}].  Since $S$ is a product of signed transposition matrices, it may be described as a permutation matrix in which some rows have been multiplied by $-1$.  Since $UBQ$ and $SUBQ$ are strictly positive, it must be that $S$ is a permutation matrix.  Also, $\det( S ) = 1$, so if $S \neq I_{K_{1}}$, then $S$ is a permutation matrix which is a product of $3$-cycles.  So it is enough to realize the positive equivalence through \Mplus[K_{1}\times K_{2}] in the case that $S$ is the matrix of a $3$-cycle.  For this we write the matrix
\begin{align*}
C = \begin{pmatrix} 0 & 1 & 0 \\ 0 & 0 & 1 \\ 1 & 0 & 0 \end{pmatrix}
\end{align*}
as the following product $C_{0} C_{1} C_{2} C_{3} C_{4} C_{5}$:
\begin{align*}
\begin{pmatrix}
1 & 0 & 0 \\
0 & 1 & 0 \\
0 & -1 & 1
\end{pmatrix}
\begin{pmatrix}
1 & 0 & 0 \\
-1 & 1 & 0 \\
0 & 0 & 1
\end{pmatrix}
\begin{pmatrix}
1 & 0 & -1 \\
0 & 1 & 0 \\
0 & 0 & 1
\end{pmatrix}
\begin{pmatrix}
1 & 1 & 0 \\
0 & 1 & 0 \\
0 & 0 & 1
\end{pmatrix}
\begin{pmatrix}
1 & 0 & 0 \\
0 & 1 & 0 \\
1 & 0 & 1
\end{pmatrix}
\begin{pmatrix}
1 & 0 & 0 \\
0 & 1 & 1 \\
0 & 0 & 1
\end{pmatrix}.
\end{align*}
For $0 \leq i \leq 5$, the matrix $C_{i} C_{i+1} \cdots C_{5}$ is nonnegative and has no zero row. Therefore, the equivalence $\ftn{ ( C, I ) }{ D }{ CD }$ is a positive equivalence through \Mplus[K_{1}\times K_{2}] whenever $D \in \Mplus[K_{1}\times K_{2}]$.  Therefore, $\ftn{ ( S, I_{K_{2}} ) }{ U B Q }{ SU B Q }$ is a positive equivalence through \Mplus[K_{1}\times K_{2}] proving the claim.  Therefore, $\ftn{ ( S^{-1} , I_{K_{2} } ) }{ SU BQ }{ UBQ }$ is a positive equivalence through \Mplus[K_{1}\times K_{2}].  Since $Q$ is the product of nonnegative basic elementary matrices and $UB \in \Mplus[K_{1}\times K_{2}]$, the equivalence $\ftn{ ( I_{K_{1}} , Q ) }{ UB }{ UBQ }$ is a positive equivalence through \Mplus[K_{1}\times K_{2}].  Thus, $\ftn{ ( I_{K_{1}} , Q^{-1} ) }{ UBQ }{ UB }$ is a positive equivalence through \Mplus[K_{1}\times K_{2}].  Now the composition of positive equivalences through \Mplus[K_{1}\times K_{2}]
\begin{align*}
\xymatrix{
B \ar[rr]^-{ ( SU, Q ) }_-{+}  & & SU BQ \ar[rr]^-{ ( S^{-1} , I_{K_{2} } ) }_-{+} & & UBQ \ar[rr]^-{ ( I_{K_{1}} , Q^{-1} ) }_-{+} & & UB
}
\end{align*}
is positive equivalence through \Mplus[K_{1}\times K_{2}] but the composition of these equivalences is equal to the equivalence $\ftn{ ( U, I_{K_{2} } ) }{ B }{ UB }$.  Hence, the equivalence $\ftn{ ( U, I_{K_{2} } ) }{ B }{ UB }$ is a positive equivalence through \Mplus[K_{1}\times K_{2}].
\end{proof}

The proof of the next lemma is similar to the proof of \cite[Lemma~5.5]{MR1907894}.  Since there are some differences between the two proofs we provide the entire argument.  

\begin{lemma}[\cf\ {\cite[Lemma~5.5]{MR1907894}}]\label{lem:BoyleLemma5.5b}
Let $B$ and $B'$ be elements of \Mplus[K_{1}\times K_{2}] with $K_{1} ,K_{2} \geq 3$, and the rank of $B$ and $B'$ at least 2.  Suppose $U \in \SLZ[K_{1}]$ and $W \in \SLZ[K_{2}]$ such that $U B$ has at least one strictly positive entry and $UB = B' W$.  Then the equivalence $\ftn{ (U , W^{-1} ) }{ B }{ B' }$ is a positive equivalence through \Mplus[K_{1}\times K_{2}]. 
\end{lemma}

\begin{proof}
We will first reduce to the case that $UB > 0$.  By assumption $(UB)(i,j) > 0$ for some $(i,j)$.  We can repeatedly add column $j$ to other columns of until row $i$ of $UB$ has all entries strictly positive.  This corresponds to multiplying from the right by a nonnegative matrix $Q$ in \SLZ[K_{2}], where $Q$ is the product of nonnegative basic elementary matrices, giving $UB Q = B' W Q$.  Then we can repeatedly add row $i$ of $UBQ$ to other rows until all entries are positive.  This corresponds to multiplying from the left by a nonnegative matrix $P$ in \SLZ[K_{1}], where $P$ is the product of nonnegative basic elementary matrices, giving $(PU)(BQ) = (PB')(WQ) > 0$.  We also have positive equivalences through \Mplus[K_{1}\times K_{2}] given by 
\begin{align*}
\ftn{ ( I , Q ) }{ B }{ BQ } \quad \text{and} \quad \ftn{ (P,I) }{ B' }{ PB' }.
\end{align*}

Note that the equivalence $\ftn{ (U, W^{-1}) }{ B }{ B' }$ is the composition of equivalences, $\ftn{( I, Q ) }{ B }{ BQ }$ followed by $\ftn{ (PU, (WQ)^{-1} ) }{ BQ }{ PB' }$ followed by $\ftn{ ( P^{-1} , I ) }{ PB' }{ B' }$.  Since $\ftn{( I, Q ) }{ B }{ BQ }$ and $\ftn{ ( P^{-1} , I ) }{ PB' }{ B' }$ are positive equivalences through \Mplus[K_{1}\times K_{2}], it is enough to show that the equivalence $\ftn{ (PU, (WQ)^{-1} ) }{ BQ }{ PB' }$ is a positive equivalence through \Mplus[K_{1}\times K_{2}].  Therefore, after replacing $(U, B, B' ,W)$ with $(PU, BQ, PB', WQ)$, we may assume without loss of generality that $UB > 0$.

By Lemma~\ref{lem:BoyleLemma5.5a}, the equivalence $\ftn{ ( U , I_{K_{2}} ) }{ B }{ UB }$ is a positive equivalence through \Mplus[K_{1}\times K_{2}]. Therefore, by Lemma~\ref{lem:BoyleLemma5.5a}, $\ftn{ ( ( W)^{T} , I_{K_{1}} ) }{ ( B')^{T} }{ W^{T} ( B')^{T} }$ is a positive equivalence through \Mplus[K_{2}\times K_{1}] which implies the equivalence $\ftn{ ( I_{K_{1}} , W )}{ B' }{ B' W }$ is a positive equivalence through \Mplus[K_{1}\times K_{2}].  Thus, the equivalence $\ftn{ ( I_{K_{1}} , W^{-1} ) }{ B' W }{ B' }$ is a positive equivalence through \Mplus[K_{1}\times K_{2}].  Since the equivalence $\ftn{ (U , W^{-1} ) }{ B }{ B' }$ is the composition of positive equivalences: $\ftn{ ( U, I_{K_{2}} ) }{ B }{ UB }$ followed by $\ftn{ (I_{K_{1}} , W^{-1} ) }{ B' W }{ B' }$, the equivalence $\ftn{ (U , W^{-1} ) }{ B }{ B' }$ a positive equivalence through \Mplus[K_{1}\times K_{2}].  
\end{proof}

\begin{theorem}[{\cf\ \cite[Theorem~5.1]{MR1907894}}]\label{thm:BoyleTheorem5.1}
Let $K_{1} ,K_{2} \geq 3$ and let $B \in \Mplus[K_{1}\times K_{2}]$.  Suppose $U \in \SLZ[K_{1}]$ and $V \in \SLZ[K_{2}]$ such that $U B V \in \Mplus[K_{1}\times K_{2}]$ and suppose that $X \in \SLZ[K_{1}]$ and $Y \in \SLZ[K_{2}]$ such that 
\begin{align*}
X B Y = 
\begin{pmatrix}
\begin{matrix} 1 & 0  \\
0 & 1 
\end{matrix} & 0 \\ 
0 & F
\end{pmatrix}.
\end{align*}
Then the equivalence $\ftn{ ( U , V ) }{ B }{ UB V }$ is a positive equivalence through \Mplus[K_{1}\times K_{2}].
\end{theorem}

\begin{proof}
Note that for any $H \in \SLZ[2]$, the $K_{1} \times K_{1}$ matrix $G_{H, 1}$ and $K_{2} \times K_{2}$ matrix $G_{H, 2}$, given by 
$$G_{H, 1} = \begin{pmatrix} H & 0 \\ 0 & I_{K_{1}-2}  \end{pmatrix}\quad\text{and}\quad G_{H, 2} = \begin{pmatrix} H & 0 \\ 0 & I_{K_{2}-2}  \end{pmatrix},$$ 
give a self-equivalence $\ftn{ ( X^{-1} G_{H,1} X , Y G_{H,2}^{-1} Y^{-1} ) }{ B }{ B}$.

For a matrix $Q$, we let $Q( 12; * )$ denote the submatrix consisting of the first two rows.  Since $( XBY)( 12; * )$ has rank $2$ and $Y$ is invertible, we have that $(XB) ( 12; * )$ has rank two.  Therefore, there exists $H' \in \SLZ[2]$ such that the first row $r = \begin{pmatrix} r_{1}, \dots, r_{K_{2}} \end{pmatrix}$ of $H' [ ( XB ) ( 12; * ) ]$ has both a positive entry and a negative entry.   

Let $c = \begin{pmatrix} c_{1} \\ \vdots \\ c_{K_{1}} \end{pmatrix}$ denote the first column of $X^{-1}$, and note that it is nonzero.  Since $cr$ is the $K_{1} \times K_{2}$ matrix with $(i,j)$ entry equal to $c_{i} r_{j}$, we have that $cr$ has a positive and a negative entry.  For each $m \in \N$, set $H_{m} = \begin{pmatrix} m & -1 \\ 1 & 0 \end{pmatrix} H'$.  Choose $m$ large enough such that the entries of the two matrices $X^{-1} G_{H_{m}, 1} X B$ and $mc r$ will have the same sign wherever the entries of $mcr$ are nonzero.  In particular, $X^{-1} G_{H_{m} , 1 } X B$ will have a positive entry.  By Lemma~\ref{lem:BoyleLemma5.5b}, $\ftn{ ( X^{-1} G_{ H_{m} , 1} X ,   Y G_{ H_{m} , 2}^{-1} Y^{-1} ) }{ B }{ B }$ gives a positive equivalence through \Mplus[K_{1}\times K_{2}].      

Similarly for large enough $m$, the entries of $U X^{-1} G_{H_{m} , 1} X B$ will agree in sign with the entries $Ucr$ whenever the entries of the latter matrix are nonzero.  Since $U$ is invertible, the matrix $Ucr$ is nonzero, and thus contains positive and negative entries, because $r$ does.  Therefore, $U X^{-1} G_{H_{m} , 1} X B$ contains a positive entry.  By Lemma~\ref{lem:BoyleLemma5.5b}, 
\begin{align*}
\ftn{ ( U X^{-1} G_{ H_{m} , 1} X ,   Y G_{ H_{m} , 2}^{-1} Y^{-1} V) }{ B }{ B' }
\end{align*}
gives a positive equivalence through \Mplus[K_{1}\times K_{2}] with $B' = UBV$.  Hence, the equivalence $\ftn{ ( U,V) }{ B }{ B' }$ is a positive equivalence through \Mplus[K_{1}\times K_{2}] since it is the composition of positive equivalences through \Mplus[K_{1}\times K_{2}]: 
\begin{align*}
\ftn{ ( X^{-1} G_{ H_{m} , 1}^{-1} X  , Y G_{ H_{m} , 2} Y^{-1} ) }{ B }{ B }
\end{align*}
followed by 
\begin{equation*}
\ftn{ ( U X^{-1} G_{ H_{m} , 1} X ,   Y G_{ H_{m} , 2}^{-1} Y^{-1} V) }{ B }{ B' } \qedhere
\end{equation*}
\end{proof}

\subsection{Factorization:  General case}

We now use the results of the previous section to prove a factorization for general $B, B' \in \MPplusZ$ with $n_{i}\neq 0$ that are \SLPE.  Again, many of the arguments follow the arguments of Boyle in \cite{MR1907894}.

\begin{lemma}[{\cf\ \cite[Lemma~4.6]{MR1907894}}]\label{lem:BoyleLemma4.6}
Let $B, B' \in\MPplusZ$ with $n_{i} \neq 0$ for all $i$.  If $\ftn{ (U, V) }{ B }{ B' }$ is an \SLPEe such that $U\{ i \}$ and $V \{ j\}$ are the identity matrices of the appropriate size whenever they are not the empty matrix, then $\ftn{ (U,V) }{ B }{ B' }$ is a positive equivalence.
\end{lemma}

\begin{proof}
We will first find $Q$ in \SLPZ which is a product of nonnegative basic elementary matrices such that $\ftn{ ( U, Q ) }{ B }{ U B Q }$ is a positive equivalence. 
We may assume that $U$ is not the identity matrix. 
Factor $U = U_{n} \cdots U_{1}$ where for each $U_{t}$ there is an associated pair $( i_{t}, j_{t} )$ such that the following hold
\begin{itemize}
\item  $U_{t} = I$ except in the block $U_{t} \{ i_{t} , j_{t} \}$, where it is nonzero
\item if $s \neq t$, then $( i_{s} , j_{s} ) \neq ( i_{t} , j_{t} )$.
\end{itemize} 
Factor $U_{1} = U_{1}^{-} U_{1}^{+}$, where $U_{1}^{-}$ and $U_{1}^{+}$ are equal to $I$ outside the block $\{ i_{1} , j_{1} \}$, $U_{1}^{-} \{ i_{1} , j_{1} \}$ is the nonpositive part of $U_{1} \{ i_{1} , j_{1} \}$ and $U_{1}^{+} \{ i_{1} , j_{1} \}$ is the nonnegative part of $U_{1} \{ i_{1} , j_{1} \}$. Note that $U_{1}^{+}$ is a product of nonnegative basic elementary matrices in \SLPZ[\mathbf{m}] and $U_{1}^{-}$ is a product of nonpositive basic elementary matrices in \SLPZ[\mathbf{m}]. It is now clear that $\ftn{ ( U_{1}^{+} , I ) }{ B }{ U_{1}^{+}B }$ is a positive equivalence.  

Now, note that $U_{1}^{-} U_{1}^{+} B = U_{1}^{+} B$ outside the blocks $\{ i_{1} , k \}$ such that $i_{1} \prec j_{1} \preceq k$. Also note that $m_{i_{1}} \neq 0$ (since $U_1\neq I$). Since $n_{i} \neq 0$ for all $i$, we have that $B \{ i_{1} \}$ is not the empty matrix. Therefore, $B\{i_{1}\} > 0$ since $B\in\MPplusZ$. Hence, $(U_{1}^{+} B) \{ i_{1} \} > 0$ since $\ftn{ ( U_{1}^{+} , I ) }{ B }{ U_{1}^{+}B }$ is a positive equivalence. We can now add columns of  $(U_{1}^{+}B )\{ i_{1} \}$ to columns of $(U_{1}^{+}B) \{ i_{1} , k\}$ for all $i_{1} \prec j_{1} \preceq k$ enough times to obtain a $Q_{1}$ which is a product of nonnegative basic elementary matrices in \SLPZ such that $\ftn{ ( U_{1}^{-} , Q_{1} ) }{ U_{1}^{+} B }{ U_{1}^{-} U_{1}^{+} B Q_{1} }$ is a positive equivalence. Since $\ftn{ (U_{1} , Q_{1} ) }{ B }{ U_{1} B }$ is the composition of positive equivalences
\begin{align*}
\xymatrix{
B \ar[rr]^-{ ( U_{1}^{+} , I ) }_-{+} & & U_{1}^{+}B \ar[rr]^-{ ( U_{1}^{-} , Q_{1} ) }_-{+} & & U_{1} B Q_{1}
}
\end{align*}
we get that the equivalence $\ftn{ (U_{1} , Q_{1} ) }{ B }{ U_{1} B Q_{1} }$ is a positive equivalence.  

Repeat the process for the matrices $U_{1} B Q_{1}$ and $U_{2} U_{1} B Q_{1}$, we get $Q_{2}$ which is the product of nonnegative elementary matrices in \SLPZ such that the equivalence $\ftn{ ( U_{2} , Q_{2} ) }{ U_{1} B Q_{1} }{ U_{2} U_{1} B Q_{1} Q_{2} }$ is a positive equivalence. We continue this process to get $Q_{i}$ that is the product of nonnegative elementary matrices in \SLPZ for $1 \leq i \leq n$ such that $\ftn{ ( U, Q ) }{ B }{ U B Q }$ is a positive equivalence, where $Q = Q_{1}\cdots Q_{n}$. 

We now show that there exists $P$ that is a product of nonnegative basic elementary matrices in \SLPZ[\mathbf{m}] such that $\ftn{ ( P , V^{-1} ) }{ B' }{P B' V^{-1} }$ is a positive equivalence. 
Throughout the rest of the proof, if $M \in \MPplusZ$, then $M \{ \{1, 2, \dots, i \} \}$ will denote the block matrix whose $\{ s, r \}$ block is $M \{ s, r \}$ for all $1 \leq s, r \leq i$.  First note that there are matrices $V_2, \dots, V_N$ in \SLPZ such that $V^{-1} = V_2 V_3 \cdots V_N$, each $V_i$ is the identity matrix except for the blocks $V_i \{ l, i \}$, and 
\[
V_2 \cdots V_i = 
\begin{pmatrix}
V^{-1} \{ \{ 1, \dots, i \} \} & 0 \\
0 & I
\end{pmatrix}.
\]
Let $V_i^-$ be the matrix in \SLPZ that is the identity matrix except for the blocks $V_i \{ l, i \}$ and $V_i^- \{ l, i \}$ is the nonpositive part of $V_i\{ l , i\}$ and  let $V_i^+$ be the matrix in \SLPZ that is the identity matrix except for the blocks $V_i \{ l, i \}$ and $V_i^+ \{ l, i \}$ is the nonnegative part of $V_i\{ l , i\}$.  Note that $V_i^+ V_i^-$ is equal to the identity matrix except for the blocks $V_i \{ l, i \}$ and $(V_i^+ V_i^-)\{ l,i\} = V_i^+\{l, i\} + V_i^-\{l, i \} = V_i \{ l , i \}$.  Therefore, $V_i = V_i^+ V_i^-$

We will inductively construct matrices $P_2 , P_3, \dots, P_N$ in \SLPZ[\mathbf{m}] such that each $P_i$ is the product of nonnegative basic matrices such that each $P_i$ is the identity outside of the blocks $\{l, i \}$ for $l \prec i$ and for each $2 \leq i \leq N$, we have that $\ftn{ ( P_i  ,  V_i ) }{ P_{i-1} \cdots P_2 B' V_2 \cdots V_{i-1} }{  P_i \cdots P_2 B' V_2 \dots V_i }$ is a positive equivalence.  Note that if we have constructed $P_i$, then the composition of these positive equivalences gives a positive equivalence $\ftn{ ( P , V^{-1}) }{ B' }{ P B' V^{-1} }$, where $P = P_k \cdots P_2$.  Thus, the lemma holds.

We now prove the claim.  We first construct $P_2$.  Note that if $1$ is not a predecessor of $2$, then $V_2^+ = V_2^- = I$.  Therefore, $\ftn{ ( I , V_2 ) }{ B' }{ B' V_2 }$ is a positive equivalence.  Suppose $1 \preceq 2$.  Suppose $m_1 = 0$.  Then $B' V_2^+ V_2^- = B' V_2^- = B'$ which implies that $\ftn{ ( I, V_2^- ) }{ B' V_2^+ }{ B' V_2 }$ is a positive equivalence.  So, $\ftn{ ( I, V_2 ) }{ B' }{ B'V_2 }$ is a positive equivalence since it is the composition of the positive equivalences $( I, V_2^+ )$ and $( I, V_2^-)$.  Suppose $m_1 \neq 0$.  In this situation, we have two cases, $m_2 \neq 0$ and $m_2 = 0$.  

Suppose $m_2 \neq 0$.  Note that $B' V_2^+ V_2^-$ is equal to $B'$ except for the $\{1,2\}$ block. We have that $B' V_2^+ > 0$ since $B' > 0$ and $\ftn{ ( I , V_2^+ ) }{ B' }{ B'V_2^+ }$ is a positive equivalence.  Hence, we may add rows of $(B'V_2^+) \{2\}$ to rows of $(B'V_2^+)\{1,2\}$ to get a matrix $P_2$ in \SLPZ[\mathbf{m}] that is the product of nonnegative basic matrices and is the identity outside of the block $\{1,2\}$ such that $\ftn{ (P_2, V_2^- ) }{ B'V_2^+ }{ B' V_2 }$ is a positive equivalence.  Composing the positive equivalences $(I , V_2^+)$ and $( P_2 , V_2^- )$, we get a positive equivalence $\ftn{ ( P_2 , V_2 ) }{ B' }{ P_2 B V_2 }$.  

Suppose $m_2 = 0$.  Then 
\begin{align*}
(B' V_2) \{ 1, 2 \} &= B' \{1\}V_2\{1,2\} + B'\{1,2\} V_2 \{2\} \\
			&= B' \{1\} V^{-1} \{1,2\} + B'\{1,2\}  \\
			&= ( B' V^{-1} ) \{1,2\}, 
\end{align*}
since $V_2 \{ \{1,2\} \} = V^{-1} \{ \{1,2\} \}$ and $V_2 \{ 2 \} = V^{-1}  \{ 2 \} = I$.
Therefore, 
\begin{align*}
(B' V_2^+) V_2^-) \{1,2\} &= (B' V^{-1} ) \{1,2\} = ( UB ) \{ 1, 2 \} \\ 
&= U \{1 \} B \{1,2\} + U \{1, 2 \} B \{2 \} \\ 
&= B \{1,2\} > 0,
\end{align*} 
since $U \{1, 2\}$ is the empty matrix and $U\{1\} = I$.  Therefore $\ftn{ ( I , V_2^- ) }{ B' V_2^+ }{ B' V_2 }$ is a positive equivalence and by composing the positive equivalences $( I, V_2^+ )$ and $(I , V_2^-)$, we get a positive equivalence $\ftn{ (I, V_2 ) }{ B' }{ B' V_2 }$.

So, in all cases, we have found a matrix $P_2$ in \SLPZ[\mathbf{m}] that is the product of nonnegative basic elementary matrices and is the identity outside of the block $\{1,2\}$ such that $\ftn{ ( P_2 , V_2 ) }{ B' }{ P_2 B' V_2 }$ is a positive equivalence.  

Let $2 \leq n \leq N-1$ and suppose we have constructed $P_2 , P_3, \dots, P_n$ in \SLPZ[\mathbf{m}] such that each $P_i$ is the product of nonnegative basic matrices and $P_i$ is the identity outside of the blocks $\{l, i \}$ with $l \prec i$ and for each $2 \leq i \leq n$, we have that $\ftn{ ( P_i  ,V_i ) }{ P_{i-1} \cdots P_2 B' V_2 \cdots V_{i-1} }{  P_i \cdots P_2 B' V_2 \dots V_i }$ is a positive equivalence.  

To simplify the notation, we set $B_i' = P_i \cdots P_2 B' V_2 \dots V_i$.  Since $B_n' > 0$, we get a positive equivalence $\ftn{ ( I, V_{n+1}^+ ) }{ B_n' }{ B_n'V_{n+1}^+ }$.  Note that $B_n'V_{n+1}^+V_{n+1}^-$ is equal to $B_n'$ except for the blocks $\{ i , n+1 \}$ with $i \preceq n+1$.  

Suppose $m_{n+1} \neq 0$.  Then$(B_n' V_{n+1}^+) \{ n+1 \}  > 0$.  Hence, we may add rows of $( B_n' V_{n+1}^+) \{ n+1 \}$ to rows of $(B_n' V_{n+1}^+)\{i, n+1 \}$ for all $i \prec n+1$, to obtain a matrix $P_{n+1}$ in \SLPZ[\mathbf{m}] which is the product of nonnegative basic matrices and is the identity outside of the blocks $\{ i, n+1\}$ for $i \prec n+1$ such that $\ftn{ ( P_{n+1}, V_{n+1}^- )}{B_n'V_{n+1}^+ }{ P_{n+1}B_n'V_{n+1}}$ is a positive equivalence.  Composing the positive equivalences $( I , V_{n+1}^+ )$ and $( P_{n+1} , V_{n+1}^-)$, we get that $\ftn{ ( P_{n+1}, V_{n+1} ) }{ B_n'}{  P_{n+1} B_n' V_{n+1} }$ is a positive equivalence.  

Suppose $m_{n+1} = 0$.  Let $I = \{ i_0, \dots, i_t \}$ be the set of elements $i_s\in\mathcal{P}$ that satisfy $i_s \preceq n+1$, $m_{i_s} \neq 0$, and if $i_s \prec l \preceq n+1$, then $m_l = 0$.
Note that for all distinct $i_s,i_r\in I$, $i_s$ is not a predecessor of $i_r$.
Note that if $I = \emptyset$, then $B_n' V_{n+1}^+ V_{n+1}^- = B_n' V_{n+1}^- = B_n'$.  This would imply that $\ftn{ ( I, V_{n+1}^- )}{ B_n' V_{n+1}^+ }{ B_n' V_{n+1}}$ is a positive equivalence and hence $\ftn{ ( I , V_{n+1} ) }{ B_n' }{ B_n' V_{n+1} }$ is a positive equivalence. 

Suppose $I \neq \emptyset$.  Note that for each $i_s \in I$, 
\begin{align*}
(B_n'V_{n+1} )\{ i_s, n+1 \} &= \sum_{ i_s \preceq l \preceq n+1 }  (P_n \cdots P_2 B' ) \{ i_s , l \}  (V_2 \dots V_nV_{n+1})\{ l, n+1 \}
 \\  				&=\sum_{ i_s \preceq l \preceq n+1 }  (P_n \cdots P_2 B' ) \{ i_s , l \}  V^{-1}\{ l, n+1 \} \\ \\
 				&= ( ( P_n \cdots P_2 B' ) V^{-1} ) \{ i_s , n+1 \}
\end{align*}
 since $(V_2 \dots V_nV_{n+1})\{\{1, \dots n+1\}\} = V^{-1} \{ \{1, \dots, n+1 \} \}$.  Since $P_n \cdots P_2 U B = P_n \cdots P_2 B' V^{-1}$,
\begin{align*}
 ( ( P_n \cdots P_2 B' ) V^{-1} ) \{ i_s , n+1 \} &= ((P_n \cdots P_2 U) B) \{i_s, n+1 \} \\
 				&= \sum_{i_s \preceq l \preceq n+1 } (P_n \cdots P_2 U ) \{ i_s , l \}  B\{ l, n+1 \}. 
\end{align*}
Using the fact that $m_l = 0$ for all $i_s \prec l \preceq n+1$ and $(P_n \cdots P_2 U ) \{ i_s  \} = I$, we get that 
\begin{align*}
(B_n'V_{n+1} )\{ i_s, n+1 \} =  B \{ i_s, n+1 \}.
\end{align*}
Moreover, $(B_n'V_{n+1} )\{ i_s, n+1 \} = B \{ i_s, n+1 \} > 0$ because $m_{i_s} \neq 0$ and $B \in \MPplusZ$.

For each $l\prec n+1$, there exists an $s$ such that $l\preceq i_s$. 
Recall that $(B_n' V_{n+1}^+ )\{ i_s, n+1 \} > 0$, so if $l\prec i_s$, we may add rows of $(B_n'V_{n+1}^+ )\{ i_s, n+1 \}$ to rows of $( B_n'V_{n+1}^+)\{l , n+1 \}$, to get a matrix $P_{n+1}^l$ in \SLPZ[\mathbf{m}] that is the product of nonnegative basic elementary matrices and is the identity outside of the block $\{l,n+1\}$ such that $( P_{n+1}^l B_n' V_{n+1})\{ l, n+1 \} > 0$.  
Doing this for all $l \prec n+1$, we get a matrix $P_{n+1}$ in \SLPZ[\mathbf{m}] that is the product of nonnegative basic elementary matrices and is the identity outside of the blocks $\{l,n+1\}$ for $l \prec n+1$ such that $\ftn{  ( P_{n+1} , V_{n+1}^- )}{ B_n' V_{n+1}^+}{ P_{n+1} B_n' V_{n+1} }$ is a positive equivalence.  Composing the positive equivalences $( I, V_{n+1}^+ )$ and $( P_{n+1} , V_{n+1}^- )$, we get that $\ftn{  ( P_{n+1} , V_{n+1})}{ B_n' }{ P_{n+1} B_n' V_{n+1} }$ is a positive equivalence. 

In all cases, we get a matrix $P_{n+1}$ in \SLPZ[\mathbf{m}] that is the product of nonnegative basic elementary matrices and is the identity outside of the blocks $\{l,n+1\}$ for $l \prec n+1$ such that $\ftn{  ( P_{n+1} , V_{n+1})}{ B_n' }{ P_{n+1} B_n' V_{n+1} }$ is a positive equivalence.  The claim now follows by induction.
\end{proof}

The next lemma allows us to reduce the general case to the case that the diagonal blocks $U\{ i \}$ and $V\{j\}$ are the identity matrices of the appropriate sizes when they are not the empty matrices.  This will allow us to use Lemma~\ref{lem:BoyleLemma4.6} to get the desired positive equivalence.

\begin{lemma}[{\cite[Lemma~4.9]{MR1907894}}]\label{lem:BoyleLemma4.9}
Let $B, B' \in \MPplusZ$ with $n_{l} \neq 0$ for all $l$.  Fix $i$ with $m_{i} \neq 0$.    
\begin{itemize}
\item[(1)]  Suppose $E$ is a basic elementary matrix in \SLPZ[\mathbf{m}] such that $E \{ j, k \} = I \{ j, k \}$ when $(j,k) \neq (i, i)$ and 
\begin{align*}
\ftn{ ( E\{ i \} , I ) }{ B \{ i  \} }{ B' \{  i \}}
\end{align*}
is a positive equivalence.  Then there exists $V \in \SLPZ$ which is the product of nonnegative basic elementary matrices in \SLPZ such that $V \{  k \} = I$ for all $k$ and 
\begin{align*}
\ftn{ ( E, V ) }{ B }{ E B V }
\end{align*}
is a positive equivalence.  

\item[(2)]  Suppose $E$ is a basic elementary matrix in \SLPZ such that $E \{ j, k \} = I \{ j, k \}$ when $(j,k) \neq (i, i)$ and 
\begin{align*}
\ftn{ (  I, E\{ i \} ) }{ B \{  i \} }{ B' \{  i \}}
\end{align*}
is a positive equivalence.
Then there exists $U \in \SLPZ$ which is the product of nonnegative basic elementary matrices in \SLPZ such that $U \{ k \} = I$ for all $k$ and 
\begin{align*}
\ftn{ ( U, E ) }{ B }{ U B' E }
\end{align*}
is a positive equivalence.
\end{itemize}
\end{lemma}

\begin{proof}
We prove (1).  The proof of (2) is similar.  Let $E(s,t)$ be the nonzero offdiagonal entry of $E$.  If $E(s,t) =1$, then set $V = I$.  Suppose $E( s, t ) = -1$.  So, $E$ acts from the left to subtract row $t$ from row $s$.  Since $m_{i} \neq 0$ and 
\begin{align*}
\ftn{ ( E\{ i \} , I ) }{ B \{ i  \} }{ B' \{ i\}},
\end{align*}
is a positive equivalence, we have that $(EB) \{ i \} > 0$.  Thus, there exists $r$ an index for a column through the $\{ i, i \}$ block such that $B'( s, r ) > B'( t, r )$.  Let $V$ be the matrix in \SLPZ which acts from the right to add column $r$ to column $q$, $M$ times, for every $q$ indexing a column through an $\{ i , j \}$ block for which $i \prec j$.  Choosing $M$ large enough, we have that $\ftn{ (E, I ) }{ B' V }{ E B' V }$ is a positive equivalence.  Therefore, $\ftn{ ( E, V ) }{ B }{ E B V  }$ is a positive equivalence since it is the composition of two positive equivalences: $\ftn{ ( I, V ) }{ B }{ B V }$ followed by $\ftn{ ( E , I ) }{ B V }{ E B V }$.
\end{proof}

We are now ready to prove the main result of this section.  This result will be used to show that if the adjacency matrices of $E$ and $F$ are \SLPE, then $E$ is move equivalent to $F$.  Consequently, $C^{*} (E)$ is Morita equivalent to $C^{*} (F)$.

\begin{theorem}[{\cite[Theorem~4.4]{MR1907894}}]\label{thm:BoyleTheorem4.4}
Let $B, B' \in \MPplusZ$ with $n_{i} \neq 0$ for all $i$.  Suppose there exist $U \in \SLPZ[\mathbf{m}]$ and $V \in \SLPZ$ such that $U B V = B'$.  Then $\ftn{ ( U, V ) }{ B }{ B' }$ is a positive equivalence. 
\end{theorem}

\begin{proof}
By Theorem~\ref{thm:BoyleTheorem5.1}, for each $i$ with $m_{i} \neq 0$, we have that $\ftn{ ( U\{ i \} , V \{ i \} ) }{ B \{ i \} }{B' \{ i \} }$ is a positive equivalence since by (ii) of Definition \ref{def: positive matrices}, a $2\times2$-submatrix can be extracted as stipulated.  So, we may find a string of elementary equivalences say $( E_{1}, F_{1} ), \dots, ( E_{t} , F_{t} )$, with every $E_{t} \{ i, j \} = I$, $F_{t} \{ i, j \} = I$ unless $i = j$ with $m_{i} \neq 0$, which accomplishes the elementary positive equivalences decomposition inside the diagonal blocks.  By Lemma~\ref{lem:BoyleLemma4.9}, we may find $( U_{1} , V_{1} ), \dots, ( U_{t} , V_{t} )$ such that $U_{s} \in \SLPZ[\mathbf{m}]$, $V_{s} \in \SLPZ$, $U_{s} \{ k \} = I$, $V_{s} \{ k  \} = I$, and such that we have the following positive equivalences 
\begin{align*}
\xymatrix{
B \ar[r]_-{+}^-{( U_{1} , F_{1} )} & \cdot \ar[r]_-{+}^-{( E_{1} , V_{1} )} & \cdots \ar[r]_-{+}^-{( U_{t} , F_{t} )} & \cdot \ar[r]_-{+}^-{( E_{t} , V_{t} )} & B''.
}
\end{align*}
Let $X = E_{t} U_{t} \cdots E_{2} U_{2} E_{1} U_{1}$ and $Y = F_{1} V_{1} F_{2} V_{2} \cdots F_{t} V_{t}$.  Then for all $i$, we have that $X\{ i \} = U \{ i \}$ and $Y \{ i \} = V \{ i \}$.  Therefore, $( UX^{-1} ) \{ i \} = I$ and $( Y^{-1} V ) \{ i\} = I$ for all $i$.  Then by Lemma~\ref{lem:BoyleLemma4.6},
\begin{align*}
\xymatrix{
B'' \ar[rr]_-{+}^-{( U X^{-1} , Y^{-1} V ) } & & B'
}
\end{align*}
is a positive equivalence.  Thus, $\ftn{ ( U, V ) }{ B }{ B' }$ is a positive equivalence since it is the composition of two positive equivalences
\begin{equation*}
\xymatrix{
B \ar[r]_-{+}^-{ ( X , Y ) } & B'' \ar[rr]^-{( U X^{-1} , Y^{-1} V ) }_-{+} & & B'.
} \qedhere
\end{equation*}
\end{proof}


\section{Putting it all together/Proof of main theorem}
\label{sec:proof}

\begin{theorem}\label{thm:putting-it-all-together}
Let $E_1$ and $E_2$ be graphs with finitely many vertices satisfying Condition~(K) and assume that $\FKRplus(C^*(E_1))\cong \FKRplus(C^*(E_2))$. 
Let $F_1$ and $F_2$ be chosen according to Proposition~\ref{prop:  standard form plus gcd}. 
Then there exists a \GLPEe $(U,V)$ from $\Bsf_{F_1}^\bullet$ to $\Bsf_{F_2}^\bullet$ that satisfies that $V\{i\}$ is the identity matrix whenever $n_i=1$. 
\end{theorem}
\begin{proof}
As usual, we define $\calP$, $\mathbf{m}$ and $\mathbf{n}$ according to the matrices $\Bsf_{F_1}^\bullet$ to $\Bsf_{F_2}^\bullet$ so that it reflects the ideal structure of the associated \cas. 
Here, $\Bsf_{F_1}^\bullet,\Bsf_{F_2}^\bullet\in\MPZ$. 
We let $\calP^\mathsf{T}$ denote the set $\calP$ with the opposite order defined by $i\preceq j$ in $\calP^\mathsf{T}$ if and only if $N+1-j\preceq N+1-i$ in $\calP$, for $i=1,2,\ldots,N$. 
Moreover, we let  $\mathbf{m}^\mathsf{T}=(m_N,\ldots,m_2,m_1)$ and $\mathbf{n}^\mathsf{T}=(n_N,\ldots,n_2,n_1)$, we let $m=m_1+m_2+\cdots+m_N$ and $n=n_1+n_2+\cdots+n_N$, and we let $J_m$ and $J_n$ be the $m\times m$ respectively $n\times n$ permutation matrix that reverses the order.
Then $J_n\left(\Bsf_{F_1}^\bullet\right)^\mathsf{T}J_m, J_n\left(\Bsf_{F_2}^\bullet\right)^\mathsf{T}J_m \in\mathfrak{M}_{\calP^\mathsf{T}}(\mathbf{n}^\mathsf{T}\times\mathbf{m}^\mathsf{T},\Z)$.
So in a similar way as in the proof of \cite[Proposition~8.3]{MR2270572}, where we use \cite[Theorem~4.1 and Remark~4.2]{MR2922394} in the place of \cite[Proposition~3.4]{MR2270572}, we see that this ordered filtered $K$-theory isomorphism induces a $K$-web isomorphism from $K\left(J_n(\Bsf_{F_1}^\bullet)^\mathsf{T}J_m\right)$ to $K\left(J_n(\Bsf_{F_2}^\bullet)^\mathsf{T}J_m\right)$. 
When $n_i=1$, positivity implies that the isomorphism from $\cok\left(J_n(\Bsf_{F_1}^\bullet)^\mathsf{T}J_m\right)$ to $\cok\left(J_n(\Bsf_{F_2}^\bullet)^\mathsf{T}J_m\right)$ is the identity map. 

Now we use Theorem~\ref{thm:mainBH} to get a $\operatorname{GL}_{\calP^\mathsf{T}}$-equivalence $(U,V)$ from $J_n(\Bsf_{F_1}^\bullet)^\mathsf{T}J_m$ to $J_n(\Bsf_{F_2}^\bullet)^\mathsf{T}J_m$ that induces exactly this $K$-web isomorphism. 
Note that $U\{i\}$ is the identity matrix whenever $n_{N+1-i}=1$. 
As in \cite[Remark~8.2]{MR2270572}, we see that $\left(J_mV^\mathsf{T}J_m,J_nU^\mathsf{T}J_n\right)$ is a \GLPEe from $\Bsf_{F_1}^\bullet$ to $\Bsf_{F_2}^\bullet$ that satisfies that $\left(J_nU^\mathsf{T}J_n\right)\{i\}$ is the identity matrix whenever $n_i=1$. 
\end{proof}

\begin{proof}[Proof of Theorem~\ref{thm:main}]
$(\ref{thm:main-item-1}) \implies (\ref{thm:main-item-2})$: It follows from Theorem~\ref{thm:moveimpliesstableisomorphism} that the moves \OO, \II, \RR, \SSS preserve stable isomorphism. 
By Proposition~\ref{prop:cuntzspliceinvariant} \CC also preserves stable isomorphism so $\MCeq$ preserves stable isomorphism. 

$(\ref{thm:main-item-2}) \implies (\ref{thm:main-item-3})$: Holds in general. 

$(\ref{thm:main-item-3}) \implies (\ref{thm:main-item-1})$: Suppose we are given graphs $E_1,E_2$ that satisfy Condition~(K) and have finitely many vertices. 
By Theorem~\ref{thm:putting-it-all-together} we can find graphs $F_1$ and $F_2$ with finally many vertices such that $E_1 \Meq F_1$, $E_2 \Meq F_2$ and $(F_1,F_2)$ are in standard form and there exists a \GLPEe $(U,V)$ from $\Bsf_{F_1}^\bullet$ to $\Bsf_{F_2}^\bullet$ that satisfies that $V\{i\}$ is the identity matrix whenever $n_i=1$. 
Theorem~\ref{thm:GLtoSL} lets us find graphs $G_1,G_2$ in standard form such that $G_1 \MCeq F_1, G_2 \MCeq F_2$ and $\Bsf^{\bullet}_{G_1}$ and $\Bsf^{\bullet}_{G_2}$ are \SLPE. 
By Theorem~\ref{thm:BoyleTheorem4.4} this equivalence is a positive equivalence and so by Corollary~\ref{cor:rc-in-bullet} $G_1 \Meq G_2$. 
Thus we have 
\[
 E_1 \Meq F_1 \MCeq G_1 \Meq G_2 \MCeq F_2 \Meq E_2. \qedhere
\]
\end{proof}



\section*{Acknowledgements}

This work was partially supported by the Danish National Research Foundation through the Centre for Symmetry and Deformation (DNRF92), by VILLUM FONDEN through the network for Experimental Mathematics in Number Theory, Operator Algebras, and Topology, by a grant from the Simons Foundation (\# 279369 to Efren Ruiz), and by the Danish Council for Independent Research | Natural Sciences. 

The third and fourth named authors would also like to thank the School of Mathematics and Applied Statistics at the University of Wollongong for hospitality during their visit where part of this work was carried out. 
The authors would also like to thank Mike Boyle for many fruitful discussions. 



\begin{thebibliography}{BHRS02}

\bibitem[AMP07]{MR2310414}
Pere Ara, M.~Angeles Moreno, and Enrique Pardo, \emph{Nonstable {$K$}-theory
  for graph algebras}, Algebr. Represent. Theory \textbf{10} (2007), no.~2,
  157--178, URL: \url{http://dx.doi.org/10.1007/s10468-006-9044-z}, \href
  {http://dx.doi.org/10.1007/s10468-006-9044-z}
  {\path{doi:10.1007/s10468-006-9044-z}}. \MR{2310414 (2008b:46094)}

\bibitem[AR12]{arXiv:1209.4336v3}
Sara~E. Arklint and Efren Ruiz, \emph{Corners of {C}untz-{K}rieger algebras},
  ArXiv e-prints (2012), to appear in Trans. Amer. Math. Soc, \href
  {http://arxiv.org/abs/1209.4336v3} {\path{arXiv:1209.4336v3}}.

\bibitem[ARR12]{MR2949216}
Sara Arklint, Gunnar Restorff, and Efren Ruiz, \emph{Filtrated {K}-theory for
  real rank zero {$C^*$}-algebras}, Internat. J. Math. \textbf{23} (2012),
  no.~8, 1250078, 19, URL: \url{http://dx.doi.org/10.1142/S0129167X12500784},
  \href {http://dx.doi.org/10.1142/S0129167X12500784}
  {\path{doi:10.1142/S0129167X12500784}}. \MR{2949216}

\bibitem[BCW14]{arXiv:1410.2308v1}
Nathan Brownlowe, Toke~Meier Carlsen, and Michael~F. {Whittaker}, \emph{{Graph
  algebras and orbit equivalence}}, ArXiv e-prints (2014), to appear in Ergodic
  Th.\ Dynam.\ Sys., \href {http://arxiv.org/abs/1410.2308}
  {\path{arXiv:1410.2308}}.

\bibitem[BD96]{MR1396721}
Lawrence~G. Brown and Marius Dadarlat, \emph{Extensions of {$C^\ast$}-algebras
  and quasidiagonality}, J. London Math. Soc. (2) \textbf{53} (1996), no.~3,
  582--600, URL: \url{http://dx.doi.org/10.1112/jlms/53.3.582}, \href
  {http://dx.doi.org/10.1112/jlms/53.3.582} {\path{doi:10.1112/jlms/53.3.582}}.
  \MR{1396721 (97d:46086)}

\bibitem[BH03]{MR1990568}
Mike Boyle and Danrun Huang, \emph{Poset block equivalence of integral
  matrices}, Trans. Amer. Math. Soc. \textbf{355} (2003), no.~10, 3861--3886
  (electronic), URL: \url{http://dx.doi.org/10.1090/S0002-9947-03-02947-7},
  \href {http://dx.doi.org/10.1090/S0002-9947-03-02947-7}
  {\path{doi:10.1090/S0002-9947-03-02947-7}}. \MR{1990568 (2004f:15020)}

\bibitem[BHRS02]{MR1988256}
Teresa Bates, Jeong~Hee Hong, Iain Raeburn, and Wojciech Szyma{\'n}ski,
  \emph{The ideal structure of the {$C^*$}-algebras of infinite graphs},
  Illinois J. Math. \textbf{46} (2002), no.~4, 1159--1176, URL:
  \url{http://projecteuclid.org/euclid.ijm/1258138472}. \MR{1988256
  (2004i:46105)}

\bibitem[BK11]{arXiv:1101.5702v3}
Rasmus Bentmann and Manuel K\"ohler, \emph{Universal coefficient theorems for
  {$C\sp*$}-algebras over finite topological spaces}, ArXiv e-prints (2011),
  \href {http://arxiv.org/abs/1101.5702v3} {\path{arXiv:1101.5702v3}}.

\bibitem[BM14]{arXiv:1405.6512v1}
Rasmus Bentmann and Ralf Meyer, \emph{A more general method to classify up to
  equivariant {KK}-equivalence}, ArXiv e-prints (2014), \href
  {http://arxiv.org/abs/1405.6512v1} {\path{arXiv:1405.6512v1}}.

\bibitem[Boy02]{MR1907894}
Mike Boyle, \emph{Flow equivalence of shifts of finite type via positive
  factorizations}, Pacific J. Math. \textbf{204} (2002), no.~2, 273--317, URL:
  \url{http://dx.doi.org/10.2140/pjm.2002.204.273}, \href
  {http://dx.doi.org/10.2140/pjm.2002.204.273}
  {\path{doi:10.2140/pjm.2002.204.273}}. \MR{1907894 (2003f:37018)}

\bibitem[BP04]{MR2054048}
Teresa Bates and David Pask, \emph{Flow equivalence of graph algebras}, Ergodic
  Theory Dynam. Systems \textbf{24} (2004), no.~2, 367--382, URL:
  \url{http://dx.doi.org/10.1017/S0143385703000348}, \href
  {http://dx.doi.org/10.1017/S0143385703000348}
  {\path{doi:10.1017/S0143385703000348}}. \MR{2054048 (2004m:37019)}

\bibitem[CET12]{MR2922394}
Toke~Meier Carlsen, S{\o}ren Eilers, and Mark Tomforde, \emph{Index maps in the
  {$K$}-theory of graph algebras}, J. K-Theory \textbf{9} (2012), no.~2,
  385--406, URL: \url{http://dx.doi.org/10.1017/is011004017jkt156}, \href
  {http://dx.doi.org/10.1017/is011004017jkt156}
  {\path{doi:10.1017/is011004017jkt156}}. \MR{2922394}

\bibitem[DG97]{MR1465599}
Marius Dadarlat and Guihua Gong, \emph{A classification result for
  approximately homogeneous {$C^*$}-algebras of real rank zero}, Geom. Funct.
  Anal. \textbf{7} (1997), no.~4, 646--711, URL:
  \url{http://dx.doi.org/10.1007/s000390050023}, \href
  {http://dx.doi.org/10.1007/s000390050023} {\path{doi:10.1007/s000390050023}}.
  \MR{1465599 (98j:46062)}

\bibitem[Eil96]{MR1402768}
S{\o}ren Eilers, \emph{A complete invariant for {$AD$} algebras with real rank
  zero and bounded torsion in {$K_1$}}, J. Funct. Anal. \textbf{139} (1996),
  no.~2, 325--348, URL: \url{http://dx.doi.org/10.1006/jfan.1996.0088}, \href
  {http://dx.doi.org/10.1006/jfan.1996.0088}
  {\path{doi:10.1006/jfan.1996.0088}}. \MR{1402768 (97e:46096)}

\bibitem[Ell93]{MR1241132}
George~A. Elliott, \emph{On the classification of {$C^*$}-algebras of real rank
  zero}, J. Reine Angew. Math. \textbf{443} (1993), 179--219, URL:
  \url{http://dx.doi.org/10.1515/crll.1993.443.179}, \href
  {http://dx.doi.org/10.1515/crll.1993.443.179}
  {\path{doi:10.1515/crll.1993.443.179}}. \MR{1241132 (94i:46074)}

\bibitem[ERR10]{MR2666426}
S{\o}ren Eilers, Gunnar Restorff, and Efren Ruiz, \emph{On graph
  {$C^*$}-algebras with a linear ideal lattice}, Bull. Malays. Math. Sci. Soc.
  (2) \textbf{33} (2010), no.~2, 233--241. \MR{2666426 (2012h:46086)}

\bibitem[ERR13a]{MR3142033}
\bysame, \emph{Classification of graph {$C^*$}-algebras with no more than four
  primitive ideals}, Operator algebra and dynamics, Springer Proc. Math. Stat.,
  vol.~58, Springer, Heidelberg, 2013, pp.~89--129, URL:
  \url{http://dx.doi.org/10.1007/978-3-642-39459-1_5}, \href
  {http://dx.doi.org/10.1007/978-3-642-39459-1_5}
  {\path{doi:10.1007/978-3-642-39459-1_5}}. \MR{3142033}

\bibitem[ERR13b]{MR3056712}
\bysame, \emph{Classifying {$C^*$}-algebras with both finite and infinite
  subquotients}, J. Funct. Anal. \textbf{265} (2013), no.~3, 449--468, URL:
  \url{http://dx.doi.org/10.1016/j.jfa.2013.05.006}, \href
  {http://dx.doi.org/10.1016/j.jfa.2013.05.006}
  {\path{doi:10.1016/j.jfa.2013.05.006}}. \MR{3056712}

\bibitem[ERR13c]{arXiv:1301.7695v1}
\bysame, \emph{Strong classification of extensions of classifiable
  {$C\sp*$}-algebras}, ArXiv e-prints (2013), \href
  {http://arxiv.org/abs/1301.7695v1} {\path{arXiv:1301.7695v1}}.

\bibitem[ERS12]{MR3047630}
S{\o}ren Eilers, Efren Ruiz, and Adam P.~W. S{\o}rensen, \emph{Amplified graph
  {$C^*$}-algebras}, M\"unster J. Math. \textbf{5} (2012), 121--150.
  \MR{3047630}

\bibitem[ERS15]{Eilers-Ruiz-Sorensen}
\bysame, \emph{Geometric classification of {$C^*$}-algebras over finite
  graphs}, In preparation, 2015.

\bibitem[ET10]{MR2563693}
S{\o}ren Eilers and Mark Tomforde, \emph{On the classification of nonsimple
  graph {$C^*$}-algebras}, Math. Ann. \textbf{346} (2010), no.~2, 393--418,
  URL: \url{http://dx.doi.org/10.1007/s00208-009-0403-z}, \href
  {http://dx.doi.org/10.1007/s00208-009-0403-z}
  {\path{doi:10.1007/s00208-009-0403-z}}. \MR{2563693 (2010k:46072)}

\bibitem[FLR00]{MR1670363}
Neal~J. Fowler, Marcelo Laca, and Iain Raeburn, \emph{The {$C^*$}-algebras of
  infinite graphs}, Proc. Amer. Math. Soc. \textbf{128} (2000), no.~8,
  2319--2327, URL: \url{http://dx.doi.org/10.1090/S0002-9939-99-05378-2}, \href
  {http://dx.doi.org/10.1090/S0002-9939-99-05378-2}
  {\path{doi:10.1090/S0002-9939-99-05378-2}}. \MR{1670363 (2000k:46079)}

\bibitem[HS03]{MR1989499}
Jeong~Hee Hong and Wojciech Szyma{\'n}ski, \emph{Purely infinite
  {C}untz-{K}rieger algebras of directed graphs}, Bull. London Math. Soc.
  \textbf{35} (2003), no.~5, 689--696, URL:
  \url{http://dx.doi.org/10.1112/S0024609303002364}, \href
  {http://dx.doi.org/10.1112/S0024609303002364}
  {\path{doi:10.1112/S0024609303002364}}. \MR{1989499 (2005c:46097)}

\bibitem[MM14]{MR3276420}
Kengo Matsumoto and Hiroki Matui, \emph{Continuous orbit equivalence of
  topological {M}arkov shifts and {C}untz-{K}rieger algebras}, Kyoto J. Math.
  \textbf{54} (2014), no.~4, 863--877, URL:
  \url{http://dx.doi.org/10.1215/21562261-2801849}, \href
  {http://dx.doi.org/10.1215/21562261-2801849}
  {\path{doi:10.1215/21562261-2801849}}. \MR{3276420}

\bibitem[MT04]{MR2054981}
Paul~S. Muhly and Mark Tomforde, \emph{Adding tails to
  {$C^*$}-correspondences}, Doc. Math. \textbf{9} (2004), 79--106. \MR{2054981
  (2005a:46117)}

\bibitem[New72]{MR0340283}
Morris Newman, \emph{Integral matrices}, Academic Press, New York, 1972, Pure
  and Applied Mathematics, Vol. 45. \MR{0340283 (49 \#5038)}

\bibitem[Rae05]{MR2135030}
Iain Raeburn, \emph{Graph algebras}, CBMS Regional Conference Series in
  Mathematics, vol. 103, Published for the Conference Board of the Mathematical
  Sciences, Washington, DC; by the American Mathematical Society, Providence,
  RI, 2005. \MR{2135030 (2005k:46141)}

\bibitem[Res06]{MR2270572}
Gunnar Restorff, \emph{Classification of {C}untz-{K}rieger algebras up to
  stable isomorphism}, J. Reine Angew. Math. \textbf{598} (2006), 185--210,
  URL: \url{http://dx.doi.org/10.1515/CRELLE.2006.074}, \href
  {http://dx.doi.org/10.1515/CRELLE.2006.074}
  {\path{doi:10.1515/CRELLE.2006.074}}. \MR{2270572 (2007m:46090)}

\bibitem[R{\o}r95]{MR1340839}
Mikael R{\o}rdam, \emph{Classification of {C}untz-{K}rieger algebras},
  $K$-Theory \textbf{9} (1995), no.~1, 31--58, URL:
  \url{http://dx.doi.org/10.1007/BF00965458}, \href
  {http://dx.doi.org/10.1007/BF00965458} {\path{doi:10.1007/BF00965458}}.
  \MR{1340839 (96k:46103)}

\bibitem[R{\o}r97]{MR1446202}
\bysame, \emph{Classification of extensions of certain {$C^*$}-algebras by
  their six term exact sequences in {$K$}-theory}, Math. Ann. \textbf{308}
  (1997), no.~1, 93--117, URL: \url{http://dx.doi.org/10.1007/s002080050067},
  \href {http://dx.doi.org/10.1007/s002080050067}
  {\path{doi:10.1007/s002080050067}}. \MR{1446202 (99b:46108)}

\bibitem[S{\o}r13]{MR3082546}
Adam P.~W. S{\o}rensen, \emph{Geometric classification of simple graph
  algebras}, Ergodic Theory Dynam. Systems \textbf{33} (2013), no.~4,
  1199--1220, URL: \url{http://dx.doi.org/10.1017/S0143385712000260}, \href
  {http://dx.doi.org/10.1017/S0143385712000260}
  {\path{doi:10.1017/S0143385712000260}}. \MR{3082546}

\bibitem[Szy02]{MR1914564}
Wojciech Szyma{\'n}ski, \emph{General {C}untz-{K}rieger uniqueness theorem},
  Internat. J. Math. \textbf{13} (2002), no.~5, 549--555, URL:
  \url{http://dx.doi.org/10.1142/S0129167X0200137X}, \href
  {http://dx.doi.org/10.1142/S0129167X0200137X}
  {\path{doi:10.1142/S0129167X0200137X}}. \MR{1914564 (2003h:46083)}

\end{thebibliography}

\providecommand{\bysame}{\leavevmode\hbox to3em{\hrulefill}\thinspace}
\providecommand{\MR}{\relax\ifhmode\unskip\space\fi MR }
\providecommand{\MRhref}[2]{%
  \href{http://www.ams.org/mathscinet-getitem?mr=#1}{#2}
}
\providecommand{\href}[2]{#2}

\end{document}